\pdfoutput=1
\documentclass[12pt]{amsart}
\usepackage[utf8]{inputenc}
\usepackage[english]{babel}
\usepackage{amsmath,amssymb, mathtools, MnSymbol}
\usepackage{tikz, tikz-cd}
\usetikzlibrary{decorations.pathreplacing, calc,intersections,fit, matrix}
\usepackage{tkz-euclide}
\usetkzobj{all}
\usepackage{wrapfig}
\usepackage{csquotes}
\usepackage{multicol}
\usepackage{cutwin}

\usepackage[backend=biber,style=alphabetic, maxnames=5]{biblatex}
\renewbibmacro{in:}{\ifentrytype{article}{}{\printtext{\bibstring{in}\intitlepunct}}}
\usepackage{enumitem}
\setlist[enumerate]{nosep, label=(\arabic*)}

\usepackage{hyperref}
\hypersetup{
    pdftitle={Categories of Two-Colored Pair Partitions, Part I},
    pdfauthor={Alexander Mang and Moritz Weber},
    pdfsubject={Categories of Two-Colored Pair Partitions},
    pdfkeywords={quantum groups, unitary easy quantum groups, unitary group,  half-liberation, tensor category, two-colored partitions, partitions of sets, categories of partitions}
}
\usepackage{xpatch}
\xpatchcmd{\proof}{\itshape}{\normalfont\proofnamefont}{}{}

\newcommand{\proofnamefont}{\scshape}

\theoremstyle{plain}

\newtheorem{main}{Main Theorem}

\newtheorem{theorem}{Theorem}[section]
\newtheorem{proposition}[theorem]{Proposition}
\newtheorem{lemma}[theorem]{Lemma}

\theoremstyle{definition}
\newtheorem{definition}[theorem]{Definition}

\numberwithin{equation}{section}

\DeclareMathOperator{\inte}{int}

\newcommand{\mc}{\mathcal}
\newcommand{\eqpd}{\coloneqq}
\newcommand{\upp}[1]{\prescript{\filledsquare}{}{#1}}
\newcommand{\lop}[1]{\prescript{}{\filledsquare}{#1}}

\newcommand{\nnint}{\mathbb{N}_0}
\newcommand{\pint}{\mathbb{N}}
\newcommand{\integers}{\mathbb{Z}}
\newcommand{\Cp}{\mathcal{P}^{\circ\bullet}}
\newcommand{\Cpp}{\mathcal{P}^{\circ\bullet}_{2}}
\newcommand{\Cppnb}{\mathcal{P}^{\circ\bullet}_{2,\mathrm{nb}}}
\newcommand{\resbr}{\mathcal{B}_{\mathrm{res}}}
\newcommand{\brt}[3]{\mathrm{Br}\left({#1}\mid {#3}\mid {#2}\right)}
\newcommand{\trmw}{\mathrm{WIn}}
\newcommand{\trms}{\mathrm{SIn}}
\newcommand{\Arg}{\mathrm{Arg}}
\newcommand{\colors}{\{\circ,\bullet\}}

\colorlet{lightgray}{black!40}
\colorlet{gray}{black!60}
\colorlet{darkgray}{black!80}

\newcommand{\PartBracketBBWW}{%
  \begin{tikzpicture}[scale=0.25,baseline=0.04cm]
    \def\xdist{0.666}
    \def\ydist{1.25}
    \def\ypar{0.5}
    \node [scale=0.33, circle, draw=black, fill=black] (a1) at ({0*\xdist},{0*\ydist}) {};
    \node [scale=0.33, circle, draw=black, fill=black] (a2) at ({1*\xdist},{0*\ydist}) {};
    \node [scale=0.33, circle, draw=black, fill=white] (a3) at ({2*\xdist},{0*\ydist}) {};
    \node [scale=0.33, circle, draw=black, fill=white] (a4) at ({3*\xdist},{0*\ydist}) {};
    \node [scale=0.33, circle, draw=black, fill=black] (b1) at ({0*\xdist},{1*\ydist}) {};
    \node [scale=0.33, circle, draw=black, fill=black] (b2) at ({1*\xdist},{1*\ydist}) {};
    \node [scale=0.33, circle, draw=black, fill=white] (b3) at ({2*\xdist},{1*\ydist}) {};
    \node [scale=0.33, circle, draw=black, fill=white] (b4) at ({3*\xdist},{1*\ydist}) {};
    \draw (a1) -- ++ (0,{\ypar}) -| (a4);
    \draw (b1) -- ++ (0,{-\ypar}) -| (b4);
    \draw (a2) to (b2);
    \draw (a3) to (b3);    
  \end{tikzpicture}
}

\newcommand{\PartBracketBBwW}{%
  \begin{tikzpicture}[scale=0.25,baseline=0.04cm]
    \def\xdist{0.666}
    \def\ydist{1.25}
    \def\ypar{0.5}
    \node [scale=0.33, circle, draw=black, fill=black] (a1) at ({0*\xdist},{0*\ydist}) {};
    \node [scale=0.33, circle, draw=black, fill=black] (a2) at ({1*\xdist},{0*\ydist}) {};
    \node [scale=0.33, circle, draw=black, fill=white] (a3) at ({4.2*\xdist},{0*\ydist}) {};
    \node [scale=0.33, circle, draw=black, fill=black] (b1) at ({0*\xdist},{1*\ydist}) {};
    \node [scale=0.33, circle, draw=black, fill=black] (b2) at ({1*\xdist},{1*\ydist}) {};
    \node [scale=0.33, circle, draw=black, fill=white] (b3) at ({4.2*\xdist},{1*\ydist}) {};
    \node [scale=0.55] (l1) at ({2.5*\xdist},{1*\ydist}) {$\otimes w$ };
    \draw (a1) -- ++ (0,{\ypar}) -| (a3);
    \draw (b1) -- ++ (0,{-\ypar}) -| (b3);
    \draw (a2) to (b2);
  \end{tikzpicture}
}

\newcommand{\PartBracketBWwW}{%
  \begin{tikzpicture}[scale=0.25,baseline=0.04cm]
    \def\xdist{0.666}
    \def\ydist{1.25}
    \def\ypar{0.5}
    \node [scale=0.33, circle, draw=black, fill=black] (a1) at ({0*\xdist},{0*\ydist}) {};
    \node [scale=0.33, circle, draw=black, fill=white] (a2) at ({1*\xdist},{0*\ydist}) {};
    \node [scale=0.33, circle, draw=black, fill=white] (a3) at ({4.2*\xdist},{0*\ydist}) {};
    \node [scale=0.33, circle, draw=black, fill=black] (b1) at ({0*\xdist},{1*\ydist}) {};
    \node [scale=0.33, circle, draw=black, fill=white] (b2) at ({1*\xdist},{1*\ydist}) {};
    \node [scale=0.33, circle, draw=black, fill=white] (b3) at ({4.2*\xdist},{1*\ydist}) {};
    \node [scale=0.55] (l1) at ({2.5*\xdist},{1*\ydist}) {$\otimes w$ };
    \draw (a1) -- ++ (0,{\ypar}) -| (a3);
    \draw (b1) -- ++ (0,{-\ypar}) -| (b3);
    \draw (a2) to (b2);
  \end{tikzpicture}
}

\newcommand{\PartBracketWBwB}{%
  \begin{tikzpicture}[scale=0.25,baseline=0.04cm]
    \def\xdist{0.666}
    \def\ydist{1.25}
    \def\ypar{0.5}
    \node [scale=0.33, circle, draw=black, fill=white] (a1) at ({0*\xdist},{0*\ydist}) {};
    \node [scale=0.33, circle, draw=black, fill=black] (a2) at ({1*\xdist},{0*\ydist}) {};
    \node [scale=0.33, circle, draw=black, fill=black] (a3) at ({4.2*\xdist},{0*\ydist}) {};
    \node [scale=0.33, circle, draw=black, fill=white] (b1) at ({0*\xdist},{1*\ydist}) {};
    \node [scale=0.33, circle, draw=black, fill=black] (b2) at ({1*\xdist},{1*\ydist}) {};
    \node [scale=0.33, circle, draw=black, fill=black] (b3) at ({4.2*\xdist},{1*\ydist}) {};
    \node [scale=0.55] (l1) at ({2.5*\xdist},{1*\ydist}) {$\otimes w$ };
    \draw (a1) -- ++ (0,{\ypar}) -| (a3);
    \draw (b1) -- ++ (0,{-\ypar}) -| (b3);
    \draw (a2) to (b2);
  \end{tikzpicture}
}

\newcommand{\PartBracketWWwB}{%
  \begin{tikzpicture}[scale=0.25,baseline=0.04cm]
    \def\xdist{0.666}
    \def\ydist{1.25}
    \def\ypar{0.5}
    \node [scale=0.33, circle, draw=black, fill=white] (a1) at ({0*\xdist},{0*\ydist}) {};
    \node [scale=0.33, circle, draw=black, fill=white] (a2) at ({1*\xdist},{0*\ydist}) {};
    \node [scale=0.33, circle, draw=black, fill=black] (a3) at ({4.2*\xdist},{0*\ydist}) {};
    \node [scale=0.33, circle, draw=black, fill=white] (b1) at ({0*\xdist},{1*\ydist}) {};
    \node [scale=0.33, circle, draw=black, fill=white] (b2) at ({1*\xdist},{1*\ydist}) {};
    \node [scale=0.33, circle, draw=black, fill=black] (b3) at ({4.2*\xdist},{1*\ydist}) {};
    \node [scale=0.55] (l1) at ({2.5*\xdist},{1*\ydist}) {$\otimes w$ };
    \draw (a1) -- ++ (0,{\ypar}) -| (a3);
    \draw (b1) -- ++ (0,{-\ypar}) -| (b3);
    \draw (a2) to (b2);
  \end{tikzpicture}
}

\newcommand{\PartBracketBWBW}{%
  \begin{tikzpicture}[scale=0.25,baseline=0.04cm]
    \def\xdist{0.666}
    \def\ydist{1.25}
    \def\ypar{0.5}
    \node [scale=0.33, circle, draw=black, fill=black] (a1) at ({0*\xdist},{0*\ydist}) {};
    \node [scale=0.33, circle, draw=black, fill=white] (a2) at ({1*\xdist},{0*\ydist}) {};
    \node [scale=0.33, circle, draw=black, fill=black] (a3) at ({2*\xdist},{0*\ydist}) {};
    \node [scale=0.33, circle, draw=black, fill=white] (a4) at ({3*\xdist},{0*\ydist}) {};
    \node [scale=0.33, circle, draw=black, fill=black] (b1) at ({0*\xdist},{1*\ydist}) {};
    \node [scale=0.33, circle, draw=black, fill=white] (b2) at ({1*\xdist},{1*\ydist}) {};
    \node [scale=0.33, circle, draw=black, fill=black] (b3) at ({2*\xdist},{1*\ydist}) {};
    \node [scale=0.33, circle, draw=black, fill=white] (b4) at ({3*\xdist},{1*\ydist}) {};
    \draw (a1) -- ++ (0,{\ypar}) -| (a4);
    \draw (b1) -- ++ (0,{-\ypar}) -| (b4);
    \draw (a2) to (b2);
    \draw (a3) to (b3);    
  \end{tikzpicture}
}

\newcommand{\PartBracketWBWB}{%
  \begin{tikzpicture}[scale=0.25,baseline=0.04cm]
    \def\xdist{0.666}
    \def\ydist{1.25}
    \def\ypar{0.5}
    \node [scale=0.33, circle, draw=black, fill=white] (a1) at ({0*\xdist},{0*\ydist}) {};
    \node [scale=0.33, circle, draw=black, fill=black] (a2) at ({1*\xdist},{0*\ydist}) {};
    \node [scale=0.33, circle, draw=black, fill=white] (a3) at ({2*\xdist},{0*\ydist}) {};
    \node [scale=0.33, circle, draw=black, fill=black] (a4) at ({3*\xdist},{0*\ydist}) {};
    \node [scale=0.33, circle, draw=black, fill=white] (b1) at ({0*\xdist},{1*\ydist}) {};
    \node [scale=0.33, circle, draw=black, fill=black] (b2) at ({1*\xdist},{1*\ydist}) {};
    \node [scale=0.33, circle, draw=black, fill=white] (b3) at ({2*\xdist},{1*\ydist}) {};
    \node [scale=0.33, circle, draw=black, fill=black] (b4) at ({3*\xdist},{1*\ydist}) {};
    \draw (a1) -- ++ (0,{\ypar}) -| (a4);
    \draw (b1) -- ++ (0,{-\ypar}) -| (b4);
    \draw (a2) to (b2);
    \draw (a3) to (b3);    
  \end{tikzpicture}
}

\newcommand{\PartBracketBBvWvW}{%
  \begin{tikzpicture}[scale=0.25,baseline=0.04cm]
    \def\xdist{0.666}
    \def\ydist{1.25}
    \def\ypar{0.5}
    \node [scale=0.33, circle, draw=black, fill=black] (a1) at ({0*\xdist},{0*\ydist}) {};
    \node [scale=0.33, circle, draw=black, fill=black] (a2) at ({1*\xdist},{0*\ydist}) {};
    \node [scale=0.33, circle, draw=black, fill=white] (a3) at ({3.75*\xdist},{0*\ydist}) {};
    \node [scale=0.33, circle, draw=black, fill=white] (a4) at ({6.5*\xdist},{0*\ydist}) {};
    \node [scale=0.33, circle, draw=black, fill=black] (b1) at ({0*\xdist},{1*\ydist}) {};
    \node [scale=0.33, circle, draw=black, fill=black] (b2) at ({1*\xdist},{1*\ydist}) {};
    \node [scale=0.33, circle, draw=black, fill=white] (b3) at ({3.75*\xdist},{1*\ydist}) {};
    \node [scale=0.33, circle, draw=black, fill=white] (b4) at ({6.5*\xdist},{1*\ydist}) {};
    \node [scale=0.55] (l1) at ({2.25*\xdist},{1*\ydist}) {$\otimes v$ };
    \node [scale=0.55] (l2) at ({5*\xdist},{1*\ydist}) {$\otimes v$ };        
    \draw (a1) -- ++ (0,{\ypar}) -| (a4);
    \draw (b1) -- ++ (0,{-\ypar}) -| (b4);
    \draw (a2) to (b2);
    \draw (a3) to (b3);    
  \end{tikzpicture}
}

\newcommand{\PartBracketWWvBvB}{%
  \begin{tikzpicture}[scale=0.25,baseline=0.04cm]
    \def\xdist{0.666}
    \def\ydist{1.25}
    \def\ypar{0.5}
    \node [scale=0.33, circle, draw=black, fill=white] (a1) at ({0*\xdist},{0*\ydist}) {};
    \node [scale=0.33, circle, draw=black, fill=white] (a2) at ({1*\xdist},{0*\ydist}) {};
    \node [scale=0.33, circle, draw=black, fill=black] (a3) at ({3.75*\xdist},{0*\ydist}) {};
    \node [scale=0.33, circle, draw=black, fill=black] (a4) at ({6.5*\xdist},{0*\ydist}) {};
    \node [scale=0.33, circle, draw=black, fill=white] (b1) at ({0*\xdist},{1*\ydist}) {};
    \node [scale=0.33, circle, draw=black, fill=white] (b2) at ({1*\xdist},{1*\ydist}) {};
    \node [scale=0.33, circle, draw=black, fill=black] (b3) at ({3.75*\xdist},{1*\ydist}) {};
    \node [scale=0.33, circle, draw=black, fill=black] (b4) at ({6.5*\xdist},{1*\ydist}) {};
    \node [scale=0.55] (l1) at ({2.25*\xdist},{1*\ydist}) {$\otimes v$ };
    \node [scale=0.55] (l2) at ({5*\xdist},{1*\ydist}) {$\otimes v$ };        
    \draw (a1) -- ++ (0,{\ypar}) -| (a4);
    \draw (b1) -- ++ (0,{-\ypar}) -| (b4);
    \draw (a2) to (b2);
    \draw (a3) to (b3);    
  \end{tikzpicture}
}

\newcommand{\PartBracketBBvdWvdW}{%
  \begin{tikzpicture}[scale=0.25,baseline=0.04cm]
    \def\xdist{0.666}
    \def\ydist{1.25}
    \def\ypar{0.5}
    \node [scale=0.33, circle, draw=black, fill=black] (a1) at ({0*\xdist},{0*\ydist}) {};
    \node [scale=0.33, circle, draw=black, fill=black] (a2) at ({1*\xdist},{0*\ydist}) {};
    \node [scale=0.33, circle, draw=black, fill=white] (a3) at ({4*\xdist},{0*\ydist}) {};
    \node [scale=0.33, circle, draw=black, fill=white] (a4) at ({7*\xdist},{0*\ydist}) {};
    \node [scale=0.33, circle, draw=black, fill=black] (b1) at ({0*\xdist},{1*\ydist}) {};
    \node [scale=0.33, circle, draw=black, fill=black] (b2) at ({1*\xdist},{1*\ydist}) {};
    \node [scale=0.33, circle, draw=black, fill=white] (b3) at ({4*\xdist},{1*\ydist}) {};
    \node [scale=0.33, circle, draw=black, fill=white] (b4) at ({7*\xdist},{1*\ydist}) {};
    \node [scale=0.55] (l1) at ({2.5*\xdist},{1*\ydist}) {$\otimes v'$ };
    \node [scale=0.55] (l2) at ({5.5*\xdist},{1*\ydist}) {$\otimes v'$ };        
    \draw (a1) -- ++ (0,{\ypar}) -| (a4);
    \draw (b1) -- ++ (0,{-\ypar}) -| (b4);
    \draw (a2) to (b2);
    \draw (a3) to (b3);    
  \end{tikzpicture}
}

\newcommand{\PartBracketWWvdBvdB}{%
  \begin{tikzpicture}[scale=0.25,baseline=0.04cm]
    \def\xdist{0.666}
    \def\ydist{1.25}
    \def\ypar{0.5}
    \node [scale=0.33, circle, draw=black, fill=white] (a1) at ({0*\xdist},{0*\ydist}) {};
    \node [scale=0.33, circle, draw=black, fill=white] (a2) at ({1*\xdist},{0*\ydist}) {};
    \node [scale=0.33, circle, draw=black, fill=black] (a3) at ({4*\xdist},{0*\ydist}) {};
    \node [scale=0.33, circle, draw=black, fill=black] (a4) at ({7*\xdist},{0*\ydist}) {};
    \node [scale=0.33, circle, draw=black, fill=white] (b1) at ({0*\xdist},{1*\ydist}) {};
    \node [scale=0.33, circle, draw=black, fill=white] (b2) at ({1*\xdist},{1*\ydist}) {};
    \node [scale=0.33, circle, draw=black, fill=black] (b3) at ({4*\xdist},{1*\ydist}) {};
    \node [scale=0.33, circle, draw=black, fill=black] (b4) at ({7*\xdist},{1*\ydist}) {};
    \node [scale=0.55] (l1) at ({2.5*\xdist},{1*\ydist}) {$\otimes v'$ };
    \node [scale=0.55] (l2) at ({5.5*\xdist},{1*\ydist}) {$\otimes v'$ };        
    \draw (a1) -- ++ (0,{\ypar}) -| (a4);
    \draw (b1) -- ++ (0,{-\ypar}) -| (b4);
    \draw (a2) to (b2);
    \draw (a3) to (b3);    
  \end{tikzpicture}
}

\newcommand{\PartHalfLibWBW}{%
  \begin{tikzpicture}[scale=0.25,baseline=0.04cm]
    \def\xdist{0.75}
    \def\ydist{1.25}
    \def\ypar{0.5}
    \node [scale=0.33, circle, draw=black, fill=white] (a1) at ({0*\xdist},{0*\ydist}) {};
    \node [scale=0.33, circle, draw=black, fill=black] (a2) at ({1*\xdist},{0*\ydist}) {};
    \node [scale=0.33, circle, draw=black, fill=white] (a3) at ({2*\xdist},{0*\ydist}) {};
    \node [scale=0.33, circle, draw=black, fill=white] (b1) at ({0*\xdist},{1*\ydist}) {};
    \node [scale=0.33, circle, draw=black, fill=black] (b2) at ({1*\xdist},{1*\ydist}) {};
    \node [scale=0.33, circle, draw=black, fill=white] (b3) at ({2*\xdist},{1*\ydist}) {};
    \draw (a1) to (b3);    
    \draw (a2) to (b2);
    \draw (a3) to (b1);    
  \end{tikzpicture}
}

\newcommand{\PartHalfLibBWB}{%
  \begin{tikzpicture}[scale=0.25,baseline=0.04cm]
    \def\xdist{0.75}
    \def\ydist{1.25}
    \def\ypar{0.5}
    \node [scale=0.33, circle, draw=black, fill=black] (a1) at ({0*\xdist},{0*\ydist}) {};
    \node [scale=0.33, circle, draw=black, fill=white] (a2) at ({1*\xdist},{0*\ydist}) {};
    \node [scale=0.33, circle, draw=black, fill=black] (a3) at ({2*\xdist},{0*\ydist}) {};
    \node [scale=0.33, circle, draw=black, fill=black] (b1) at ({0*\xdist},{1*\ydist}) {};
    \node [scale=0.33, circle, draw=black, fill=white] (b2) at ({1*\xdist},{1*\ydist}) {};
    \node [scale=0.33, circle, draw=black, fill=black] (b3) at ({2*\xdist},{1*\ydist}) {};
    \draw (a1) to (b3);    
    \draw (a2) to (b2);
    \draw (a3) to (b1);    
  \end{tikzpicture}
  }

  \newcommand{\PartHalfLibWWW}{%
  \begin{tikzpicture}[scale=0.25,baseline=0.04cm]
    \def\xdist{0.75}
    \def\ydist{1.25}
    \def\ypar{0.5}
    \node [scale=0.33, circle, draw=black, fill=white] (a1) at ({0*\xdist},{0*\ydist}) {};
    \node [scale=0.33, circle, draw=black, fill=white] (a2) at ({1*\xdist},{0*\ydist}) {};
    \node [scale=0.33, circle, draw=black, fill=white] (a3) at ({2*\xdist},{0*\ydist}) {};
    \node [scale=0.33, circle, draw=black, fill=white] (b1) at ({0*\xdist},{1*\ydist}) {};
    \node [scale=0.33, circle, draw=black, fill=white] (b2) at ({1*\xdist},{1*\ydist}) {};
    \node [scale=0.33, circle, draw=black, fill=white] (b3) at ({2*\xdist},{1*\ydist}) {};
    \draw (a1) to (b3);    
    \draw (a2) to (b2);
    \draw (a3) to (b1);    
  \end{tikzpicture}
}

\newcommand{\PartCrossWW}{%
  \begin{tikzpicture}[scale=0.25,baseline=0.015cm]
    \def\xdist{1}
    \def\ydist{1}
    \node [scale=0.33, circle, draw=black, fill=white] (a1) at ({0*\xdist},{0*\ydist}) {};
    \node [scale=0.33, circle, draw=black, fill=white] (a2) at ({1*\xdist},{0*\ydist}) {};
    \node [scale=0.33, circle, draw=black, fill=white] (b1) at ({0*\xdist},{1*\ydist}) {};
    \node [scale=0.33, circle, draw=black, fill=white] (b2) at ({1*\xdist},{1*\ydist}) {};
    \draw (a1) to (b2);    
    \draw (a2) to (b1);
  \end{tikzpicture}
  }

  \newcommand{\PartIdenB}{%
  \begin{tikzpicture}[scale=0.25,baseline=0.015cm]
    \def\xdist{1}
    \def\ydist{1}
    \node [scale=0.33, circle, draw=black, fill=black] (a1) at ({0*\xdist},{0*\ydist}) {};
    \node [scale=0.33, circle, draw=black, fill=black] (b1) at ({0*\xdist},{1*\ydist}) {};
    \draw (a1) to (b1);    
  \end{tikzpicture}
  }

  \newcommand{\PartIdenW}{%
  \begin{tikzpicture}[scale=0.25,baseline=0.015cm]
    \def\xdist{1}
    \def\ydist{1}
    \node [scale=0.33, circle, draw=black, fill=white] (a1) at ({0*\xdist},{0*\ydist}) {};
    \node [scale=0.33, circle, draw=black, fill=white] (b1) at ({0*\xdist},{1*\ydist}) {};
    \draw (a1) to (b1);    
  \end{tikzpicture}
  }
  
\newcommand{\PartIdenLoBW}{%
  \begin{tikzpicture}[scale=0.25,baseline=-0.015cm]
    \def\xdist{1}
    \def\ydist{1}
    \def\ypar{1}
    \node [scale=0.33, circle, draw=black, fill=black] (a1) at ({0*\xdist},{0*\ydist}) {};
    \node [scale=0.33, circle, draw=black, fill=white] (a2) at ({1*\xdist},{0*\ydist}) {};
    \draw (a1) -- ++ (0,{\ypar}) -| (a2);    
  \end{tikzpicture}
  }
\newcommand{\PartIdenLoWB}{%
  \begin{tikzpicture}[scale=0.25,baseline=-0.015cm]
    \def\xdist{1}
    \def\ydist{1}
    \def\ypar{1}
    \node [scale=0.33, circle, draw=black, fill=white] (a1) at ({0*\xdist},{0*\ydist}) {};
    \node [scale=0.33, circle, draw=black, fill=black] (a2) at ({1*\xdist},{0*\ydist}) {};
    \draw (a1) -- ++ (0,{\ypar}) -| (a2);    
  \end{tikzpicture}
  }

\begin{document}
\title[Categories of Two-Colored Pair Partitions, Part~I]{Categories of Two-Colored Pair Partitions\\Part~I: Categories Indexed by Cyclic Groups} 
\author{Alexander Mang}
\author{Moritz Weber}
\address{Saarland University, Fachbereich Mathematik, 
	66041 Saarbrücken, Germany}
\email{s9almang@stud.uni-saarland.de, weber@math.uni-sb.de}
\thanks{The second author was supported by the ERC Advanced Grant NCDFP, held
by Roland Spei\-cher, by the SFB-TRR 195, and by the DFG project \emph{Quantenautomorphismen von Graphen}. This work was part of the first author's Bachelor's thesis.}

\date{\today}

\begin{abstract}
  We classify certain categories of partitions of finite sets subject to specific rules on the colorization of points and the sizes of blocks. More precisely, we consider pair partitions such that each block contains exactly one white and one black point when rotated to one line; however crossings are allowed. There are two families of such categories, the first of which is indexed by cyclic groups and is covered in the present article; the second family will be the content of a follow-up article. Via a Tannaka-Krein result, the categories in the two families correspond to easy quantum groups interpolating the classical unitary group and Wang's free unitary quantum group. In fact, they are all half-liberated in some sense and our results imply that there are many more half-liberation procedures than previously expected. However, we focus on a purely combinatorial approach leaving quantum group aspects aside. 
\end{abstract}
\maketitle

	\section*{Introduction}
	This article is part of the classification program begun in \cite{TaWe15a}, having its roots in \cite{BaSp09}. Our base objects are partitions of finite sets into disjoint subsets, the blocks. In addition, the points are colored either black or white. We represent partitions pictorially as diagrams using strings representing the blocks; see also \cite{Sta12} and \cite{NiSp06}.
	If a set of partitions is closed under certain natural operations like horizontal or vertical concatenation or reflection at some axes, we call it a category of partitions. Categories of partitions play a crucial role in Banica and Spei\-cher's approach (\cite{BaSp09}, \cite{We16} and \cite{We17a}) to compact quantum groups in Woronowicz's sense (\cite{Wo87}, \cite{Wo88}, \cite{Wo91} and \cite{Wo98}). Although our investigations employ purely combinatorial means, let us briefly mention how they relate to the half-liberation procedures of Banica and Speicher. 
	\par
	The half-liberated orthogonal quantum group $O_n^*$, introduced by Banica and Spei\-cher in \cite{BaSp09}, represents a midway point between the free orthogonal quantum group $O_n^+$, constructed by Wang in \cite{Wa95}, and the classical orthogonal group $O_n$ over the complex numbers. It is defined by replacing the commutation relations $ab=ba$ by the half-commutation relations $abc=cba$. On the combinatorial side these half-commutation relations are represented by the partition
        
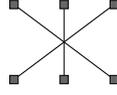
\begin{gather*}
          \begin{tikzpicture}[scale=0.666,baseline=(current  bounding  box.center)]
		\node [scale=0.4, draw=black, fill=gray] (x1-1) at (0,0) {};
		\node [scale=0.4, draw=black, fill=gray] (x2-1) at (1,0) {};
		\node [scale=0.4, draw=black, fill=gray] (x3-1) at (2,0) {};
		\node [scale=0.4, draw=black, fill=gray] (x1-2) at (0,1.5) {};
		\node [scale=0.4, draw=black, fill=gray] (x2-2) at (1,1.5) {};
		\node [scale=0.4, draw=black, fill=gray] (x3-2) at (2,1.5) {};
		\draw (x1-1) -- (x3-2);
		\draw (x2-1) -- (x2-2);
		\draw (x3-1) -- (x1-2);
		\end{tikzpicture}
              \end{gather*}
        \vspace{-0.095em}               
        on non-colored points.
        \par
        If attempting a similar procedure in the case of the free unitary quantum group $U_n^+$, also defined by Wang in \cite{Wa95}, and the classical unitary group $U_n$, it is not clear what should be the analogue of the half-commutation relations. Here, the generators are no longer self-adjoint and hence their adjoints are involved. Equivalently,  we now have to deal with partitions of two-colored points.\par
        A natural starting point would be to consider the different possibilities for coloring the points of the above partition. And, indeed, this approach yields exactly two distinct categories. Each of the partitions
        
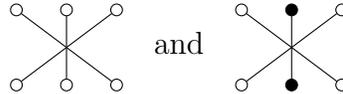
\begin{gather*}
		\begin{tikzpicture}[scale=0.666,baseline=0.416cm]
		\node [circle, scale=0.4, draw=black, fill=white] (x1-1) at (0,0) {};
		\node [circle, scale=0.4, draw=black, fill=white] (x2-1) at (1,0) {};
		\node [circle, scale=0.4, draw=black, fill=white] (x3-1) at (2,0) {};
		\node [circle, scale=0.4, draw=black, fill=white] (x1-2) at (0,1.5) {};
		\node [circle, scale=0.4, draw=black, fill=white] (x2-2) at (1,1.5) {};
		\node [circle, scale=0.4, draw=black, fill=white] (x3-2) at (2,1.5) {};
		\draw (x1-1) -- (x3-2);
		\draw (x2-1) -- (x2-2);
		\draw (x3-1) -- (x1-2);
		\end{tikzpicture}
                \quad\text{and}\quad
                		\begin{tikzpicture}[scale=0.666,baseline=0.416cm]
		\node [circle, scale=0.4, draw=black, fill=white] (x1-1) at (0,0) {};
		\node [circle, scale=0.4, draw=black, fill=black] (x2-1) at (1,0) {};
		\node [circle, scale=0.4, draw=black, fill=white] (x3-1) at (2,0) {};
		\node [circle, scale=0.4, draw=black, fill=white] (x1-2) at (0,1.5) {};
		\node [circle, scale=0.4, draw=black, fill=black] (x2-2) at (1,1.5) {};
		\node [circle, scale=0.4, draw=black, fill=white] (x3-2) at (2,1.5) {};
		\draw (x1-1) -- (x3-2);
		\draw (x2-1) -- (x2-2);
		\draw (x3-1) -- (x1-2);
              \end{tikzpicture}
            \end{gather*}
            generates one.
        These categories (or rather their associated quantum groups) appeared first in \cite[Example 4.10]{BiDu13} and \cite[Definition~5.5]{BhDADa11}. Elaborating on the initial idea, one might next study the shapes
        
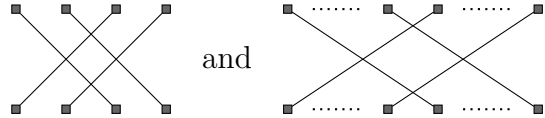
\begin{gather*}
	\begin{tikzpicture}[scale=0.666,baseline=0.555cm]
          \node [scale=0.4, draw=black, fill=gray] (x1-1) at (0,0) {};
		\node [scale=0.4, draw=black, fill=gray] (x2-1) at (1,0) {};
		\node [scale=0.4, draw=black, fill=gray] (x3-1) at (2,0) {};
		\node [scale=0.4, draw=black, fill=gray] (x4-1) at (3,0) {};
		\node [scale=0.4, draw=black, fill=gray] (x1-2) at (0,2) {};
		\node [scale=0.4, draw=black, fill=gray] (x2-2) at (1,2) {};
		\node [scale=0.4, draw=black, fill=gray] (x3-2) at (2,2) {};
		\node [scale=0.4, draw=black, fill=gray] (x4-2) at (3,2) {};
		\draw (x1-1) -- (x3-2);
		\draw (x2-1) -- (x4-2);
		\draw (x3-1) -- (x1-2);
		\draw (x4-1) -- (x2-2);		
              \end{tikzpicture}
              \quad\text{and}\quad
	\begin{tikzpicture}[scale=0.666,baseline=0.555cm]
                \node [scale=0.4, draw=black, fill=gray] (x1-1) at (0,0) {};
		\node [scale=0.4, draw=black, fill=gray] (x2-1) at (2,0) {};
		\node [scale=0.4, draw=black, fill=gray] (x3-1) at (3,0) {};
		\node [scale=0.4, draw=black, fill=gray] (x4-1) at (5,0) {};
		\node [scale=0.4, draw=black, fill=gray] (x1-2) at (0,2) {};
		\node [scale=0.4, draw=black, fill=gray] (x2-2) at (2,2) {};
		\node [scale=0.4, draw=black, fill=gray] (x3-2) at (3,2) {};
		\node [scale=0.4, draw=black, fill=gray] (x4-2) at (5,2) {};
		\draw (x1-1) -- (x3-2);
		\draw (x2-1) -- (x4-2);
		\draw (x3-1) -- (x1-2);
		\draw (x4-1) -- (x2-2);
                \draw[thick, dotted] (0.5,0) -- (1.5,0);
                \draw[thick, dotted] (3.5,0) -- (4.5,0);
                \draw[thick, dotted] (0.5,2) -- (1.5,2);
                \draw[thick, dotted] (3.5,2) -- (4.5,2);
              \end{tikzpicture}              
            \end{gather*}
            and apply various colorings to them. Thus, one arrives at a number of further categories, discovered by Banica and Bichon in \cite{BaBi17b}.
            See Section~\ref{section:concluding_remarks} 
            for an overview of the previous work on half-liberations of $U_n$. \par
            Here, however, we take a different approach. Even though both of the partitions
                
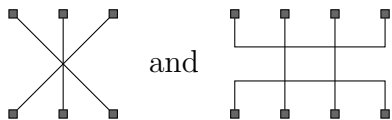
\begin{gather*}
		\begin{tikzpicture}[scale=0.666, baseline=0.555cm]
		\node[scale=0.4, draw=black, fill=gray] (x1-1) at (0,0) {};
		\node[scale=0.4, draw=black, fill=gray] (x2-1) at (1,0) {};
		\node[scale=0.4, draw=black, fill=gray] (x3-1) at (2,0) {};
		\node[scale=0.4, draw=black, fill=gray] (x1-2) at (0,2) {};			
		\node[scale=0.4, draw=black, fill=gray] (x2-2) at (1,2) {};
		\node[scale=0.4, draw=black, fill=gray] (x3-2) at (2,2) {};
		\draw (x1-1) -- (x3-2);				
		\draw (x2-1) -- (x2-2);				
		\draw (x3-1) -- (x1-2);								
              \end{tikzpicture}
		\quad\text{and}\quad              
		\begin{tikzpicture}[scale=0.666, baseline=0.555cm]
		\node[scale=0.4, draw=black, fill=gray] (x1-1) at (0,0) {};
		\node[scale=0.4, draw=black, fill=gray] (y1-1) at (1,0) {};
		\node[scale=0.4, draw=black, fill=gray] (y2-1) at (2,0) {};
		\node[scale=0.4, draw=black, fill=gray] (x2-1) at (3,0) {};			
		\node[scale=0.4, draw=black, fill=gray] (x1-2) at (0,2) {};
		\node[scale=0.4, draw=black, fill=gray] (y1-2) at (1,2) {};
		\node[scale=0.4, draw=black, fill=gray] (y2-2) at (2,2) {};
		\node[scale=0.4, draw=black, fill=gray] (x2-2) at (3,2) {};	
		\draw (x1-1) -- ++(0,0.66)-| (x2-1);
		\draw (x1-2) -- ++(0,-0.66)-| (x2-2);
		\draw (y1-1) -- (y1-2);				
		\draw (y2-1) -- (y2-2);								
		\end{tikzpicture}
              \end{gather*}
              embody the half-commutation relations $abc=cba$ in the orthogonal case, the simple, but crucial observation which unlocks the combinatorics of half-liberation in the unitary case is choosing the latter partition, which we call a \emph{bracket}, rather than the former as our model. The power of this \enquote{bracket approach} becomes especially apparent in the follow-up article \cite{MaWe18b}.
Combining the results of both articles, we eventually determine
all possible half-liberations of $U_n$, i.e., all unitary easy quantum groups $G$ (see \cite{TaWe15a} and \cite{TaWe15b}) with $U_n\subseteq G\subseteq U_n^+$, or equivalently, all categories $\mc C$ with $\langle \PartCrossWW\rangle\supseteq \mc C\supseteq \langle \emptyset\rangle$. In the present article, we classify all categories above a certain category $\mc S_0$; the remainder will be found in \cite{MaWe18b}. Throughout, we employ combinatorial means exclusively, making no use of any quantum group notions or techniques.

	\section{Main Results}
	\label{section:main}
        We introduce and describe certain categories of partitions $(\mc S_w)_{w\in \nnint}$ which may be seen as categories indexed by the cyclic groups $(\integers/w\integers)_{w\in \nnint}$.
	\begin{main}
          \label{theorem:main_1}
                          \begin{enumerate}[label=(\alph*)]
 	      \item \label{item:main_1-1}
                For every $w\in\nnint$, a category of two-colored partitions is given by the set $\mc S_w$ containing all two-colored pair partitions with the following two properties satisfied when rotated to one line:
                \begin{enumerate}
                \item Each block contains one point each of every color.
                \item Between the two legs of any block, the difference in the numbers of black and white points is a multiple of $w$.
                \end{enumerate}

                \item \label{item:main_1-2}
                  The categories $(\mc S_w)_{w\in \nnint}$ are pairwise distinct, and, for all $w\in\mathbb{N}$ (i.e.\ $w\neq 0$), the inclusions and equalities
        \begin{align*}
          \mc S_0=\bigcap_{w'\in \pint}\mc S_{w'}\subseteq \mc S_w\subseteq \mc S_1=\langle \PartCrossWW\rangle
        \end{align*}
        are valid, and for all $w,w'\in\mathbb{N}$ holds
	\begin{align*}
		w'\integers \subseteq w\integers \implies \mathcal{S}_{w'} \subseteq \mathcal{S}_{w}.
	\end{align*}

                  \item \label{item:main_1-3}
 For all $w\in\mathbb{N}$, the category $\mathcal{S}_w$ is generated by the partition
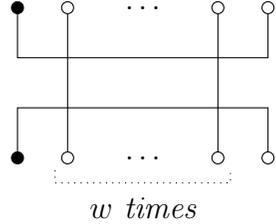
\begin{gather*}
\begin{tikzpicture}[scale=0.666, baseline=0]
\node [scale=0.4, fill=black, draw=black, circle] (x1) at (0,0) {};
\node [scale=0.4, fill=black, draw=black, circle] (x3) at (0,3) {};
\node [scale=0.4, fill=white, draw=black, circle] (y1) at (1,0) {};		
\node [scale=0.4, fill=white, draw=black, circle] (y2) at (1,3) {};				
\node [scale=0.4, fill=white, draw=black, circle] (z1) at (4,0) {};		
\node [scale=0.4, fill=white, draw=black, circle] (z2) at (4,3) {};					
\node [scale=0.4, fill=white, draw=black, circle] (x2) at (5,0) {};
\node [scale=0.4, fill=white, draw=black, circle] (x4) at (5,3) {};				
\draw (x1) -- ++(0,1) -| (x2);
\draw (x3) -- ++(0,-1) -| (x4);		
\draw (y1) -- (y2);
\draw (z1) -- (z2);	
\node at (2.5,0) {$\ldots$};
\node at (2.5,3) {$\ldots$};
\draw [dotted] (0.75,-0.25) -- ++(0,-0.25) -- (4.25,-0.5) -- ++ (0,0.25);
\node at (2.5,-1) {$w$ times};
\end{tikzpicture},
\end{gather*}	
 while  the category $\mc S_0$ is cumulatively generated by the partitions $\PartHalfLibWBW$ and
  
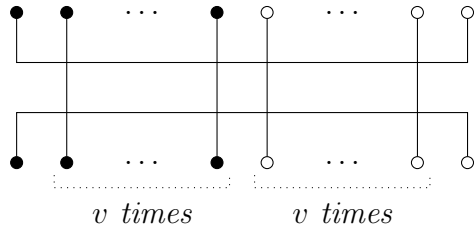
\begin{align*}
    \begin{tikzpicture}[scale=0.666, baseline=(current  bounding  box.center)]
\node [scale=0.4, fill=black, draw=black,circle] (x1) at (0,0) {};
\node [scale=0.4, fill=black, draw=black,circle] (x3) at (0,3) {};
\node [scale=0.4, fill=black, draw=black,circle] (y1) at (1,0) {};		
\node [scale=0.4, fill=black, draw=black,circle] (y2) at (1,3) {};				
\node [scale=0.4, fill=black, draw=black,circle] (z1) at (4,0) {};		
\node [scale=0.4, fill=black, draw=black,circle] (z2) at (4,3) {};					
\node [scale=0.4, fill=white, draw=black,circle] (w1) at (5,0) {};		
\node [scale=0.4, fill=white, draw=black,circle] (w2) at (5,3) {};				
\node [scale=0.4, fill=white, draw=black,circle] (v1) at (8,0) {};		
\node [scale=0.4, fill=white, draw=black,circle] (v2) at (8,3) {};		
\node [scale=0.4, fill=white, draw=black,circle] (x2) at (9,0) {};
\node [scale=0.4, fill=white, draw=black,circle] (x4) at (9,3) {};				
\draw (x1) -- ++(0,1) -| (x2);
\draw (x3) -- ++(0,-1) -| (x4);		
\draw (y1) -- (y2);
\draw (z1) -- (z2);	
\draw (w1) -- (w2);
\draw (v1) -- (v2);	
\node at (2.5,0) {$\ldots$};
\node at (6.5,0) {$\ldots$};
\node at (2.5,3) {$\ldots$};
\node at (6.5,3) {$\ldots$};
\draw [dotted] (0.75,-0.25) -- ++(0,-0.25) -- (4.25,-0.5) -- ++ (0,0.25);
\draw [dotted, xshift=4cm] (0.75,-0.25) -- ++(0,-0.25) -- (4.25,-0.5) -- ++ (0,0.25);
\node at (2.5,-1) {$v$ times};
\node at (6.5,-1) {$v$ times};
\end{tikzpicture}
  \end{align*}
  for all $v\in \pint$.
        		\end{enumerate}
                      \end{main}
                      We classify all categories not contained in $\mc S_0$.
        \begin{main}
          \label{theorem:main_2}
          For every category $\mc C$ of two-colored partitions with \[\langle \emptyset\rangle\subseteq \mc C\subseteq  \langle\PartCrossWW\rangle,\] either $\mc C\subsetneq \mc S_0$ holds or there exists $w\in \nnint$ such that $\mc C=\mc S_w$.
        \end{main}
        
	In Section~\ref{section:concluding_remarks}, we interpret the above results in the quantum group context, align them with the previous research about half-liberations of $U_n$ and hint at their implications, to be fully discussed in the follow-up article.
\pagebreak

\section{Reminder about Partitions and their Categories}

We refer to \cite[Section 1]{TaWe15a} for an introduction to categories of two-colored partitions. However, let us briefly recall the basics.
\vspace{-0.5em}
\subsection{Two-Colored Partitions}
\label{subsection:two-colored-partitions}

\begin{wrapfigure}[7]{r}{5cm}
  \centering
  \vspace{-0.5em}
    \begin{tikzpicture}[scale=0.666]
    \def\dx{1}
    \def\ty{4}
    \def\ox{0}
    \def\oy{0}
    \def\dy{\ty / 4}
    \def\tx{7*\dx}
    \def\od{0.5}
    \draw [dotted] ({\ox-\od},{\oy+\ty}) --  ({\ox+4*\dx+\od},{\oy+\ty});
    \draw [dotted] ({\ox-\od},{\oy}) --  ({\ox+6*\dx+\od},{\oy});
    \node [scale=0.4, circle, draw=black, fill=black] (a0) at ({\ox+0*\dx},{\oy}) {};
    \node [scale=0.4, circle, draw=black, fill=white] (a1) at ({\ox+1*\dx},{\oy}) {};
    \node [scale=0.4, circle, draw=black, fill=black] (a2) at ({\ox+2*\dx},{\oy}) {};
    \node [scale=0.4, circle, draw=black, fill=white] (a3) at ({\ox+3*\dx},{\oy}) {};
    \node [scale=0.4, circle, draw=black, fill=black] (a4) at ({\ox+4*\dx},{\oy}) {};
    \node [scale=0.4, circle, draw=black, fill=black] (a5) at ({\ox+5*\dx},{\oy}) {};
    \node [scale=0.4, circle, draw=black, fill=white] (a6) at ({\ox+6*\dx},{\oy}) {};
    \node [scale=0.4, circle, draw=black, fill=black] (b0) at ({\ox+0*\dx},{\oy+\ty}) {};
    \node [scale=0.4, circle, draw=black, fill=white] (b1) at ({\ox+1*\dx},{\oy+\ty}) {};
    \node [scale=0.4, circle, draw=black, fill=black] (b2) at ({\ox+2*\dx},{\oy+\ty}) {};
    \node [scale=0.4, circle, draw=black, fill=white] (b3) at ({\ox+3*\dx},{\oy+\ty}) {};
    \node [scale=0.4, circle, draw=black, fill=black] (b4) at ({\ox+4*\dx},{\oy+\ty}) {};
    \draw (a0) -- ++(0,{\dy}) -| (a1);
    \draw (a3) -- ++(0,{\dy}) -| (a5);
    \draw (a4) -- ++(0,{2*\dy}) -| (a6);
    \draw (a2) -- (b2);
    \draw (b0) -- ++(0,{-2*\dy}) -| (b3);
    \draw (b1) -- ++(0,{-1*\dy}) -| (b4);
  \end{tikzpicture}
\end{wrapfigure}
A \emph{(two-colored) partition}, is a partition of  a finite set into disjoint subsets, the \emph{blocks}. Moreover, the points are distinguished into an \emph{upper} and a \emph{lower row} of points, each totally ordered.  Finally, each point is \emph{colored} either \emph{white} ($\circ$) or \emph{black} ($\bullet$). We say that these two colors are \emph{inverse} to each other. If each block  of a partition consists of exactly two points, we speak of a \emph{pair partition}. The two points of such a block are called its \emph{legs}. In this article we exclusively deal with pair partitions.\par 
We illustrate two-colored partitions by two lines of black and white dots connected by strings representing the blocks. \par

\vspace{-0.5em}
\subsection{Operations}
\label{subsection:operations}
On the set $\mathcal{P}^{\circ\bullet}$ of all two-colored partitions, one can introduce several operations: The \emph{tensor product} $\otimes$ concatenates two partitions horizontally. 
\par
If we swap the roles of the upper and the lower row of $p\in \mathcal{P}^{\circ\bullet}$, we call this the \emph{involution} $p^*$ of $p$. 
\par A pair $(p,p')$ from $\mathcal{P}^{\circ\bullet}$ is \emph{composable} if the upper row of $p$ and the lower row of $p'$ concur in size and coloration. In the \emph{composition} $pp'$ of $(p,p')$, we then take the points and colors on the lower row from $p$ and on the upper row from $p'$. The blocks of $pp'$ are obtained by vertical concatenation of $p$ and $p'$. 
\par
Reversing the order in each row of a partition is called \emph{reflection}. If that is followed by inversion of colors, we speak of the \emph{verticolor reflection} $\tilde{p}$ of $p\in\mathcal{P}^{\circ\bullet}$. 
\par
Multiple kinds of \emph{rotations} can be defined: Given $p\in \mathcal P^{\circ \bullet}$, we remove the leftmost point $\alpha$ on the upper row and add left to the points on the lower row a point $\beta$ of the inverse color of $\alpha$. The new point $\beta$ belongs to the same block $\alpha$ did before. The resulting partition is denoted by $p^{\rcurvearrowdown}$. Likewise, we can move the rightmost point on the upper row to the right end of the lower row, yielding $p^{\lcurvearrowdown}$. Rotating the outermost points from the lower to the upper row produces the partitions $p^{\lcurvearrowup}$ and $p^{\rcurvearrowup}$. \emph{Cyclic rotations} are defined by $p^{\circlearrowleft}\eqpd(p^{\rcurvearrowdown})^{\rcurvearrowup}$ and $p^{\circlearrowright}\eqpd(p^{\lcurvearrowup})^{\lcurvearrowdown}$.
\par  Lastly, if we \emph{erase} a set $S$ of points in $p\in \mathcal{P}^{\circ\bullet}$, we obtain the partition $E(p,S)$, which is constructed from $p$ by removing $S$ and combining all the remnants of the blocks which had a non-empty intersection with $S$ into one block.
\par
See \cite{TaWe15a} for examples of these operations. 
\vspace{-0.5em}
\subsection{Categories}
\label{section:introdcution-subsection:categories}
If a subset of $\mathcal{P}^{\circ\bullet}$ is closed under tensor products, involution and composition and if it contains $\PartIdenW$,  $\PartIdenB$, $\PartIdenLoBW$, and $\PartIdenLoWB$, it is called a \emph{category of partitions}. Categories are invariant with respect to all rotations and verticolor reflection. Note that they might not be closed under mere reflection. For any set $\mc G\subseteq\Cp$, we denote the smallest category containing $\mc G$ by $\langle \mc G\rangle$ and say that $\mc G$ \emph{generates} $\langle \mc G\rangle$. See  \cite{TaWe15a}. Categories of partitions (with uncolored points) and the above operations were introduced by Banica and Speicher in \cite{BaSp09}. 

		\nopagebreak
		
\section{Basic Concepts for Pair Partitions with Neutral Blocks}
		
\vspace{0.5em}
\noindent
\begin{minipage}{0.62\textwidth}
  \subsection{Cyclic Order}
  In a partition $p\in \mathcal P^{\circ\bullet}$, the lower row $L$ and the upper row $U$ are each natively equipped with a total order, $\leq_L$ and $\leq_U$ respectively. We additionally endow the entirety $L\cup U$ of the points of $p$ with a \emph{cyclic order}. It is uniquely determined by the following four conditions:
It induces $\leq_L$ on $L$, but the \emph{reverse} of $\leq_U$ on $U$, the minimum of $\leq_L$ succeeds the minimum of $\leq_U$, and the maximum of $\leq_L$ precedes the maximum of $\leq_U$.
\end{minipage}
\noindent
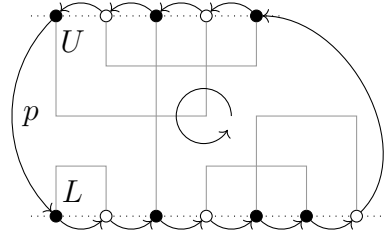
\begin{minipage}{0.38\textwidth}
  \vspace{-0.5em}
	\flushright
 \begin{tikzpicture}[scale=0.666]
    \def\dx{1}
    \def\ty{4}
    \def\ox{0}
    \def\oy{0}
    \def\dy{\ty / 4}
    \def\tx{7*\dx}
    \def\od{0.5}
    \draw [dotted] ({\ox-\od},{\oy+\ty}) --  ({\ox+4*\dx+\od},{\oy+\ty});
    \draw [dotted] ({\ox-\od},{\oy}) --  ({\ox+6*\dx+\od},{\oy});
    \node [scale=0.4, circle, draw=black, fill=black] (a0) at ({\ox+0*\dx},{\oy}) {};
    \node [scale=0.4, circle, draw=black, fill=white] (a1) at ({\ox+1*\dx},{\oy}) {};
    \node [scale=0.4, circle, draw=black, fill=black] (a2) at ({\ox+2*\dx},{\oy}) {};
    \node [scale=0.4, circle, draw=black, fill=white] (a3) at ({\ox+3*\dx},{\oy}) {};
    \node [scale=0.4, circle, draw=black, fill=black] (a4) at ({\ox+4*\dx},{\oy}) {};
    \node [scale=0.4, circle, draw=black, fill=black] (a5) at ({\ox+5*\dx},{\oy}) {};
    \node [scale=0.4, circle, draw=black, fill=white] (a6) at ({\ox+6*\dx},{\oy}) {};
    \node [scale=0.4, circle, draw=black, fill=black] (b0) at ({\ox+0*\dx},{\oy+\ty}) {};
    \node [scale=0.4, circle, draw=black, fill=white] (b1) at ({\ox+1*\dx},{\oy+\ty}) {};
    \node [scale=0.4, circle, draw=black, fill=black] (b2) at ({\ox+2*\dx},{\oy+\ty}) {};
    \node [scale=0.4, circle, draw=black, fill=white] (b3) at ({\ox+3*\dx},{\oy+\ty}) {};
    \node [scale=0.4, circle, draw=black, fill=black] (b4) at ({\ox+4*\dx},{\oy+\ty}) {};
    \draw [lightgray] (a0) -- ++(0,{\dy}) -| (a1);
    \draw [lightgray] (a3) -- ++(0,{\dy}) -| (a5);
    \draw [lightgray] (a4) -- ++(0,{2*\dy}) -| (a6);
    \draw [lightgray] (a2) -- (b2);
    \draw [lightgray] (b0) -- ++(0,{-2*\dy}) -| (b3);
    \draw [lightgray] (b1) -- ++(0,{-1*\dy}) -| (b4);
    	\draw[->] (b4) to [out=135, in=45] (b3);
	\draw[->] (b3) to [out=135, in=45] (b2);
	\draw[->] (b2) to [out=135, in=45] (b1);
	\draw[->] (b1) to [out=135, in=45] (b0);
	\draw[->] (b0) to [out=225, in=135] (a0);
	\draw[->] (a0) to [out=-45, in=225] (a1);
	\draw[->] (a1) to [out=-45, in=225] (a2);
	\draw[->] (a2) to [out=-45, in=225] (a3);
	\draw[->] (a3) to [out=-45, in=225] (a4);
	\draw[->] (a4) to [out=-45, in=225] (a5);
	\draw[->] (a5) to [out=-45, in=225] (a6);
	\draw[->] (a6) to [out=45, in=0] (b4);
        \node at (0.33,3.5) {$U$};
        \node at (0.33,0.5) {$L$};
        \node at (-0.5,2) {$p$};
        \draw[->] (3.5,2) arc (0:330:0.55);        
  \end{tikzpicture}
\end{minipage}
\noindent
\begin{minipage}{0.5\textwidth}
  
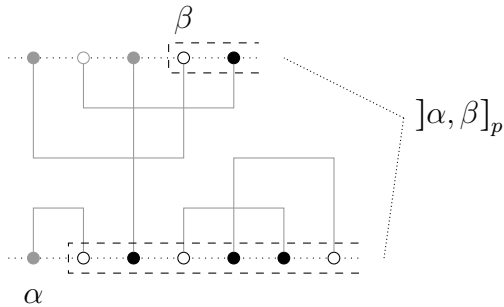
\centering
  \begin{tikzpicture}[scale=0.666]
    \def\dx{1}
    \def\ty{4}
    \def\ox{0}
    \def\oy{0}
    \def\dy{\ty / 4}
    \def\tx{7*\dx}
    \def\od{0.5}
    \def\cx{0.3}
    \def\cy{0.3}    
    \draw [dotted] ({\ox-\od},{\oy+\ty}) --  ({\ox+4*\dx+\od},{\oy+\ty});
    \draw [dotted] ({\ox-\od},{\oy}) --  ({\ox+6*\dx+\od},{\oy});
    \node [scale=0.4, circle, draw=lightgray, fill=lightgray] (a0) at ({\ox+0*\dx},{\oy}) {};
    \node [scale=0.4, circle, draw=black, fill=white] (a1) at ({\ox+1*\dx},{\oy}) {};
    \node [scale=0.4, circle, draw=black, fill=black] (a2) at ({\ox+2*\dx},{\oy}) {};
    \node [scale=0.4, circle, draw=black, fill=white] (a3) at ({\ox+3*\dx},{\oy}) {};
    \node [scale=0.4, circle, draw=black, fill=black] (a4) at ({\ox+4*\dx},{\oy}) {};
    \node [scale=0.4, circle, draw=black, fill=black] (a5) at ({\ox+5*\dx},{\oy}) {};
    \node [scale=0.4, circle, draw=black, fill=white] (a6) at ({\ox+6*\dx},{\oy}) {};
    \node [scale=0.4, circle, draw=lightgray, fill=lightgray] (b0) at ({\ox+0*\dx},{\oy+\ty}) {};
    \node [scale=0.4, circle, draw=lightgray, fill=white] (b1) at ({\ox+1*\dx},{\oy+\ty}) {};
    \node [scale=0.4, circle, draw=lightgray, fill=lightgray] (b2) at ({\ox+2*\dx},{\oy+\ty}) {};
    \node [scale=0.4, circle, draw=black, fill=white] (b3) at ({\ox+3*\dx},{\oy+\ty}) {};
    \node [scale=0.4, circle, draw=black, fill=black] (b4) at ({\ox+4*\dx},{\oy+\ty}) {};
    \draw[lightgray] (a0) -- ++(0,{\dy}) -| (a1);
    \draw[lightgray] (a3) -- ++(0,{\dy}) -| (a5);
    \draw[lightgray] (a4) -- ++(0,{2*\dy}) -| (a6);
    \draw[lightgray] (a2) -- (b2);
    \draw[lightgray] (b0) -- ++(0,{-2*\dy}) -| (b3);
    \draw[lightgray] (b1) -- ++(0,{-1*\dy}) -| (b4);
    \node[yshift=-0.5cm] at (a0) {$\alpha$};
    \node[yshift=0.5cm] at (b3) {$\beta$};
    \draw[dashed] ({\ox+6*\dx+\od},{\oy-\cy})  -- ({\ox+\dx-\cx},{\oy-\cy}) -- ({\ox+\dx-\cx},{\oy+\cy}) -- ({\ox+6*\dx+\od},{\oy+\cy});
    \draw[dashed] ({\ox+4*\dx+\od},{\oy+\ty-\cy})  -- ({\ox+3*\dx-\cx},{\oy+\ty-\cy}) -- ({\ox+3*\dx-\cx},{\oy+\ty+\cy}) -- ({\ox+4*\dx+\od},{\oy+\ty+\cy});
    \node (lab) at ({\ox+\tx+1.5*\dx},{\oy+0.7*\ty}) {$\left]\alpha,\beta\right]_p$};
      \draw[densely dotted] (lab.west) to ({\ox+5*\dx},{\oy+\ty});
      \draw[densely dotted] (lab.west) to ({\ox+7*\dx},{\oy});      
  \end{tikzpicture}
\end{minipage}
\begin{minipage}{0.5\textwidth}
  \vspace{0.25em}
  In our convention of depicting the native ordering left to right on both rows, the cyclic order amounts to the counter-clockwise orientation. \par
 \hspace{0.5em} With respect to this cyclic order it makes sense to speak of \emph{intervals} like $[\alpha,\beta]_p$, $]\alpha,\beta]_p$, etc.\ for points $\alpha,\beta$ in $p$, even if $\alpha$ and $\beta$ are not both contained in the same row of $p$.
\end{minipage}
\begin{minipage}{0.7\textwidth}
  \vspace{0.5em}
\subsection{Connectedness}
\label{subsection:connectedness}
Blocks $B$ and $B'$ in $p\in \mathcal P^{\circ\bullet}$ are said to \emph{cross} if $B\neq B'$ and if there are pairwise distinct points $\alpha,\beta\in B$
   and $\gamma,\delta\in B'$ occurring in the sequence $(\alpha,\gamma,\beta,\delta)$ with respect to the cyclic order.
  If no two blocks cross in $p$, we say that $p$ is \emph{non-crossing}, in short: $p\in \mathcal {NC}^{\circ\bullet}$.  (See  \cite{TaWe15a} for all subcategories of $\mathcal {NC}^{\circ\bullet}$.)
\end{minipage}
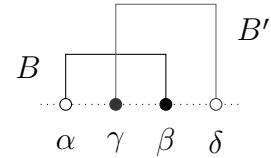
\begin{minipage}{0.3\textwidth}
  \centering
\begin{tikzpicture}[scale=0.666]
    \def\dx{1}
    \def\ty{4}
    \def\ox{0}
    \def\oy{0}
    \def\dy{\ty / 4}
    \def\tx{7*\dx}
    \def\od{0.5}
    \def\cx{0.3}
    \def\cy{0.3}    
    \draw [dotted] ({\ox-\od},{\oy}) --  ({\ox+3*\dx+\od},{\oy});
    \node [scale=0.4, circle, draw=black, fill=white] (a0) at ({\ox+0*\dx},{\oy}) {};
    \node [scale=0.4, circle, draw=darkgray, fill=darkgray] (a1) at ({\ox+1*\dx},{\oy}) {};
    \node [scale=0.4, circle, draw=black, fill=black] (a2) at ({\ox+2*\dx},{\oy}) {};
    \node [scale=0.4, circle, draw=darkgray, fill=white] (a3) at ({\ox+3*\dx},{\oy}) {};
    \draw (a0) -- ++(0,{\dy}) -| (a2);
    \draw[gray] (a1) -- ++(0,{2*\dy}) -| (a3);
    \node [yshift=-0.5cm] at (a0) {$\alpha$};
    \node [yshift=-0.5cm] at (a1) {$\gamma$};
    \node [yshift=-0.5cm] at (a2) {$\beta$};
    \node [yshift=-0.5cm] at (a3) {$\delta$};
    \node [yshift=0.5cm,xshift=-0.5cm] at (a0) {$B$};
    \node [yshift=1cm,xshift=0.5cm] at (a3) {$B'$};            
\end{tikzpicture}
\end{minipage}
\begin{minipage}{0.4\textwidth}
  
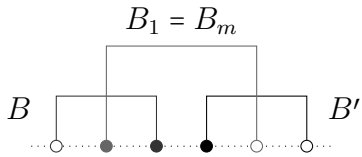
\centering
\begin{tikzpicture}[scale=0.666]
    \def\dx{1}
    \def\ty{4}
    \def\ox{0}
    \def\oy{0}
    \def\dy{\ty / 4}
    \def\tx{7*\dx}
    \def\od{0.5}
    \def\cx{0.3}
    \def\cy{0.3}    
    \draw [dotted] ({\ox-\od},{\oy}) --  ({\ox+5*\dx+\od},{\oy});
    \node [scale=0.4, circle, draw=darkgray, fill=white] (a0) at ({\ox+0*\dx},{\oy}) {};
    \node [scale=0.4, circle, draw=gray, fill=gray] (a1) at ({\ox+1*\dx},{\oy}) {};
    \node [scale=0.4, circle, draw=darkgray, fill=darkgray] (a2) at ({\ox+2*\dx},{\oy}) {};
    \node [scale=0.4, circle, draw=black, fill=black] (a3) at ({\ox+3*\dx},{\oy}) {};
    \node [scale=0.4, circle, draw=gray, fill=white] (a4) at ({\ox+4*\dx},{\oy}) {};
    \node [scale=0.4, circle, draw=black, fill=white] (a5) at ({\ox+5*\dx},{\oy}) {};
    \draw[darkgray] (a0) -- ++(0,{\dy}) -| (a2);
    \draw[black] (a3) -- ++(0,{\dy}) -| (a5);    
    \draw[gray] (a1) -- ++(0,{2*\dy}) -| (a4);
    \node [yshift=0.5cm,xshift=-0.5cm] at (a0) {$B$};
    \node [yshift=0.5cm,xshift=0.5cm] at (a5) {$B'$};
    \node at ({2.5*\dx+\ox},{\oy+2.5*\dy}) {$B_1=B_m$};                
\end{tikzpicture}
\end{minipage}
\begin{minipage}{0.6\textwidth}
    \vspace{0.5em}
\hspace{0.5em} We call the blocks $B$ and $B'$ \emph{connected} if $B=B'$, if $B$ and $B'$ cross, or if  there exist blocks $B_1,\ldots,B_m$ in $p$ such that, writing $B_0\eqpd B$ and $B_{m+1}\eqpd B'$, for every $i\in \mathbb N_0$ with $i\leq m$, block $B_i$ crosses block $B_{i+1}$.
\end{minipage}
\par\vspace{0.5em}
The classes of this equivalence relation are the \emph{connected components} of $p$. And we say that $p$ is \emph{connected} if it has only a single connected component. Erasing the complement of any connected component $S$ of $p$ yields the \emph{factor partition} of $S$.\par
  \vspace{1em}
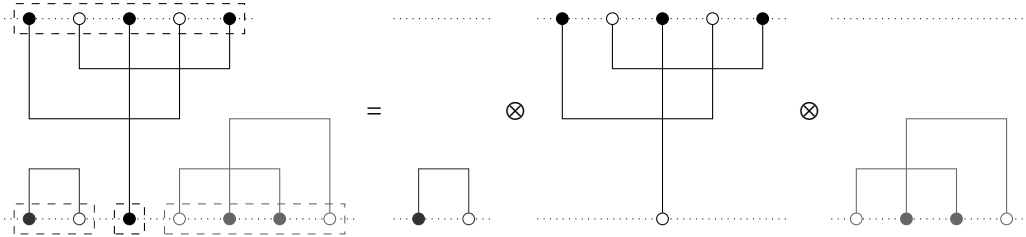
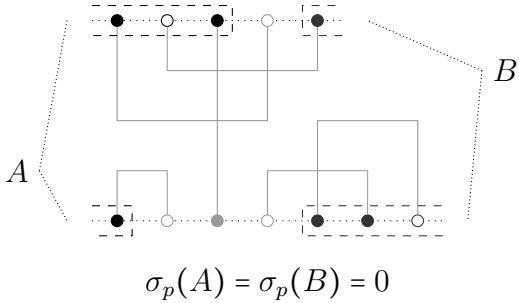
\begin{minipage}{\textwidth}
  \begin{tikzpicture}[scale=0.666,baseline=1.333cm]
    \def\dx{1}
    \def\ty{4}
    \def\ox{0}
    \def\oy{0}
    \def\dy{\ty / 4}
    \def\tx{7*\dx}
    \def\od{0.5}
    \def\cx{0.3}
    \def\cy{0.3}    
    \draw [dotted] ({\ox-\od},{\oy+\ty}) --  ({\ox+4*\dx+\od},{\oy+\ty});
    \draw [dotted] ({\ox-\od},{\oy}) --  ({\ox+6*\dx+\od},{\oy});
    \node [scale=0.4, circle, draw=darkgray, fill=darkgray] (a0) at ({\ox+0*\dx},{\oy}) {};
    \node [scale=0.4, circle, draw=darkgray, fill=white] (a1) at ({\ox+1*\dx},{\oy}) {};
    \node [scale=0.4, circle, draw=black, fill=black] (a2) at ({\ox+2*\dx},{\oy}) {};
    \node [scale=0.4, circle, draw=gray, fill=white] (a3) at ({\ox+3*\dx},{\oy}) {};
    \node [scale=0.4, circle, draw=gray, fill=gray] (a4) at ({\ox+4*\dx},{\oy}) {};
    \node [scale=0.4, circle, draw=gray, fill=gray] (a5) at ({\ox+5*\dx},{\oy}) {};
    \node [scale=0.4, circle, draw=gray, fill=white] (a6) at ({\ox+6*\dx},{\oy}) {};
    \node [scale=0.4, circle, draw=black, fill=black] (b0) at ({\ox+0*\dx},{\oy+\ty}) {};
    \node [scale=0.4, circle, draw=black, fill=white] (b1) at ({\ox+1*\dx},{\oy+\ty}) {};
    \node [scale=0.4, circle, draw=black, fill=black] (b2) at ({\ox+2*\dx},{\oy+\ty}) {};
    \node [scale=0.4, circle, draw=black, fill=white] (b3) at ({\ox+3*\dx},{\oy+\ty}) {};
    \node [scale=0.4, circle, draw=black, fill=black] (b4) at ({\ox+4*\dx},{\oy+\ty}) {};
    \draw[darkgray] (a0) -- ++(0,{\dy}) -| (a1);
    \draw[gray] (a3) -- ++(0,{\dy}) -| (a5);
    \draw[gray] (a4) -- ++(0,{2*\dy}) -| (a6);
    \draw[black] (a2) -- (b2);
    \draw[black] (b0) -- ++(0,{-2*\dy}) -| (b3);
    \draw[black] (b1) -- ++(0,{-1*\dy}) -| (b4);
    \draw [dashed, gray] ({\ox+3*\dx-\cx},{\oy-\cy}) rectangle ({\ox+6*\dx+\cx},{\oy+\cy});
    \draw [dashed, darkgray] ({\ox+0*\dx-\cx},{\oy-\cy}) rectangle ({\ox+1*\dx+\cx},{\oy+\cy});
        \draw [dashed] ({\ox+2*\dx-\cx},{\oy-\cy}) rectangle ({\ox+2*\dx+\cx},{\oy+\cy});
        \draw [dashed] ({\ox+0*\dx-\cx},{\oy+\ty-\cy}) rectangle ({\ox+4*\dx+\cx},{\oy+\ty+\cy});
  \end{tikzpicture}
$=$
  \begin{tikzpicture}[scale=0.666,baseline=1.333cm]
    \def\dx{1}
    \def\ty{4}
    \def\ox{0}
    \def\oy{0}
    \def\dy{\ty / 4}
    \def\tx{7*\dx}
    \def\od{0.5}
    \def\cx{0.3}
    \def\cy{0.3}    
    \draw [dotted] ({\ox-\od},{\oy}) --  ({\ox+1*\dx+\od},{\oy});
    \draw [dotted] ({\ox-\od},{\oy+\ty}) --  ({\ox+1*\dx+\od},{\oy+\ty});    
    \node [scale=0.4, circle, draw=darkgray, fill=darkgray] (a0) at ({\ox+0*\dx},{\oy}) {};
    \node [scale=0.4, circle, draw=darkgray, fill=white] (a1) at ({\ox+1*\dx},{\oy}) {};
    \draw[darkgray] (a0) -- ++(0,{\dy}) -| (a1);
  \end{tikzpicture}
  $\otimes$
        \begin{tikzpicture}[scale=0.666,baseline=1.333cm]
    \def\dx{1}
    \def\ty{4}
    \def\ox{0}
    \def\oy{0}
    \def\dy{\ty / 4}
    \def\tx{7*\dx}
    \def\od{0.5}
    \def\cx{0.3}
    \def\cy{0.3}    
    \draw [dotted] ({\ox-\od},{\oy+\ty}) --  ({\ox+4*\dx+\od},{\oy+\ty});
    \draw [dotted] ({\ox-\od},{\oy}) --  ({\ox+4*\dx+\od},{\oy});
    \node [scale=0.4, circle, draw=black, fill=white] (a2) at ({\ox+2*\dx},{\oy}) {};
    \node [scale=0.4, circle, draw=black, fill=black] (b0) at ({\ox+0*\dx},{\oy+\ty}) {};
    \node [scale=0.4, circle, draw=black, fill=white] (b1) at ({\ox+1*\dx},{\oy+\ty}) {};
    \node [scale=0.4, circle, draw=black, fill=black] (b2) at ({\ox+2*\dx},{\oy+\ty}) {};
    \node [scale=0.4, circle, draw=black, fill=white] (b3) at ({\ox+3*\dx},{\oy+\ty}) {};
    \node [scale=0.4, circle, draw=black, fill=black] (b4) at ({\ox+4*\dx},{\oy+\ty}) {};
    \draw[black] (a2) -- (b2);
    \draw[black] (b0) -- ++(0,{-2*\dy}) -| (b3);
    \draw[black] (b1) -- ++(0,{-1*\dy}) -| (b4);
  \end{tikzpicture}
  $\otimes$
        \begin{tikzpicture}[scale=0.666,baseline=1.333cm]
    \def\dx{1}
    \def\ty{4}
    \def\ox{0}
    \def\oy{0}
    \def\dy{\ty / 4}
    \def\tx{7*\dx}
    \def\od{0.5}
    \def\cx{0.3}
    \def\cy{0.3}    
    \draw [dotted] ({\ox-\od},{\oy+\ty}) --  ({\ox+3*\dx+\od},{\oy+\ty});
    \draw [dotted] ({\ox-\od},{\oy}) --  ({\ox+3*\dx+\od},{\oy});
    \node [scale=0.4, circle, draw=gray, fill=white] (a3) at ({\ox+0*\dx},{\oy}) {};
    \node [scale=0.4, circle, draw=gray, fill=gray] (a4) at ({\ox+1*\dx},{\oy}) {};
    \node [scale=0.4, circle, draw=gray, fill=gray] (a5) at ({\ox+2*\dx},{\oy}) {};
    \node [scale=0.4, circle, draw=gray, fill=white] (a6) at ({\ox+3*\dx},{\oy}) {};
    \draw[gray] (a3) -- ++(0,{\dy}) -| (a5);
    \draw[gray] (a4) -- ++(0,{2*\dy}) -| (a6);
  \end{tikzpicture}
\end{minipage}
\par
\pagebreak
\begin{wrapfigure}[9]{l}{7cm}
  \begin{tikzpicture}[scale=0.666]
    \def\dx{1}
    \def\ty{4}
    \def\ox{0}
    \def\oy{0}
    \def\dy{\ty / 4}
    \def\tx{7*\dx}
    \def\od{0.5}
    \def\cx{0.3}
    \def\cy{0.3}        
    \draw [dotted] ({\ox-\od},{\oy+\ty}) --  ({\ox+4*\dx+\od},{\oy+\ty});
    \draw [dotted] ({\ox-\od},{\oy}) --  ({\ox+6*\dx+\od},{\oy});
    \node [scale=0.4, circle, draw=black, fill=black] (a0) at ({\ox+0*\dx},{\oy}) {};
    \node [scale=0.4, circle, draw=lightgray, fill=white] (a1) at ({\ox+1*\dx},{\oy}) {};
    \node [scale=0.4, circle, draw=lightgray, fill=lightgray] (a2) at ({\ox+2*\dx},{\oy}) {};
    \node [scale=0.4, circle, draw=lightgray, fill=white] (a3) at ({\ox+3*\dx},{\oy}) {};
    \node [scale=0.4, circle, draw=darkgray, fill=darkgray] (a4) at ({\ox+4*\dx},{\oy}) {};
    \node [scale=0.4, circle, draw=darkgray, fill=darkgray] (a5) at ({\ox+5*\dx},{\oy}) {};
    \node [scale=0.4, circle, draw=darkgray, fill=white] (a6) at ({\ox+6*\dx},{\oy}) {};
    \node [scale=0.4, circle, draw=black, fill=black] (b0) at ({\ox+0*\dx},{\oy+\ty}) {};
    \node [scale=0.4, circle, draw=black, fill=white] (b1) at ({\ox+1*\dx},{\oy+\ty}) {};
    \node [scale=0.4, circle, draw=black, fill=black] (b2) at ({\ox+2*\dx},{\oy+\ty}) {};
    \node [scale=0.4, circle, draw=lightgray, fill=white] (b3) at ({\ox+3*\dx},{\oy+\ty}) {};
    \node [scale=0.4, circle, draw=darkgray, fill=darkgray] (b4) at ({\ox+4*\dx},{\oy+\ty}) {};
    \draw[lightgray] (a0) -- ++(0,{\dy}) -| (a1);
    \draw[lightgray] (a3) -- ++(0,{\dy}) -| (a5);
    \draw[lightgray] (a4) -- ++(0,{2*\dy}) -| (a6);
    \draw[lightgray] (a2) -- (b2);
    \draw[lightgray] (b0) -- ++(0,{-2*\dy}) -| (b3);
    \draw[lightgray] (b1) -- ++(0,{-1*\dy}) -| (b4);
    \draw[darkgray,dashed] ({\ox+6*\dx+\od},{\oy-\cy})  -- ({\ox+4*\dx-\cx},{\oy-\cy}) -- ({\ox+4*\dx-\cx},{\oy+\cy}) -- ({\ox+6*\dx+\od},{\oy+\cy});
    \draw[darkgray,dashed] ({\ox+4*\dx+\od},{\oy+\ty-\cy})  -- ({\ox+4*\dx-\cx},{\oy+\ty-\cy}) -- ({\ox+4*\dx-\cx},{\oy+\ty+\cy}) -- ({\ox+4*\dx+\od},{\oy+\ty+\cy});
    \draw[black,dashed] ({\ox+0*\dx-\od},{\oy+\ty+\cy})  -- ({\ox+2*\dx+\cx},{\oy+\ty+\cy}) -- ({\ox+2*\dx+\cx},{\oy+\ty-\cy}) -- ({\ox+0*\dx-\od},{\oy+\ty-\cy});
    \draw[black,dashed] ({\ox+0*\dx-\od},{\oy+\cy})  -- ({\ox+0*\dx+\cx},{\oy+\cy}) -- ({\ox+0*\dx+\cx},{\oy-\cy}) -- ({\ox+0*\dx-\od},{\oy-\cy});
    \node (labA) at ({\ox-2*\dx},{\oy+0.25*\ty}) {$A$};
    \node (labB) at ({\ox+\tx+0.75*\dx},{\oy+0.75*\ty}) {$B$};
    \draw[densely dotted] (labA.east) to ({\ox+(-1)*\dx},{\oy+\ty});
    \draw[densely dotted] (labA.east) to ({\ox+(-1)*\dx},{\oy});          
    \draw[densely dotted] (labB.west) to ({\ox+5*\dx},{\oy+\ty});
    \draw[densely dotted] (labB.west) to ({\ox+7*\dx},{\oy});      
    \node at ({\ox+0.5*6*\dx},{\oy-1.25*\dy}) {$\sigma_p(A)=\sigma_p(B)=0$};
  \end{tikzpicture}
\end{wrapfigure}
  \subsection{Color Sum and Neutral Sets}
\label{subsection_neutral_subsets_and_the_color_sum}
The \emph{normalized color} of a point is simply its color in the case of a lower point, but the inverse of its color in the case of an upper point. We think of it as the color the point would have if rotated to the lower line.
\par
For any set $S$ of points in $p\in\mathcal{P}^{\circ\bullet}$, the \emph{color sum} $\sigma_p(S)$  is given by the difference between the numbers of points in $S$ of normalized color $\bullet$ and of normalized color $\circ$.
\par
A \emph{neutral} set $S$ of points of $p$ is a null set of the signed measure $\sigma_p$, i.e.\ one with $\sigma_p(S)=0$. Categories of partitions are closed under erasing neutral intervals: Simply use compositions with the pair partitions $\PartIdenLoBW$ and $\PartIdenLoWB$ and their involutions. \par

            \begin{definition}
		We say that $p\in\mathcal{P}^{\circ\bullet}$ is a \emph{pair partition with neutral blocks},  $p\in\mathcal{P}^{\circ\bullet}_{2,\mathrm{nb}}$ in short, if $p$ is a pair partition and all its blocks are neutral.\par 
              \end{definition}
              For example, the partition depicted in Section~\ref{subsection:two-colored-partitions} is a pair partition with neutral blocks. The set $\mathcal P^{\circ\bullet}_{2,\mathrm{nb}}=\left\langle\PartCrossWW\right\rangle$ is in fact a category of partitions, generated by the crossing partition. (See $\mathcal O_{\mathrm{grp,loc}}$ in \cite[Proposition~3.3, Lemma~8.2, Theorem~ 8.3]{TaWe15a}.) \par
 \subsection{Sectors}
              \begin{wrapfigure}[8]{r}{7cm}
                \centering
                \vspace{-0.25em}
  \begin{tikzpicture}[scale=0.666]
    \def\dx{1}
    \def\ty{4}
    \def\ox{0}
    \def\oy{0}
    \def\dy{\ty / 4}
    \def\tx{7*\dx}
    \def\od{0.5}
    \def\cx{0.3}
    \def\cy{0.3}            
    \draw [dotted] ({\ox-\od},{\oy+\ty}) --  ({\ox+4*\dx+\od},{\oy+\ty});
    \draw [dotted] ({\ox-\od},{\oy}) --  ({\ox+6*\dx+\od},{\oy});
    \node [scale=0.4, circle, draw=lightgray, fill=lightgray] (a0) at ({\ox+0*\dx},{\oy}) {};
    \node [scale=0.4, circle, draw=lightgray, fill=white] (a1) at ({\ox+1*\dx},{\oy}) {};
    \node [scale=0.4, circle, draw=black, fill=black, thick] (a2) at ({\ox+2*\dx},{\oy}) {};
    \node [scale=0.4, circle, draw=black, fill=white] (a3) at ({\ox+3*\dx},{\oy}) {};
    \node [scale=0.4, circle, draw=black, fill=black] (a4) at ({\ox+4*\dx},{\oy}) {};
    \node [scale=0.4, circle, draw=black, fill=black] (a5) at ({\ox+5*\dx},{\oy}) {};
    \node [scale=0.4, circle, draw=black, fill=white] (a6) at ({\ox+6*\dx},{\oy}) {};
    \node [scale=0.4, circle, draw=lightgray, fill=lightgray] (b0) at ({\ox+0*\dx},{\oy+\ty}) {};
    \node [scale=0.4, circle, draw=lightgray, fill=white] (b1) at ({\ox+1*\dx},{\oy+\ty}) {};
    \node [scale=0.4, circle, draw=black, fill=black, thick] (b2) at ({\ox+2*\dx},{\oy+\ty}) {};
    \node [scale=0.4, circle, draw=black, fill=white] (b3) at ({\ox+3*\dx},{\oy+\ty}) {};
    \node [scale=0.4, circle, draw=black, fill=black] (b4) at ({\ox+4*\dx},{\oy+\ty}) {};
    \draw [lightgray] (a0) -- ++(0,{\dy}) -| (a1);
    \draw [lightgray] (a3) -- ++(0,{\dy}) -| (a5);
    \draw [lightgray] (a4) -- ++(0,{2*\dy}) -| (a6);
    \draw [lightgray] (b0) -- ++(0,{-2*\dy}) -| (b3);
    \draw [lightgray] (b1) -- ++(0,{-1*\dy}) -| (b4);
    \draw [thick] (a2) -- (b2);
    \draw[dashed] ({\ox+6*\dx+\od},{\oy-\cy})  -- ({\ox+2*\dx-\cx},{\oy-\cy}) -- ({\ox+2*\dx-\cx},{\oy+\cy}) -- ({\ox+6*\dx+\od},{\oy+\cy});
    \draw[dashed] ({\ox+4*\dx+\od},{\oy+\ty-\cy})  -- ({\ox+2*\dx-\cx},{\oy+\ty-\cy}) -- ({\ox+2*\dx-\cx},{\oy+\ty+\cy}) -- ({\ox+4*\dx+\od},{\oy+\ty+\cy});
    \node (lab) at ({\ox+\tx+\dx},{\oy+0.5*\ty}) {$S$};
      \draw[densely dotted] (lab.west) to ({\ox+5*\dx},{\oy+\ty});
      \draw[densely dotted] (lab.west) to ({\ox+7*\dx},{\oy});
      \node at ({\ox+2*\dx-1.75*\cx},{\oy+0.375*\ty}) {$\partial S$};
  \end{tikzpicture}
 	\end{wrapfigure}
   If a proper subset $S$ of points in $p\in\mathcal{P}^{\circ\bullet}$ is an interval with respect to the cyclic order, we refer to the set comprising exactly the first point of $S$  and the last point of $S$ as the \emph{boundary} $\partial S$ of $S$.  Also, define $\inte (S)\eqpd S \backslash \partial S$, the \emph{interior} of the set $S$.\par We call $S$ a \emph{sector} in $p$ if $\partial S$ is a block of $p$. A sector $S'$ in $p$ with $S'\subseteq S$ is called a \emph{subsector} of $S$.
  \par

  \section{Definition\texorpdfstring{ of $\mc S_w$}{}, Distinctness and Set Relationships \texorpdfstring{\\{[Main Theorem~\ref*{theorem:main_1}~{\normalfont\ref*{item:main_1-2}}]}}{}}
  \label{section:definition_set_relationships}
The central result of this article concerns the following sets of  partitions:
\begin{definition}
	\label{definition:the_three_families_of_sets_of_neutral_pair_partitions}
	For every $w\in\mathbb{N}_0$, denote by
 $\mathcal{S}_w$ the set of all partitions $p\in\mathcal P^{\circ\bullet}_{2,\mathrm{nb}}$ such that $\sigma_p(S)\in w\mathbb{Z}$ for all sectors $S$ in $p$.
\end{definition}
Part \ref{item:main_1-2} of Main~Theorem~\ref{theorem:main_1} follows immediately from this definition. Moreover, we observe that the choices $w\integers$ or $\integers/w\integers$ as an index for $\mc S_w$ would be natural alternatives. This fact, together with observations made in Section~\ref{section:concluding_remarks} and \cite{MaWe18b}, justify the title of the present article.
\begin{proposition} 
	\label{proposition:set_relationships_between_the_categories}
	The following identities and inclusions are valid for all $w\in\mathbb{N}$:
        \begin{align*}
          \mathcal{S}_0 =\bigcap_{w'\in\mathbb{N}} \mathcal{S}_{w'} \subseteq \mathcal{S}_w\subseteq \mathcal{S}_1=\mathcal P^{\circ\bullet}_{2,\mathrm{nb}}=\langle\PartCrossWW\rangle
        \end{align*}
	Furthermore, for all $w,w'\in\mathbb{N}$ holds
	\begin{align*}
	w'\integers \subseteq w\integers \implies \mathcal{S}_{w'} \subseteq\mathcal{S}_{w}.
	\end{align*}
        Moreover, $\mc S_w\neq \mc S_{w'}$ for all $w,w'\in \nnint$ with $w\neq w'$.
 \end{proposition}
\begin{proof}
	\begin{enumerate}[wide, labelwidth = !]
			\item For every $p\in \mathcal P^{\circ\bullet}_{2,\mathrm{nb}}$, all its sectors $S$ necessarily satisfy $\sigma_p(S)\in\mathbb Z$ by definition. Hence, $\mathcal{S}_1=\mathcal P^{\circ\bullet}_{2,\mathrm{nb}}$ follows. In \cite[Proposition~3.3, Lemma~8.2, Theorem~8.3]{TaWe15a}, it was proven that the category $\mathcal P^{\circ\bullet}_{2,\mathrm{nb}}$ is generated by $\PartCrossWW$.
			\item For $w,w'\in\mathbb{N}$, a partition $p\in \mathcal P^{\circ\bullet}_{2,\mathrm{nb}}$ and a sector $S$ in $p$, the relation $\sigma_p(S)\in w'\mathbb{Z}$ implies $\sigma_p(S)\in w\mathbb{Z}$ if  $w'\integers\subseteq w\integers$. Hence, $\mathcal{S}_{w'} \subseteq\mathcal{S}_{w}$.
		\item For every integer $w\in\mathbb Z$ holds $w=0$ if and only if $w\integers \subseteq w'\integers$ is true for all $w'\in\mathbb N$. We conclude $\bigcap_{w'\in \mathbb{N}}\mathcal{S}_{w'}=\mathcal{S}_{0}$.
                \item The partitions listed in Main~Theorem~\ref{theorem:main_1} provide examples of elements of the sets $\mathcal S_w$ and prove $\mc S_{w'}\neq \mc S_{w}$ for $w,w'\in \nnint$ with $w'\neq w$.\qedhere
	\end{enumerate}
      \end{proof}

\section{Category Property\texorpdfstring{ of $\mathcal{S}_w$\\{[Main Theorem~\ref*{theorem:main_1}~{\normalfont\ref*{item:main_1-1}}]}}{}}
\label{section:invariance}
To show that $\mathcal S_w$ is a category of partitions for every $w\in \mathbb N_0$, we need to check that the category operations preserve the defining property of $\mathcal S_w$. The proof is facilitated by two simple facts:
\begin{lemma}
  \label{lemma:anti_symmetry_of_sector_color_sums}
  If $S$ and $S'$ are the two sectors of a block in $p\in \mathcal P^{\circ\bullet}_{2,\mathrm{nb}}$, then
  \begin{align*}
   \sigma_p(S')=-\sigma_p(S). 
  \end{align*}
\end{lemma}
\begin{proof}
  The entirety $P_p$ of the points of $p$ is partitioned by the blocks of $p$. Since those are all neutral due to $p\in\mathcal P^{\circ\bullet}_{2,\mathrm{nb}}$, so is their union: $\sigma_p(P_p)=0$.
  The assumptions on $S$ and $S'$ imply that $P_p$ is also partitioned by $\mathrm{int}(S)$, $\partial S=\partial S'$, and $\mathrm{int}(S')$.   Now, $\sigma_p(\partial S)=\sigma_p(\partial S')=0$ yields the assertion.  
\end{proof}
In conclusion, it suffices to check $\sigma_p(S)\in w\mathbb Z$ for only a single and arbitrary one of the two sectors $S$ of each block in a given $p\in\mathcal P^{\circ\bullet}_{2,\mathrm{nb}}$ in order to prove $p\in\mathcal{S}_w$. This can be done comfortably:
\begin{lemma}
 \label{lemma:color_sum_decomposition}
 If $\alpha,\alpha'$ and $\beta$ are arbitrary points in $p\in\mathcal P^{\circ\bullet}_{2,\mathrm{nb}}$, then
 \begin{align*}
 \sigma_p\left(\left]\alpha,\alpha'\right]_p\right)= \sigma_p\left(\left]\alpha,\beta\right]_p\right)+\sigma_p\left(\left]\beta,\alpha'\right]_p\right),
 \end{align*}
 regardless of the order in which $\alpha,\alpha'$ and $\beta$ appear in $p$.\par
 Similar results hold for other kinds of intervals and other decompositions, e.g., $\sigma_p\left(\left]\alpha,\alpha'\right[_p\right)= \sigma_p\left(\left]\alpha,\beta\right]_p\right)+\sigma_p\left(\left]\beta,\alpha'\right[_p\right)$ for points $\alpha$, $\alpha'$, $\beta$ in $p \in \mathcal P^{\circ\bullet}_{2,\mathrm{nb}}$.
\end{lemma}
\begin{proof}
  This lemma is clear if the order is $(\alpha,\beta,\alpha')$. If it is $(\alpha,\alpha',\beta)$ instead, we promptly employ the evident case of the lemma to find $\sigma_p\left(\left]\alpha,\beta\right]_p\right)=\sigma_p\left(\left]\alpha,\alpha'\right]_p\right)+\sigma_p\left(\left]\alpha',\beta\right]\right)$ and, thus, $\sigma_p\left(\left]\alpha,\beta\right]_p\right)+\sigma_p\left(\left]\beta,\alpha'\right]_p\right)=\sigma_p\left(\left]\alpha,\alpha'\right]_p\right)+\sigma_p\left(\left]\alpha',\alpha'\right]_p\right)$. The claim then follows from $\sigma_p\left(\left]\alpha',\alpha'\right]_p\right)=0$, i.e.\ the fact that $p$ as a whole is neutral.
\end{proof}

\begin{proposition}
	\label{proposition:category_property_of_the_sets_S_w}
	For all $w\in\mathbb{N}_0$, the set $\mathcal{S}_{w}$ is a category of partitions. 
\end{proposition}
	\noindent{\scshape Proof.} 
Involution merely changes the sign of the color sum of a given sector. Moreover, since each partition in $\mathcal P^{\circ\bullet}_{2,\mathrm{nb}}$, seen as the set of all of its points, is neutral, the color sums of sectors remain unchanged even if an entire such partition is injected via tensor product into a sector which is spread across both lines. In order to prove that $\mathcal S_w$ is a category, we thus only need to show that $\mathcal{S}_w$ is invariant under composition.\par 
		\begin{wrapfigure}[13]{r}{7.5cm}
			\vspace{-1em}
			\centering
			\begin{tikzpicture}[scale=0.75]
			\draw [dotted] (-0.5,6) -- (5.5,6);
			\draw [dotted] (-0.5,3) -- (7.5,3);
			\draw [dotted] (-0.5,0) -- (5.5,0);		
			
			\coordinate (a1) at (1.5,4);
			\coordinate (a2) at (6,4);		
			\coordinate (a3) at (6,2);
			\coordinate (a4) at (4.5,2);		
			\coordinate (a5) at (4.5,5);
			\coordinate (a6) at (0,5);				
			\coordinate (a7) at (0,2);				
			\coordinate (a8) at (3,2);				
			\node[fill=white, draw,circle,scale=0.4, thick] (x3) at (1.5,0) {}; 								
			\node[fill=white,draw,circle,scale=0.4, thick] (x4) at (1.5,3) {}; 						
			\node[fill=black,draw,circle,scale=0.4, thick] (x5) at (6,3) { }; 				
			\node[fill=white,draw,circle,scale=0.4, thick] (x6) at (4.5,3) { }; 				
			\node[fill=black,draw,circle,scale=0.4, thick] (x7) at (0,3) { }; 		
			\node[fill=white,draw,circle,scale=0.4, thick] (x8) at (3,3) { }; 
			\node[fill=white,draw,circle,scale=0.4, thick] (x9) at (3,6) { }; 			
			\draw[thick,->] (x8) to (x9);
			\draw[thick,->] (x7) to (a7) to (a8) to (x8);		
			\draw[thick,->] (x6) to (a5) to (a6) to (x7);
			\draw[thick,->] (x5) to (a3) to (a4) to (x6);		
			\draw[thick,->] (x4) to (a1) to (a2) to (x5);				
			\draw[thick,->] (x3) to (x4);							
			\node at (7.25,1.25) {$p$};
			\node at (7.25,4.75) {$p'$};
			\node at (3.75,2.5) {$x_{2n+1}$};
			\node at (0.5,2.5) {$x_4$};	
			\node at (2,2.5) {$x_1$};	
			\node at (5,2.5) {$x_3$};	
			\node at (6.5,2.5) {$x_2$};	
			\node[draw, fill=gray, scale=0.333] (last-bottom) at (5.5,0) {};
			\node [right of=last-bottom,xshift=-0.5cm, yshift=-0.5cm] {$m$};
			\node[draw, fill=gray, scale=0.333] (first-bottom) at (-0.5,0) {};			
			\node [left of=first-bottom, xshift=0.5cm, yshift=-0.5cm] {$1$};			
			\node[draw, fill=gray, scale=0.333] (last-middle) at (7.5,3) {};
			\node [right of=last-middle, xshift=-0.5cm] {$l$};			
			\node[draw, fill=gray, scale=0.333] (first-middle) at (-0.5,3) {};						
			\node [left of=first-middle, xshift=0.5cm] {$1$};			
			\node[draw, fill=gray, scale=0.333] (last-top) at (5.5,6) {};
			\node [right of=last-top,xshift=-0.5cm, yshift=0.5cm] {$k$};			
			\node[draw, fill=gray, scale=0.333] (first-top) at (-0.5,6) {};									
			\node [left of=first-top, xshift=0.5cm, yshift=0.5cm] {$1$};			
			\node [below of=x3, yshift=0.5cm] {$a$};			
			\node [above of=x9, yshift=-0.5cm] {$b$};						
			\end{tikzpicture}
		\end{wrapfigure}
			Let $(p,p')$ be a composable pairing from $\mathcal{S}_w$.  When dealing with compositions, it is convenient to use the following notation:
We can reference points in partitions by their rank $x\in\mathbb{N}$ in the ordering of the row (not the cyclic order) if we write $\lop{x}$ for a lower point and $\upp{x}$ for an upper one.			\par

We only treat the case of a through block $B$ in $pp'$, i.e.\ a block comprising both a lower point $\lop{a}$ and an upper point $\upp{b}$. We now prove $\sigma_{pp'}\left(\left]\lop{a},\upp{b}\right[_{pp'}\right)\in w\mathbb Z$.\par
Because $p$ and $p'$ are pair partitions, the points of $B$ must be linked through a sequence of blocks, alternately of $p$ and of $p'$. The blocks successively connect $\lop{a}$, oddly many points of indices $x_1,\ldots,x_{2n+1}$ on the common row, and, finally, $\upp{b}$. Suppose that $p'$ has $k$ upper and $l$ lower points and that $p$ has $m$ lower points.\par
We decompose, using Lemma~\ref{lemma:color_sum_decomposition},
			\begin{align}
			\sigma_{pp'}\left( \left]\lop{a},\upp{b}\right[_{pp'} \right)&=\sigma_{pp'}\left( \left]\lop{a},\lop{m}\right]_{pp'} \right)+\sigma_{pp'}\left( \left[\upp{k},\upp{b}\right[_{pp'} \right)\notag\\
			&=\sigma_{p}\left( \left]\lop{a},\lop{m}\right]_{p} \right)+\sigma_{p'}\left( \left[\upp{k},\upp{b}\right[_{p'} \right)\notag\\
			&=\sigma_{p}\left( \left]\lop{a},\lop{m}\right]_{p} \right)+\sigma_{p'}\left( \left]\lop{x_{2n+1}},\upp{b}\right[_{p'} \right)-\sigma_{p'}\left( \left]\lop{x_{2n+1}},\lop{l}\right]_{p'} \right).
						\label{equation:proposition_category_property_of_the_sets_S_w-proof-1}
			\end{align}
			The order induced by the cyclic order of $p'$ on its lower row is the exact opposite of the one induced by the cyclic order of $p$ on its upper row.
			Whereas the colors of those two rows match, for a rank $z\in\mathbb{N}$, $z\leq l$, the \emph{normalized} colors of the point $\upp{z}$ in $p$ and the point $\lop{z}$ in $p'$ are inverse to each other. Applying this to the set $\left]\lop{x_{2n+1}},\lop{l}\right]_{p'} $, we find
			\begin{align*}
			\sigma_{p'}\left( \left]\lop{x_{2n+1}},\lop{l}\right]_{p'} \right)=-\sigma_{p}\left( \left[\upp{l},\upp{x_{2n+1}}\right[_{p} \right).
			\end{align*}
			Inserting this equality into Equation \eqref{equation:proposition_category_property_of_the_sets_S_w-proof-1} yields
			\begin{align}
			\sigma_{pp'}\left( \left]\lop{a},\upp{b}\right[_{pp'} \right)
			&=\sigma_{p}\left( \left]\lop{a},\lop{m}\right]_{p} \right)+\sigma_{p'}\left( \left]\lop{x_{2n+1}},\upp{b}\right[_{p'} \right)+\sigma_{p}\left( \left[\upp{l},\upp{x_{2n+1}}\right[_{p} \right)\notag \\
			&=\sigma_{p}\left( \left]\lop{a},\upp{x_{2n+1}}\right[_{p} \right)+\sigma_{p'}\left( \left]\lop{x_{2n+1}},\upp{b}\right[_{p'} \right) \label{equation:proposition_category_property_of_the_sets_S_w-proof-2}.
			\end{align}
                        \par
                        If $n=0$, the sets $\left]\lop{a},\upp{x_{2n+1}}\right[_{p}$ and $\left]\lop{x_{2n+1}},\upp{b}\right[_{p'}$ are the interiors of sectors in $p$ and $p'$, respectively. Hence, in this case, Equation~\eqref{equation:proposition_category_property_of_the_sets_S_w-proof-2} already proves the claim since $\sigma_{p}\left( \left]\lop{a},\upp{x_{2n+1}}\right[_{p} \right), \sigma_{p'}\left( \left]\lop{x_{2n+1}},\upp{b}\right[_{p'} \right)\in w\mathbb Z$ by assumption.
                        \par
			If $n>0$, we can employ again Lemma~\ref{lemma:color_sum_decomposition} iteratively to decompose
			\begin{align*}
			\sigma_{p}\left( \left]\lop{a},\upp{x_{2n+1}}\right[_{p} \right)&= \sigma_{p}\left( \left]\lop{a},\upp{x_{1}}\right[_{p} \right)+\sum_{j=1}^n\sigma_{p}\left( \left[\upp{x_{2j-1}},\upp{x_{2j}}\right]_{p} \right)+\sum_{j=1}^n\sigma_{p}\left( \left]\upp{x_{2j}},\upp{x_{2j+1}}\right[_{p} \right).
			\end{align*}
			Using again the relationship between the orientations and the normalized colors of $p$ and $p'$ and combining this equality with Equation~\eqref{equation:proposition_category_property_of_the_sets_S_w-proof-2} gives an expression for $\sigma_{pp'}\left( \left]\lop{a},\lop{b}\right[_{pp'} \right)$ as a sum of the color sums of sectors in $p$ and $p'$, proving that it must be a multiple of $w$.
			That is what we needed to show. \hfill $\square$

\section{Proof Technique: Brackets}
\label{section:brackets}
In this section, we develop general tools to classify subcategories of $\Cppnb$ and find their generators. They will be employed in Sections~\ref{section:generators} and~\ref{section:classification} to prove the remaining parts of the \hyperref[section:main]{Main Theorems} and in the follow-up article \cite{MaWe18b} to determine all subcategories of $\mc S_0$.
\subsection{Brackets} 
To identify the generators of subcategories of $\Cppnb$, we find a set $\resbr\subseteq \mathcal P^{\circ\bullet}_{2,\mathrm{nb}}$ of \enquote{universal generators}, namely with the property $\mathcal{C}=\left\langle \mathcal{C}\cap\resbr\right\rangle$ for \emph{all} categories $\mathcal{C}\subseteq \mathcal P^{\circ\bullet}_{2,\mathrm{nb}}$ (see Proposition~\ref{proposition:generation_of_categories_by_residual_brackets}). 

The set $\resbr$ will be defined in Section~\ref{subsection:residual_brackets} as a special subset of the set of all \emph{brackets}. These latter partitions we introduce and study in the following.\par
\vspace{0.25cm}
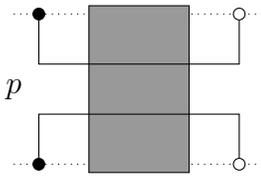
       \begin{minipage}[c]{3cm}
	\centering
	\begin{tikzpicture}[scale=0.666]
          \draw [dotted, shift={(-0.5,0)}] (0,0) -- (5,0);
	\draw [dotted, shift={(-0.5,3)}] (0,0) -- (5,0);		
	\node [scale=0.4, fill=black, draw=black,circle] (x1) at (0,0) {};
	\node [scale=0.4, fill=white, draw=black,circle] (x2) at (4,0) {};		
	\node [scale=0.4, fill=black, draw=black,circle] (y1) at (0,3) {};
	\node [scale=0.4, fill=white, draw=black,circle] (y2) at (4,3) {};				
	\draw [fill=lightgray] (1,-0.166) rectangle (3,3.166);
	\draw (x1) -- ++ (0,1) -| (x2);
	\draw (y1) -- ++ (0,-1) -| (y2);
	\node at (-0.5,1.5) {$p$};	
	\end{tikzpicture}
      \end{minipage}
      \hfill
       \begin{minipage}[c]{0.666\textwidth}
         \begin{definition}
           \begin{enumerate}[label=(\alph*), labelwidth=!]
           \item We call $p\in\mathcal P^{\circ\bullet}_{2,\mathrm{nb}}$ a \emph{bracket} if $p$ is projective, i.e.\ $p=p^*$ and $p^2=p$, and if the lower row of $p$ is a sector in $p$. 
             \item The set of all brackets is denoted by $\mc B$.
           \end{enumerate}

	\end{definition}
Categories in $\mathcal P^{\circ\bullet}_{2,\mathrm{nb}}$ are closed under a certain kind of projection operation which produces brackets as described in the sequel.
\end{minipage}

\begin{definition}
	Given $p,p'\in\mathcal P^{\circ\bullet}_{2,\mathrm{nb}}$ and sectors $S$ in $p$ and $S'$ in $p'$, we number the points in $\inte (S)$ and $\inte (S')$ with respect to the cyclic order. We say that $(p,S)$ and $(p',S')$ are \emph{equivalent} if the following four conditions are met:
	\begin{enumerate}
		\item The sectors $S$ and $S'$ are of equal size.
		\item The same normalized colors occur in the same order in $S$ and $S'$.
		\item For all $i$, the point at position $i$ in $S$ belongs to a block crossing $\partial S$ in $p$ if and only if the block of $i$-th point in $S'$ crosses $\partial S'$ in $p'$. 
		\item For all $i,j$, the points at positions $i,j$ in $S$ form a block in $p$ if and only if the $i$-th and $j$-th points of $S'$ constitute one of the blocks of $p'$.
                \end{enumerate}
                In other words: $p$ restricted to $S$ coincides with $p'$ restricted to $S'$ (up to rotation). 
\end{definition}

\begin{definition}
	Let $S$ be a sector in $p\in\mathcal P^{\circ\bullet}_{2,\mathrm{nb}}$. We refer to the (uniquely determined) bracket $q$ with lower row $M$ which satisfies that $(p,S)$ and $(q,M)$ are equivalent as  \emph{the bracket $B(p,S)$ associated with $(p,S)$}.
      \end{definition}
      We may construct such brackets as described in the proof of the following lemma.

\begin{lemma}
	\label{lemma:associated_brackets}
For all sectors $S$ in $p\in \mathcal  P^{\circ\bullet}_{2,\mathrm{nb}}$ holds $B(p,S)\in\langle p\rangle$.
\end{lemma}
\begin{proof}
	Rotations preserve equivalence. The category $\mathcal{C}$ contains as well the partition $p'$ which results from rotating $p$ in such a way that the image of $S$ comprises exactly the lower row of $p'$. 
\begin{center}
   \hspace{-1em}
  \begin{tikzpicture}[scale=0.666, baseline=0]
    \def\dx{1}
    \def\dy{1}
    \def\dd{0.5}
    \def\cx{0.3}
    \def\cy{0.3}    
    \draw [dotted] ({0*\dx-\dd},{0*\dy}) -- ({6*\dx+\dd},{0*\dy});
    \draw [dotted] ({0*\dx-\dd},{6*\dy}) -- ({6*\dx+\dd},{6*\dy});
    \node[scale=0.4, circle, draw=black, fill=white] (a0) at ({0*\dx},{0*\dy}) {};
    \node[scale=0.4, circle, draw=black, fill=black, thick] (a1) at ({1*\dx},{0*\dy}) {};
    \node[scale=0.4, circle, draw=black, fill=black] (a2) at ({2*\dx},{0*\dy}) {};
    \node[scale=0.4, circle, draw=black, fill=black] (a3) at ({3*\dx},{0*\dy}) {};
    \node[scale=0.4, circle, draw=black, fill=white] (a4) at ({4*\dx},{0*\dy}) {};
    \node[scale=0.4, circle, draw=black, fill=white] (a5) at ({5*\dx},{0*\dy}) {};
    \node[scale=0.4, circle, draw=black, fill=black] (a6) at ({6*\dx},{0*\dy}) {};
    \node[scale=0.4, circle, draw=black, fill=black] (b0) at ({0*\dx},{6*\dy}) {};
    \node[scale=0.4, circle, draw=black, fill=black, thick] (b1) at ({1*\dx},{6*\dy}) {};
    \node[scale=0.4, circle, draw=black, fill=white] (b2) at ({2*\dx},{6*\dy}) {};
    \node[scale=0.4, circle, draw=black, fill=black] (b3) at ({3*\dx},{6*\dy}) {};
    \node[scale=0.4, circle, draw=black, fill=white] (b4) at ({4*\dx},{6*\dy}) {};
    \draw (a0) -- ++ (0,{3*\dy}) -| (a6);
    \draw (a2) -- ++ (0,{2*\dy}) -| (a4);
    \draw (a3) -- ++ (0,{1*\dy}) -| (a5);
    \draw[thick] (a1) -- (b1);
    \draw (b0) -- ++ (0,{-2*\dy}) -| (b4);
    \draw (b2) -- ++ (0,{-1*\dy}) -| (b3);
    \draw[dashed]  ({-\dd},{\cy})-- ({\dx+\cx},{\cy}) --({\dx+\cx},{-\cy})--  ({-\dd},{-\cy});
    \draw[dashed]  ({-\dd},{6*\dy+\cy})-- ({\dx+\cx},{6*\dy+\cy}) --({\dx+\cx},{6*\dy-\cy})--  ({-\dd},{+6*\dy-\cy});
    \node at ({5.5*\dx},{4*\dy}) {$p$};
    \node at ({1.5*\dx},{3.5*\dy}) {$\partial S$};    
    \node (labS) at ({-1*\dx},{3*\dy}) {$S$};
    \draw[densely dotted] (labS.south) -- ({-\dd},{0});
    \draw[densely dotted] (labS.north) -- ({-\dd},{6*\dy});    
  \end{tikzpicture}
  \hspace{1em}
  \begin{tikzpicture}[scale=0.666, baseline=0]
    \def\dx{1}
    \def\dy{1}
    \def\dd{0.5}
    \draw [dotted] ({0*\dx-\dd},{0*\dy}) -- ({7*\dx+\dd},{0*\dy});
    \draw [dotted] ({0*\dx-\dd},{3*\dy}) -- ({7*\dx+\dd},{3*\dy});
    \draw [dotted] ({0*\dx-\dd},{6*\dy}) -- ({7*\dx+\dd},{6*\dy});    
    \node[scale=0.4, circle, draw=black, fill=white] (a1) at ({1*\dx},{0*\dy}) {};
    \node[scale=0.4, circle, draw=black, fill=white] (a2) at ({2*\dx},{0*\dy}) {};
    \node[scale=0.4, circle, draw=black, fill=white] (a3) at ({3*\dx},{0*\dy}) {};
    \node[scale=0.4, circle, draw=black, fill=black] (a4) at ({4*\dx},{0*\dy}) {};
    \node[scale=0.4, circle, draw=black, fill=white] (b0) at ({0*\dx},{3*\dy}) {};
    \node[scale=0.4, circle, draw=black, fill=black] (b1) at ({1*\dx},{3*\dy}) {};
    \node[scale=0.4, circle, draw=black, fill=white] (b2) at ({2*\dx},{3*\dy}) {};
    \node[scale=0.4, circle, draw=black, fill=white] (b3) at ({3*\dx},{3*\dy}) {};
    \node[scale=0.4, circle, draw=black, fill=black] (b4) at ({4*\dx},{3*\dy}) {};
    \node[scale=0.4, circle, draw=black, fill=black] (b5) at ({5*\dx},{3*\dy}) {};
    \node[scale=0.4, circle, draw=black, fill=white] (b6) at ({6*\dx},{3*\dy}) {};
    \node[scale=0.4, circle, draw=black, fill=white] (b7) at ({7*\dx},{3*\dy}) {};
    \node[scale=0.4, circle, draw=black, fill=white] (c1) at ({1*\dx},{6*\dy}) {};
    \node[scale=0.4, circle, draw=black, fill=white] (c2) at ({2*\dx},{6*\dy}) {};
    \node[scale=0.4, circle, draw=black, fill=white] (c3) at ({3*\dx},{6*\dy}) {};
    \node[scale=0.4, circle, draw=black, fill=black] (c4) at ({4*\dx},{6*\dy}) {};    
    \draw (a1) -- ++ (0,{1*\dy}) -| (a4);
    \draw (b0) -- ++ (0,{-1*\dy}) -| (b1);
    \draw (b4) -- ++ (0,{-1*\dy}) -| (b6);
    \draw (b5) -- ++ (0,{-2*\dy}) -| (b7);
    \draw (a2) -- (b2);
    \draw (a3) -- (b3);
    \draw (c1) -- ++ (0,{-1*\dy}) -| (c4);
    \draw (b0) -- ++ (0,{1*\dy}) -| (b1);
    \draw (b4) -- ++ (0,{1*\dy}) -| (b6);
    \draw (b5) -- ++ (0,{2*\dy}) -| (b7);
    \draw (c2) -- (b2);
    \draw (c3) -- (b3);
    \node at ({0.25*\dx},{1*\dy}) {$p'$};
    \node at ({0.25*\dx},{5*\dy}) {$(p')^*$};        
  \end{tikzpicture}
  \hspace{1em}
  \begin{tikzpicture}[scale=0.666, baseline=-1cm]
    \def\dx{1}
    \def\dy{1}
    \def\dd{0.5}
    \draw [dotted] ({0*\dx-\dd},{0*\dy}) -- ({3*\dx+\dd},{0*\dy});
    \draw [dotted] ({0*\dx-\dd},{3*\dy}) -- ({3*\dx+\dd},{3*\dy});
    \node[scale=0.4, circle, draw=black, fill=white] (a1) at ({0*\dx},{0*\dy}) {};
    \node[scale=0.4, circle, draw=black, fill=white] (a2) at ({1*\dx},{0*\dy}) {};
    \node[scale=0.4, circle, draw=black, fill=white] (a3) at ({2*\dx},{0*\dy}) {};
    \node[scale=0.4, circle, draw=black, fill=black] (a4) at ({3*\dx},{0*\dy}) {};
    \node[scale=0.4, circle, draw=black, fill=white] (c1) at ({0*\dx},{3*\dy}) {};
    \node[scale=0.4, circle, draw=black, fill=white] (c2) at ({1*\dx},{3*\dy}) {};
    \node[scale=0.4, circle, draw=black, fill=white] (c3) at ({2*\dx},{3*\dy}) {};
    \node[scale=0.4, circle, draw=black, fill=black] (c4) at ({3*\dx},{3*\dy}) {};    
    \draw (a1) -- ++ (0,{1*\dy}) -| (a4);
    \draw (c1) -- ++ (0,{-1*\dy}) -| (c4);
    \draw (c2) -- (a2);
    \draw (c3) -- (a3);
    \node at ({1.5*\dx},{-1*\dy}) {$B(p,S)$};
  \end{tikzpicture}
\end{center}
By a straightforward generalization of \cite[Theorem~2.12]{FrWe14} to two-colored partitions, the partition $p' \left(p'\right)^*\in\langle p\rangle$ is projective. It is already the bracket $B(p,S)$.
      \end{proof}
Brackets can act via composition on suitable partitions to undo the erasing (see Sections~\ref{subsection:operations} and~\ref{subsection_neutral_subsets_and_the_color_sum}) of certain sets of points. \par
\vspace{0.75em}
\noindent
\begin{minipage}[c]{0.62\textwidth}
\begin{definition}
	\begin{enumerate}
		\item A neutral set of two consecutive points (with respect to the cyclic order) is called a \emph{turn}. 
		\item A block $B$ in a partition $p\in\mathcal P^{\circ\bullet}_{2,\mathrm{nb}}$ is a \emph{turn block} of a turn $T$ in $p$ if $B\neq T$ and $B\cap T\neq \emptyset$.	
	\end{enumerate}
\end{definition}
\end{minipage}
\hfill
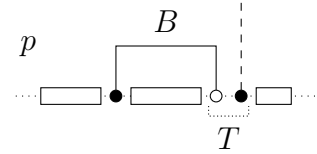
\begin{minipage}{5cm}
  \begin{tikzpicture}[scale=0.666]
    \draw [dotted] (-0.5,0) -- (5.5,0);
    \node[scale=0.4, circle,draw=black,fill=black] (x1) at (1.5,0) {};
    \node[scale=0.4, circle,draw=black,fill=white] (x2) at (3.5,0) {};
    \node[scale=0.4, circle,draw=black,fill=black] (x3) at (4,0) {};
    \draw (x2) -- ++(0,1) -| (x1);
    \draw[fill=white] (0,-0.166) rectangle (1.2,0.166);
    \draw[fill=white] (1.8,-0.166) rectangle (3.2,0.166);
    \draw[fill=white] (4.3,-0.166) rectangle (5,0.166);
    \draw[densely dotted] (3.35,-0.2) -- ++(0,-0.2) -| (4.15,-0.2);
    \draw[dashed] (x3) -- ++(0,2);
    \node at (3.75,-0.85) {$T$};
    \node at (2.5,1.5) {$B$};
    \node at (-0.25,1) {$p$};        
  \end{tikzpicture}
\end{minipage}
\par
\vspace{0.5em}
Recall from Section~\ref{subsection:operations} that, for all turns $T$ in $p\in \Cppnb$, the erasing of $T$ from $p$ is denoted by $E(p,T)$.
\begin{lemma}
  \label{lemma:reversing_erasing_of_turns}
  Let $S$ be a sector in $p\in \mathcal P^{\circ\bullet}_{2,\mathrm{nb}}$ such that $\partial S$ is a turn block of a turn $T$ in $p$. Then, \[\left\langle p\right\rangle=\left\langle E(p,T),B(p,S)\right\rangle.\]
\end{lemma}
\begin{proof}
  Since $\langle p\rangle$ is closed under erasing turns, Lemma \ref{lemma:associated_brackets} proves one direction of the claim. Conversely, start with the generators $E(p,T)$ and $B(p,S)$. As $T$ has a turn block, $p$ cannot consists of just two points. Moreover, we can assume that $p$ is rotated in such a way that it has no upper points and that $T$ consists of the first and the last point of the lower row of $p$. We only treat the case that the first point of the lower row of $p$ is black. The other case is handled in an analogous manner. The upper row of the partition $p'\eqpd(\PartIdenLoWB\otimes E(p,T))^{\circlearrowright}$, an element of $\left\langle E(p,T),B(p,S)\right\rangle$, has the same number of points and the same coloration as the lower row of $p$.\par
  \noindent
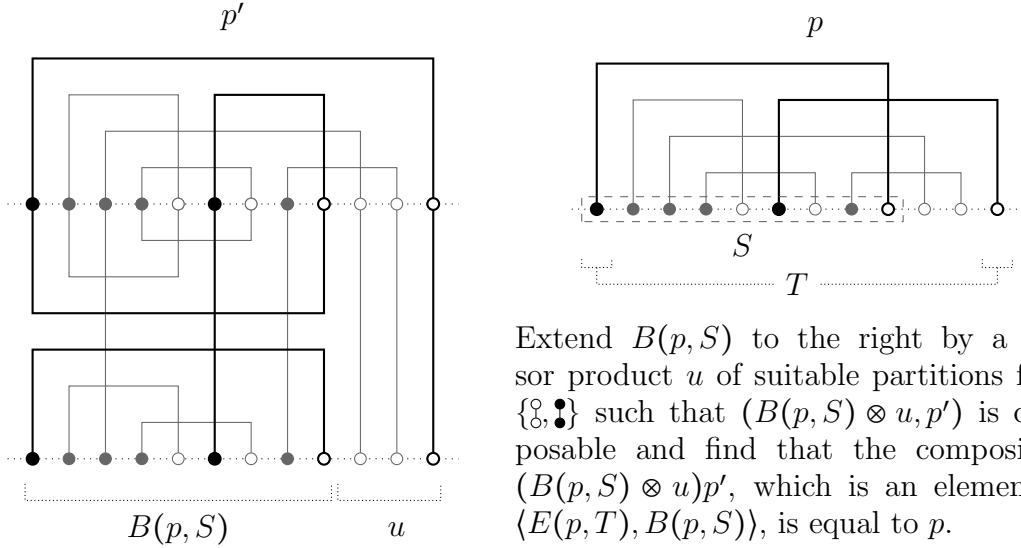
  \begin{minipage}{0.5\textwidth}
    \begin{center}
  \begin{tikzpicture}[scale=0.666, baseline=0]
    \def\dist{0.727}
    \def\hgt{0.727}
    \def\rwd{7*\hgt}
    \def\dofs{0.5}
    \def\cxoffs{0.3}
    \def\cyoffs{0.166}    
    \draw[dotted] ({-\dofs},0) -- ({8+\dofs},0);
    \draw[dotted] ({-\dofs},{-\rwd}) -- ({8+\dofs},{-\rwd});
          \node[circle, fill=black, draw=black, scale=0.4, thick] (t1) at ({0*\dist},0) {};
          \node[circle, fill=black, draw=black, scale=0.4, thick] (p2) at ({5*\dist},0) {};
          \node[circle, fill=white, draw=black, scale=0.4, thick] (p1) at ({8*\dist},0) {};
          \node[circle, fill=white, draw=black, scale=0.4, thick] (t2) at ({11*\dist},0) {};
          \node[circle, fill=black, draw=black, scale=0.4, thick] (bt1) at ({0*\dist},{-\rwd}) {};
          \node[circle, fill=black, draw=black, scale=0.4, thick] (bp2) at ({5*\dist},{-\rwd}) {};
          \node[circle, fill=white, draw=black, scale=0.4, thick] (bp1) at ({8*\dist},{-\rwd}) {};
          \node[circle, fill=white, draw=black, scale=0.4, thick] (bt2) at ({11*\dist},{-\rwd}) {};
          \node[circle, scale=0.4, draw=gray, fill=gray] (a1-1) at ({\dist},0) {};
          \node[circle, scale=0.4, draw=gray, fill=white] (a1-2) at ({4*\dist},0) {};
          \draw[gray] (a1-1) -- ++ (0,{3*\hgt}) -| (a1-2);
          \node[circle, scale=0.4, draw=gray, fill=gray] (a2-1) at ({2*\dist},0) {};
          \node[circle, scale=0.4, draw=gray, fill=white] (a2-2) at ({9*\dist},0) {};
          \draw[gray] (a2-1) -- ++ (0,{2*\hgt}) -| (a2-2);
          \node[circle, scale=0.4, draw=gray, fill=gray] (a3-1) at ({3*\dist},0) {};
          \node[circle, scale=0.4, draw=gray, fill=white] (a3-2) at ({6*\dist},0) {};
          \draw[gray] (a3-1) -- ++ (0,{\hgt}) -| (a3-2);
          \node[circle, scale=0.4, draw=gray, fill=gray] (a4-1) at ({7*\dist},0) {};
          \node[circle, scale=0.4, draw=gray, fill=white] (a4-2) at ({10*\dist},0) {};
          \draw[gray] (a4-1) -- ++ (0,{\hgt}) -| (a4-2);

          \node[circle, scale=0.4, draw=gray, fill=gray] (b1-1) at ({\dist},0) {};
          \node[circle, scale=0.4, draw=gray, fill=white] (b1-2) at ({4*\dist},0) {};
          \draw[gray] (b1-1) -- ++ (0,{-2*\hgt}) -| (b1-2);
          \node[circle, scale=0.4, draw=gray, fill=gray] (b6-1) at ({\dist},{-\rwd}) {};
          \node[circle, scale=0.4, draw=gray, fill=white] (b6-2) at ({4*\dist},{-\rwd}) {};
          \draw[gray] (b6-1) -- ++ (0,{2*\hgt}) -| (b6-2);          
          \node[circle, scale=0.4, draw=gray, fill=gray] (b2-1) at ({2*\dist},0) {};
          \node[circle, scale=0.4, draw=gray, fill=gray] (b2-2) at ({2*\dist},{-\rwd}) {};
          \draw[gray] (b2-1) to (b2-2);
          \node[circle, scale=0.4, draw=gray, fill=gray] (b3-1) at ({3*\dist},0) {};
          \node[circle, scale=0.4, draw=gray, fill=white] (b3-2) at ({6*\dist},0) {};
          \draw[gray] (b3-1) -- ++ (0,{-\hgt}) -| (b3-2);
          \node[circle, scale=0.4, draw=gray, fill=gray] (b5-1) at ({3*\dist},{-\rwd}) {};
          \node[circle, scale=0.4, draw=gray, fill=white] (b5-2) at ({6*\dist},{-\rwd}) {};
          \draw[gray] (b5-1) -- ++ (0,{\hgt}) -| (b5-2);          
          \node[circle, scale=0.4, draw=gray, fill=gray] (b4-1) at ({7*\dist},0) {};
          \node[circle, scale=0.4, draw=gray, fill=gray] (b4-2) at ({7*\dist},{-\rwd}) {};
          \draw[gray] (b4-1) to (b4-2);                                        
          \node[circle, scale=0.4, draw=gray, fill=white] (b7) at ({9*\dist},{-\rwd}) {};
          \draw[gray] (a2-2) to (b7);
          \node[circle, scale=0.4, draw=gray, fill=white] (b8) at ({10*\dist},{-\rwd}) {};
          \draw[gray] (a4-2) to (b8);
          
          \draw[thick] (t1) -- ++(0,{4*\hgt}) -| (t2);
          \draw[thick] (p2) -- ++(0,{3*\hgt}) -| (p1);
          \draw[thick] (t1) -- ++(0,{-3*\hgt}) -| (p1);
          \draw[thick] (bt1) -- ++(0,{3*\hgt}) -| (bp1);
          \draw[thick] (bp2) to (p2);
          \draw[thick] (bt2) to (t2);

          \node at (4.,3.75) {$p'$};
          \draw[densely dotted] ({-0.5*\cxoffs},{-\rwd-0.5}) -- ++(0,{-2*\cyoffs}) -| ({8*\dist+0.5*\cxoffs},{-\rwd-0.5});
          \draw[densely dotted] ({9*\dist-1.5*\cxoffs},{-\rwd-0.5}) -- ++(0,{-2*\cyoffs}) -| ({11*\dist+0.5*\cxoffs},{-\rwd-0.5});          
          \node at ({4*\dist},{-\rwd-1.4}) {$B(p,S)$};
          \node at ({10*\dist},{-\rwd-1.4}) {$u$};
        \end{tikzpicture}
        \end{center}
        \end{minipage}
        \begin{minipage}{0.5\textwidth}
          \begin{center}
 \begin{tikzpicture}[scale=0.666, baseline =0]
 \def\dist{0.727}
    \def\hgt{0.727}
    \def\rwd{7*\hgt}
    \def\dofs{0.5}
    \def\cxoffs{0.3}
    \def\cyoffs{0.166}    
    \draw[dotted] ({-\dofs},0) -- ({8+\dofs},0);
          
          \node[circle, fill=black, draw=black, scale=0.4, thick] (t1) at ({0*\dist},0) {};
          \node[circle, fill=black, draw=black, scale=0.4, thick] (p2) at ({5*\dist},0) {};
          \node[circle, fill=white, draw=black, scale=0.4, thick] (p1) at ({8*\dist},0) {};
          \node[circle, fill=white, draw=black, scale=0.4, thick] (t2) at ({11*\dist},0) {};
          
          \node[circle, scale=0.4, draw=gray, fill=gray] (a1-1) at ({\dist},0) {};
          \node[circle, scale=0.4, draw=gray, fill=white] (a1-2) at ({4*\dist},0) {};
          \draw[gray] (a1-1) -- ++ (0,{3*\hgt}) -| (a1-2);
          \node[circle, scale=0.4, draw=gray, fill=gray] (a2-1) at ({2*\dist},0) {};
          \node[circle, scale=0.4, draw=gray, fill=white] (a2-2) at ({9*\dist},0) {};
          \draw[gray] (a2-1) -- ++ (0,{2*\hgt}) -| (a2-2);
          \node[circle, scale=0.4, draw=gray, fill=gray] (a3-1) at ({3*\dist},0) {};
          \node[circle, scale=0.4, draw=gray, fill=white] (a3-2) at ({6*\dist},0) {};
          \draw[gray] (a3-1) -- ++ (0,{\hgt}) -| (a3-2);
          \node[circle, scale=0.4, draw=gray, fill=gray] (a4-1) at ({7*\dist},0) {};
          \node[circle, scale=0.4, draw=gray, fill=white] (a4-2) at ({10*\dist},0) {};
          \draw[gray] (a4-1) -- ++ (0,{\hgt}) -| (a4-2);

          \draw[thick] (p2) -- ++ (0,{3*\hgt}) -| (t2);
          \draw[thick] (t1) -- ++ (0,{4*\hgt}) -| (p1);
          
          \node at ({6*\dist},{5*\hgt}) {$p$};

          \node at ({4*\dist},{-1*\hgt}) {$S$};
          \draw[gray, dashed] ({0*\dist-\cxoffs},{-1.5*\cyoffs}) rectangle ({8*\dist+\cxoffs},{1.5*\cyoffs});

          \node (x) at ({5.5*\dist},-1.5) {$T$};
          \draw[densely dotted] ({0*\dist-\cxoffs},-1) -- ++(0,{-\cyoffs}) -| ({\cxoffs},-1);
 \draw[densely dotted] ({11*\dist-\cxoffs},-1) -- ++(0,{-\cyoffs}) -| ({11*\dist+\cxoffs},-1);
 \draw[densely dotted] ({0*\dist},{-1-\cyoffs}) |- (x.west);
 \draw[densely dotted] ({11*\dist},{-1-\cyoffs}) |- (x.east);         
	\end{tikzpicture}
\end{center}
        \vspace{0em}
        Extend $B(p,S)$ to the right by a tensor product $u$ of suitable partitions from $\{\PartIdenW,\PartIdenB\}$ such that $(B(p,S)\otimes u,p')$ is composable and find that the composition $(B(p,S)\otimes u)p'$, which is an element of $\left\langle E(p,T),B(p,S)\right\rangle$, is equal to $p$.\qedhere
        \end{minipage}
      \end{proof}

Besides by rotations, it is by this operation of \enquote{de-erasing} turns alone that the set $\mc B$ generates all partitions of $\Cppnb$, as shown in the following. 
\begin{lemma}
  \label{lemma:generation_of_categories_by_brackets} 
  For every category $\mathcal C\subseteq \mathcal P^{\circ\bullet}_{2,\mathrm{nb}}$ holds $\mc C=\langle \mc C\cap \mc B\rangle$.  Moreover, for all $n\in \mathbb N$, the brackets in $\mc C\cap \mc B$ with at most $2n$ points are sufficient to generate all partitions of $\mathcal C$ with at most $n+1$ points.
\end{lemma}
\begin{proof}
For all $n\in \mathbb N$, denote by $\mathcal B_n$ the set of all brackets with at most $2n$ points. It was shown in \cite[Proposition~3.3~(a)]{TaWe15a} that $\Cppnb\cap \mc{NC}^{\circ\bullet}=\langle \emptyset\rangle$, implying $\mc C\cap \mc{NC}^{\circ\bullet}\subseteq \langle \mc C\cap \mc B_n\rangle$ for all $n\in \pint$. Especially, all partitions of $\mc C$ with at most $2$ points are contained in  $\langle \mc C\cap\mc B_1\rangle$ as they are all non-crossing. Let $n\in \mathbb N$ with $n\geq 2$ be arbitrary, suppose that $\left\langle \mc C\cap \mc B_{n-1}\right\rangle$ contains all partitions of $\mathcal C$ with at most $n$ points and let $p\in \mathcal C$ have $n+1$ points. We can assume $p\notin \mc{NC}^{\circ\bullet}$. Then, we can find a sector $S$ of a turn block $\partial S$ of a turn $T$ in $p$. The partition $B(p,S)$, an element of $\mathcal C$ by Lemma~\ref{lemma:associated_brackets}, is a bracket with at most $2n$ points and thus an element of $\mc C\cap \mc B_n$ by definition of $\mc B_n$. The partition $E(p,T)$ has $n-1$, so no more than $n$, points and is hence contained in $\left\langle \mc C\cap \mathcal B_{n-1}\right\rangle$ by the induction hypothesis. As $\mc C\cap \mathcal B_{n-1}\subseteq \mc C\cap \mathcal B_{n}$, we infer $E(p,T)\in \langle \mc C\cap \mc B_n\rangle$. Having seen $E(p,T),B(p,S)\in \langle \mc C\cap\mc B_n\rangle$, Lemma~\ref{lemma:reversing_erasing_of_turns} now proves $p\in \langle \mc C\cap \mc B_n\rangle$ and thus the claim.
\end{proof}

\subsection{Residual Brackets}
\label{subsection:residual_brackets}
While Lemma~\ref{lemma:generation_of_categories_by_brackets} provides a first tool for understanding subcategories of $\Cppnb$, the set $\mc B$ of all brackets is still too large to be tractable. In this subsection, the subset $\resbr$ of $\mc B$ is defined and shown to also satisfy $\mc C=\langle \mc C\cap \resbr\rangle$ for all categories $\mc C\subseteq \Cppnb$. In several steps (leading up to Definition~\ref{definition:residual_bracket}), we introduce the defining properties individually and note their respective significance. A first demand we can make on our universal generators is that they be connected (see Section~\ref{subsection:connectedness}).
Connectedness is a modest assumption in  $\mathcal P^{\circ\bullet}_{2,\mathrm{nb}}$ as the following result shows. Recall that the factor partition of a connected component $S$ of $p\in\Cp$ is $E(p,S^c)$, where $S^c$ denotes the complement of $S$.
\begin{lemma}  \label{lemma:generation_of_categories_by_connected_partitions}  
  A category contains a partition $p\in \mathcal P^{\circ\bullet}_{2,nb}$ if and only if it contains all the factor partitions of the connected components of $p$.
\end{lemma}
\begin{proof}
  	 If $p$ is not connected, there must exist a connected component $S$ of $p$ which is an interval with respect to the cyclic order. Since $p$ has neutral blocks, $S$ is a neutral set. Thus, we can erase it and conclude $E(p,S)\in\langle p\rangle$. Likewise, we can erase the complement $S^c$ of $S$ and find $E(p,S^c)\in\langle p\rangle$. The partition $E(p,S)$ is the factor partition of $S$, and $E(p,S^c)$ has one fewer connected component than $p$. We repeat the procedure until we are left with connected partitions only. These elements of $\langle p\rangle$ are rotations of the factor partitions of the connected components of the original partition $p$. That proves one direction of the claim. Because  $p$ can be reassembled by appropriate tensor products and rotations from the factor partitions of its connected components, the reverse holds as well. 
       \end{proof}
       By definition, brackets are already required to be fixed points of involution and composition with themselves. We can add further such symmetry conditions.
\begin{definition}
 A partition $p\in \Cp$ is called \emph{verticolor-reflexive} if $\tilde p=p$.
\end{definition}
Verticolor-reflexive partitions necessarily have an even number of points in every one of their rows.
\begin{definition}
   We refer to a bracket $p$ with lower row $S$ as \emph{dualizable} if $p$ is ver\-ti\-co\-lor-re\-flexive, if $\inte(S)$ is non-empty, and if the two middle points of $\inte(S)$ form a turn in $p$ with turn blocks both of which cross $\partial S$.
\end{definition}
Performing a quarter rotation on a dualizable bracket gives the same partition for both directions, and this partition is a bracket as well. 
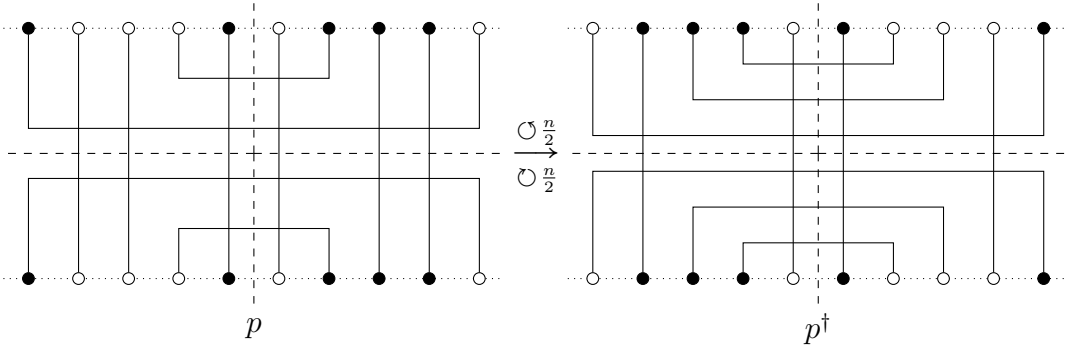
\begin{gather*}
  \begin{tikzpicture}[scale=0.666,baseline=1.56cm]
    \draw[dotted] (-0.5,0) -- (9.5,0);
    \draw[dotted] (-0.5,5) -- (9.5,5);
    \node[circle,scale=0.4,draw=black,fill=black] (x1) at (0,0) {};
    \node[circle,scale=0.4,draw=black,fill=white] (x2) at (1,0) {};
    \node[circle,scale=0.4,draw=black,fill=white] (x3) at (2,0) {};
    \node[circle,scale=0.4,draw=black,fill=white] (x4) at (3,0) {};    
    \node[circle,scale=0.4,draw=black,fill=black] (w1) at (4,0) {};
    \node[circle,scale=0.4,draw=black,fill=white] (w2) at (5,0) {};
    \node[circle,scale=0.4,draw=black,fill=black] (x5) at (6,0) {};
    \node[circle,scale=0.4,draw=black,fill=black] (x6) at (7,0) {};
    \node[circle,scale=0.4,draw=black,fill=black] (x7) at (8,0) {};
    \node[circle,scale=0.4,draw=black,fill=white] (x8) at (9,0) {};    
    \node[circle,scale=0.4,draw=black,fill=black] (y1) at (0,5) {};
    \node[circle,scale=0.4,draw=black,fill=white] (y2) at (1,5) {};
    \node[circle,scale=0.4,draw=black,fill=white] (y3) at (2,5) {};
    \node[circle,scale=0.4,draw=black,fill=white] (y4) at (3,5) {};    
    \node[circle,scale=0.4,draw=black,fill=black] (z1) at (4,5) {};
    \node[circle,scale=0.4,draw=black,fill=white] (z2) at (5,5) {};
    \node[circle,scale=0.4,draw=black,fill=black] (y5) at (6,5) {};
    \node[circle,scale=0.4,draw=black,fill=black] (y6) at (7,5) {};
    \node[circle,scale=0.4,draw=black,fill=black] (y7) at (8,5) {};
    \node[circle,scale=0.4,draw=black,fill=white] (y8) at (9,5) {};    
    \draw (x1) -- ++(0,2) -| (x8);
    \draw (y1) -- ++(0,-2) -| (y8);
    \draw (x4) -- ++(0,1) -| (x5);
    \draw (y4) -- ++(0,-1) -| (y5);    
    \draw (x2) to (y2);
    \draw (x3) to (y3);
    \draw (x6) to (y6);
    \draw (x7) to (y7);        
    \draw (w1) to (z1);
    \draw (w2) to (z2);
    \node at (4.5,-1) {$p$};    
    \draw [dashed] (4.5,2.5) -- (4.5,5.5);
    \draw [dashed] (4.5,2.5) -- (4.5,-0.5);    
    \draw [dashed] (4.5,2.5) -- (-0.5,2.5);
    \draw [dashed] (4.5,2.5) -- (9.5,2.5);        
  \end{tikzpicture}
\underset{\circlearrowright\frac{n}{2}}{\overset{\circlearrowleft\frac{n}{2}}{\longrightarrow}}
  \begin{tikzpicture}[scale=0.666,baseline=1.56cm]
    \draw[dotted] (-0.5,0) -- (9.5,0);
    \draw[dotted] (-0.5,5) -- (9.5,5);
    \node[circle,scale=0.4,draw=black,fill=white] (x1) at (0,0) {};
    \node[circle,scale=0.4,draw=black,fill=black] (x2) at (1,0) {};
    \node[circle,scale=0.4,draw=black,fill=black] (x3) at (2,0) {};
    \node[circle,scale=0.4,draw=black,fill=black] (x4) at (3,0) {};    
    \node[circle,scale=0.4,draw=black,fill=white] (w1) at (4,0) {};
    \node[circle,scale=0.4,draw=black,fill=black] (w2) at (5,0) {};
    \node[circle,scale=0.4,draw=black,fill=white] (x5) at (6,0) {};
    \node[circle,scale=0.4,draw=black,fill=white] (x6) at (7,0) {};
    \node[circle,scale=0.4,draw=black,fill=white] (x7) at (8,0) {};
    \node[circle,scale=0.4,draw=black,fill=black] (x8) at (9,0) {};    
    \node[circle,scale=0.4,draw=black,fill=white] (y1) at (0,5) {};
    \node[circle,scale=0.4,draw=black,fill=black] (y2) at (1,5) {};
    \node[circle,scale=0.4,draw=black,fill=black] (y3) at (2,5) {};
    \node[circle,scale=0.4,draw=black,fill=black] (y4) at (3,5) {};    
    \node[circle,scale=0.4,draw=black,fill=white] (z1) at (4,5) {};
    \node[circle,scale=0.4,draw=black,fill=black] (z2) at (5,5) {};
    \node[circle,scale=0.4,draw=black,fill=white] (y5) at (6,5) {};
    \node[circle,scale=0.4,draw=black,fill=white] (y6) at (7,5) {};
    \node[circle,scale=0.4,draw=black,fill=white] (y7) at (8,5) {};
    \node[circle,scale=0.4,draw=black,fill=black] (y8) at (9,5) {};    
    \draw (x1) -- ++(0,2.143) -| (x8); 
    \draw (y1) -- ++(0,-2.143) -| (y8);
    \draw (x2) to (y2);
    \draw (x7) to (y7);
    \draw (x3) -- ++(0,1.428) -| (x6);
    \draw (x4) -- ++(0,0.714) -| (x5);
    \draw (y3) -- ++(0,-1.428) -| (y6);
    \draw (y4) -- ++(0,-0.714) -| (y5);        
    \draw (w1) to (z1);
    \draw (w2) to (z2);
    \node at (4.5,-1) {$p^\dagger$};        
    \draw [dashed] (4.5,2.5) -- (4.5,5.5);
    \draw [dashed] (4.5,2.5) -- (4.5,-0.5);    
    \draw [dashed] (4.5,2.5) -- (-0.5,2.5);
    \draw [dashed] (4.5,2.5) -- (9.5,2.5);            
  \end{tikzpicture}
\end{gather*}

\begin{definition}
  For a dualizable bracket $p$ with $n$ points in its lower row, we call the bracket $p^\dagger\eqpd p^{\circlearrowleft\frac{n}{2}}=p^{\circlearrowright\frac{n}{2}}$ the \emph{dual bracket} of $p$.
\end{definition}
The dual $p^\dagger$ of a dualizable bracket $p\in \Cppnb$ is a dualizable bracket as well and it holds $(p^\dagger)^\dagger=p$ and $\langle p\rangle=\langle p^\dagger\rangle$. Combining connectedness and dualizability, we are now able to give the definition of residual brackets.
\begin{definition}
  \label{definition:residual_bracket}
  \begin{enumerate}[label=(\alph*)]
  \item \label{item:residual_bracket-item_1} 
    Let  $p$ be a bracket with lower row $S$.
    \begin{enumerate}
    \item \label{item:residual_bracket-item_1-item_1} We call $p$ \emph{residual of the first kind} if $p$ is connected and if $\inte(S)$ contains no turns of $p$.
    \item \label{item:residual_bracket-item_1-item_2} We call $p$ \emph{residual of the second kind} if $p$ is connected and dualizable and if $\inte(S)$ contains exactly one turn of $p$.
    \item \label{item:residual_bracket-item_1-item_3} We call $p$ \emph{residual} if $p$ is residual of the first or the second kind.
    \end{enumerate}
  \item The set of all residual brackets is denoted by $\mathcal{B}_{\mathrm{res}}$.
  \end{enumerate}
\end{definition}
To reduce the set of brackets $\mc B$ to its subset $\resbr$, we decrease the number of turns occurring in a given bracket using the reversible category operation from Lemma~\ref{lemma:reversing_erasing_of_turns}. Residual brackets of the first kind represents those brackets where this reduction is possible until no turns remain. Residual brackets of the second kind, in contrast, arise as the set of brackets reproducing itself under this operation. How this works in detail is seen in the proof of the following enhancement of Lemma~\ref{lemma:generation_of_categories_by_brackets} and central result of this subsection.
\begin{proposition}
  \label{proposition:generation_of_categories_by_residual_brackets}
  For every category $\mathcal C\subseteq \mathcal P^{\circ\bullet}_{2,\mathrm{nb}}$ holds $\mc C=\left\langle \mc C\cap \resbr\right\rangle$.
\end{proposition}
\begin{proof}
  A bracket has necessarily at least $4$ points. We show inductively that, for all $n\in \mathbb N$ with $n\geq 2$, the set of brackets of $\mathcal C$ with at most $2n$ points is contained in $\left\langle\mc C\cap \resbr\right\rangle$, which by Lemma~\ref{lemma:generation_of_categories_by_brackets} is sufficient to prove the claim. The set $\mc C\cap \resbr$ generates all brackets of size $4$ as these are non-crossing and as $\langle \emptyset\rangle=\Cppnb\cap \mc{NC}^{\circ\bullet}$. Let $n\in \mathbb N$ with $n\geq 3$ be arbitrary, let all brackets of $\mathcal C$ with at most $2(n-1)$ points be elements of $\langle\mc C\cap \resbr\rangle$, and let $p$ be a bracket of $\mathcal C$ with $2n$ points. We distinguish three cases.
  \par
  \emph{Case 1:} First, assume that $p$ is not connected. If $p$ had exactly two connected components, then $p$ would have $4$ points. So, $p$ has at least three connected components.
  
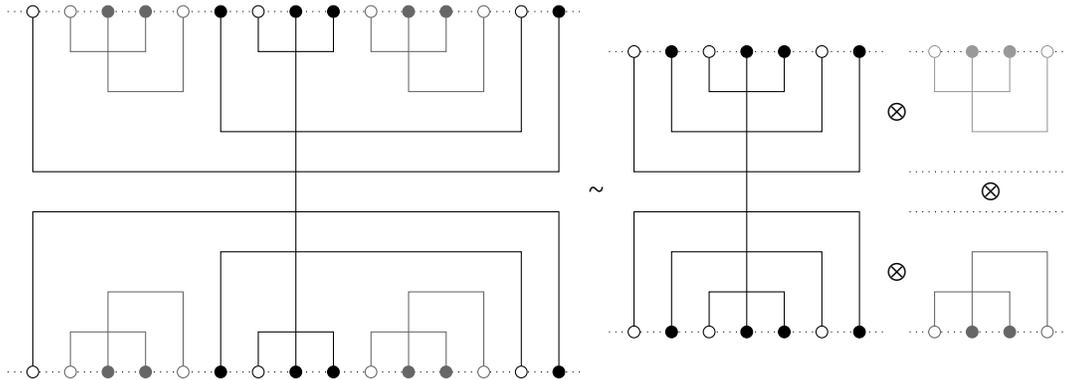
\begin{center}
    \begin{tikzpicture}[scale=0.666,baseline=0]
      \def\dx{0.75}
      \def\dy{0.8}
      \def\dd{0.5}
      \def\tx{14*\dx}
      \def\ty{9*\dy}      
      \draw[dotted]({-\dd},0) -- ({\tx+\dd},0);
      \draw[dotted]({-\dd},{\ty}) -- ({\tx+\dd},{\ty});
      \node[scale=0.4,circle,draw=black,fill=white] (a0) at ({0*\dx},{0}) {};
      \node[scale=0.4,circle,draw=gray,fill=white] (a1) at ({1*\dx},{0}) {};
      \node[scale=0.4,circle,draw=gray,fill=gray] (a2) at ({2*\dx},{0}) {};
      \node[scale=0.4,circle,draw=gray,fill=gray] (a3) at ({3*\dx},{0}) {};
      \node[scale=0.4,circle,draw=gray,fill=white] (a4) at ({4*\dx},{0}) {};
      \node[scale=0.4,circle,draw=black,fill=black] (a5) at ({5*\dx},{0}) {};
      \node[scale=0.4,circle,draw=black,fill=white] (a6) at ({6*\dx},{0}) {};
      \node[scale=0.4,circle,draw=black,fill=black] (a7) at ({7*\dx},{0}) {};
      \node[scale=0.4,circle,draw=black,fill=black] (a8) at ({8*\dx},{0}) {};
      \node[scale=0.4,circle,draw=gray,fill=white] (a9) at ({9*\dx},{0}) {};
      \node[scale=0.4,circle,draw=gray,fill=gray] (a10) at ({10*\dx},{0}) {};
      \node[scale=0.4,circle,draw=gray,fill=gray] (a11) at ({11*\dx},{0}) {};
      \node[scale=0.4,circle,draw=gray,fill=white] (a12) at ({12*\dx},{0}) {};
      \node[scale=0.4,circle,draw=black,fill=white] (a13) at ({13*\dx},{0}) {};
      \node[scale=0.4,circle,draw=black,fill=black] (a14) at ({14*\dx},{0}) {};
      \node[scale=0.4,circle,draw=black,fill=white] (b0) at ({0*\dx},{\ty}) {};
      \node[scale=0.4,circle,draw=gray,fill=white] (b1) at ({1*\dx},{\ty}) {};
      \node[scale=0.4,circle,draw=gray,fill=gray] (b2) at ({2*\dx},{\ty}) {};
      \node[scale=0.4,circle,draw=gray,fill=gray] (b3) at ({3*\dx},{\ty}) {};
      \node[scale=0.4,circle,draw=gray,fill=white] (b4) at ({4*\dx},{\ty}) {};
      \node[scale=0.4,circle,draw=black,fill=black] (b5) at ({5*\dx},{\ty}) {};
      \node[scale=0.4,circle,draw=black,fill=white] (b6) at ({6*\dx},{\ty}) {};
      \node[scale=0.4,circle,draw=black,fill=black] (b7) at ({7*\dx},{\ty}) {};
      \node[scale=0.4,circle,draw=black,fill=black] (b8) at ({8*\dx},{\ty}) {};
      \node[scale=0.4,circle,draw=gray,fill=white] (b9) at ({9*\dx},{\ty}) {};
      \node[scale=0.4,circle,draw=gray,fill=gray] (b10) at ({10*\dx},{\ty}) {};
      \node[scale=0.4,circle,draw=gray,fill=gray] (b11) at ({11*\dx},{\ty}) {};
      \node[scale=0.4,circle,draw=gray,fill=white] (b12) at ({12*\dx},{\ty}) {};
      \node[scale=0.4,circle,draw=black,fill=white] (b13) at ({13*\dx},{\ty}) {};
      \node[scale=0.4,circle,draw=black,fill=black] (b14) at ({14*\dx},{\ty}) {};
      \draw (a0) --++(0,{4*\dy})-| (a14);
      \draw (b0) --++(0,{-4*\dy})-| (b14);

      \draw[gray] (a1) --++(0,{1*\dy})-| (a3);
      \draw[gray] (a2) --++(0,{2*\dy})-| (a4);
      \draw (a5) --++(0,{3*\dy})-| (a13);
      \draw (a6) --++(0,{1*\dy})-| (a8);
      \draw[gray] (a9) --++(0,{1*\dy})-| (a11);
      \draw[gray] (a10) --++(0,{2*\dy})-| (a12);

      \draw[gray] (b1) --++(0,{-1*\dy})-| (b3);
      \draw[gray] (b2) --++(0,{-2*\dy})-| (b4);
      \draw (b5) --++(0,{-3*\dy})-| (b13);
      \draw (b6) --++(0,{-1*\dy})-| (b8);
      \draw[gray] (b9) --++(0,{-1*\dy})-| (b11);
      \draw[gray] (b10) --++(0,{-2*\dy})-| (b12);

      \draw (a7) -- (b7);
      \node at (11.25,3.6) {$\sim$};
    \begin{scope}[xshift=12cm,yshift=0.8cm]
      \def\dx{0.75}
      \def\dy{0.8}
      \def\dd{0.5}
      \def\txxx{6*\dx}
      \def\tyyy{7*\dy}      
      \draw[dotted]({-\dd},0) -- ({\txxx+\dd},0);
      \draw[dotted]({-\dd},{\tyyy}) -- ({\txxx+\dd},{\tyyy});
      \node[scale=0.4,circle,draw=black,fill=white] (a0) at ({0*\dx},{0}) {};
      \node[scale=0.4,circle,draw=black,fill=black] (a5) at ({1*\dx},{0}) {};
      \node[scale=0.4,circle,draw=black,fill=white] (a6) at ({2*\dx},{0}) {};
      \node[scale=0.4,circle,draw=black,fill=black] (a7) at ({3*\dx},{0}) {};
      \node[scale=0.4,circle,draw=black,fill=black] (a8) at ({4*\dx},{0}) {};
      \node[scale=0.4,circle,draw=black,fill=white] (a13) at ({5*\dx},{0}) {};
      \node[scale=0.4,circle,draw=black,fill=black] (a14) at ({6*\dx},{0}) {};
      \node[scale=0.4,circle,draw=black,fill=white] (b0) at ({0*\dx},{\tyyy}) {};
      \node[scale=0.4,circle,draw=black,fill=black] (b5) at ({1*\dx},{\tyyy}) {};
      \node[scale=0.4,circle,draw=black,fill=white] (b6) at ({2*\dx},{\tyyy}) {};
      \node[scale=0.4,circle,draw=black,fill=black] (b7) at ({3*\dx},{\tyyy}) {};
      \node[scale=0.4,circle,draw=black,fill=black] (b8) at ({4*\dx},{\tyyy}) {};
      \node[scale=0.4,circle,draw=black,fill=white] (b13) at ({5*\dx},{\tyyy}) {};
      \node[scale=0.4,circle,draw=black,fill=black] (b14) at ({6*\dx},{\tyyy}) {};
      \draw (a0) --++(0,{3*\dy})-| (a14);
      \draw (b0) --++(0,{-3*\dy})-| (b14);

      \draw (a5) --++(0,{2*\dy})-| (a13);
      \draw (a6) --++(0,{1*\dy})-| (a8);

      \draw (b5) --++(0,{-2*\dy})-| (b13);
      \draw (b6) --++(0,{-1*\dy})-| (b8);

      \draw (a7) -- (b7);
      \node at (5.25,1.2) {$\otimes$};
\node at (5.25,4.4) {$\otimes$};      
 \begin{scope}[xshift=6cm]
      \def\dxx{0.75}
      \def\dyy{0.8}
      \def\ddd{0.5}
      \def\txx{3*\dxx}
      \def\tyy{3*\dyy}      
      \draw[dotted]({-\ddd},0) -- ({\txx+\ddd},0);
      \draw[dotted]({-\ddd},{\tyy}) -- ({\txx+\ddd},{\tyy});      
      \node[scale=0.4,circle,draw=gray,fill=white] (a1) at ({0*\dxx},{0}) {};
      \node[scale=0.4,circle,draw=gray,fill=gray] (a2) at ({1*\dxx},{0}) {};
      \node[scale=0.4,circle,draw=gray,fill=gray] (a3) at ({2*\dxx},{0}) {};
      \node[scale=0.4,circle,draw=gray,fill=white] (a4) at ({3*\dxx},{0}) {};
      \draw[gray] (a1) --++(0,{1*\dyy})-| (a3);
      \draw[gray] (a2) --++(0,{2*\dyy})-| (a4);
      \draw[dotted]({-\ddd},{\tyy+\dyy}) -- ({\txx+\ddd},{\tyy+\dyy});
      \draw[dotted]({-\ddd},{2*\tyy+\dyy}) -- ({\txx+\ddd},{2*\tyy+\dyy});      
      \node[scale=0.4,circle,draw=lightgray,fill=white] (b1) at ({0*\dxx},{2*\tyy+\dyy}) {};
      \node[scale=0.4,circle,draw=lightgray,fill=lightgray] (b2) at ({1*\dxx},{2*\tyy+\dyy}) {};
      \node[scale=0.4,circle,draw=lightgray,fill=lightgray] (b3) at ({2*\dxx},{2*\tyy+\dyy}) {};
      \node[scale=0.4,circle,draw=lightgray,fill=white] (b4) at ({3*\dxx},{2*\tyy+\dyy}) {};
      \draw[lightgray] (b1) --++(0,{-1*\dyy})-| (b3);
      \draw[lightgray] (b2) --++(0,{-2*\dyy})-| (b4);
      \node at ({1.5*\dx},2.8) {$\otimes$};
    \end{scope}
    \end{scope}    
    \end{tikzpicture}
  \end{center}

  Then, the partition $p$ being a bracket, at most one connected component of $p$ can contain a through block and thus possibly have more than $n$ elements. (Here and in the following, such estimates for the number of points merely constitute convenient bounds and need not be optimal.)  This component encompassing through blocks has no more than $2(n-1)$ points and its factor partition is necessarily a bracket. The latter is hence contained in $\left\langle\mc C\cap \resbr\right\rangle$ by the induction hypothesis. All other connected components of $p$ have at most $n$ points as they each comprise non-through blocks exclusively. Because, by the induction hypothesis, $\left\langle \mc C\cap \resbr\right\rangle$ contains the brackets with $2(n-1)$ points or less, by Lemma~\ref{lemma:generation_of_categories_by_brackets}, all partitions with at most $n$ points are elements of $\left\langle \mc C\cap \resbr\right\rangle$. Consequently, all the factor partitions into which $p$ decomposes, and thus by Lemma~\ref{lemma:generation_of_categories_by_connected_partitions} the partition $p$ itself, lie in $\left\langle \mc C\cap\resbr\right\rangle$.
  \par
  \emph{Case 2:} Now, suppose that $p$ is connected, and let the lower row $S$ of $p$ contain a proper subsector $S_0$ such that $\partial S_0$ is a turn block of a turn $T$ in $\inte(S)$. Let $T'$ and $S_0'$ denote their respective counterparts on the upper row of $p$.
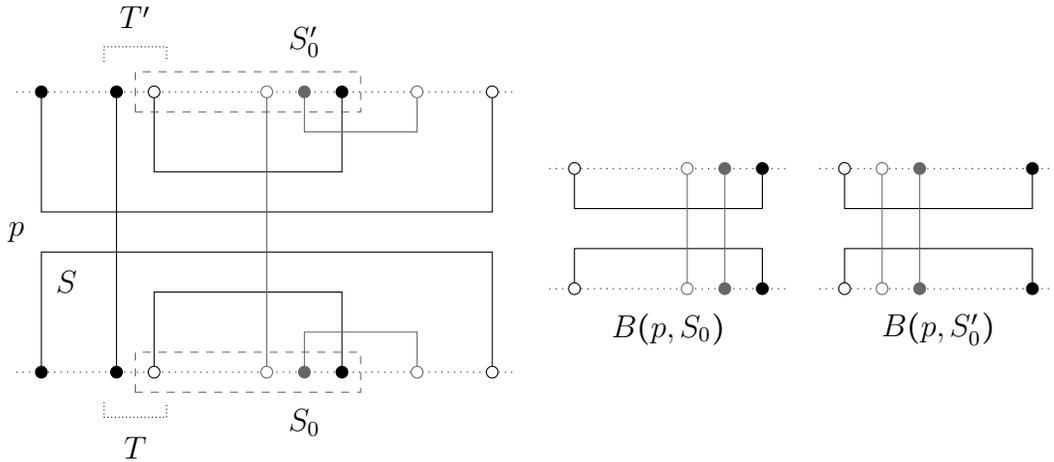
  \begin{align*}
    \begin{tikzpicture}[scale=0.666,baseline=0]
      \draw [dotted] (-0.5,0) -- (9.5,0);
      \draw [dotted] (-0.5,5.6) -- (9.5,5.6);
      \node [circle, scale=0.4, draw=black, fill=black] (x1a) at (0,0) {};
      \node [circle, scale=0.4, draw=black, fill=white] (x1b) at (9,0) {};
      \node [circle, scale=0.4, draw=black, fill=black] (x2a) at (0,5.6) {};
      \node [circle, scale=0.4, draw=black, fill=white] (x2b) at (9,5.6) {};
      \node [circle, scale=0.4, draw=black, fill=white] (y1a) at (2.25,0) {};
      \node [circle, scale=0.4, draw=black, fill=white] (y2a) at (2.25,5.6) {};
      \node [circle, scale=0.4, draw=black, fill=black] (y1b) at (6,0) {};
      \node [circle, scale=0.4, draw=black, fill=black] (y2b) at (6,5.6) {};
      \node [circle, scale=0.4, draw=black, fill=black] (y3a) at (1.5,0) {};
      \node [circle, scale=0.4, draw=black, fill=black] (y3b) at (1.5,5.6) {};            
      \draw (x1a) -- ++ (0,2.4) -| (x1b);
      \draw (x2a) -- ++ (0,-2.4) -| (x2b);
      \draw (y1a) -- ++ (0,1.6) -| (y1b);
      \draw (y2a) -- ++ (0,-1.6) -| (y2b);      
      \draw (y3a) to (y3b);
      \node [circle, scale=0.4, draw=gray, fill=gray] (z1a) at (5.25,0) {};
      \node [circle, scale=0.4, draw=gray, fill=white] (z1b) at (7.5,0) {};
      \draw[gray] (z1a) -- ++ (0,0.8) -| (z1b);
      \node [circle, scale=0.4, draw=gray, fill=gray] (z2a) at (5.25,5.6) {};
      \node [circle, scale=0.4, draw=gray, fill=white] (z2b) at (7.5,5.6) {};
      \draw[gray] (z2a) -- ++ (0,-0.8) -| (z2b);
      \node [circle, scale=0.4, draw=gray, fill=white] (z4a) at (4.5,0) {};
      \node [circle, scale=0.4, draw=gray, fill=white] (z4b) at (4.5,5.6) {};
      \draw[gray] (z4a)  to (z4b);
      \draw[gray, dashed] (1.875,-0.4) rectangle (6.375,0.4);
      \draw[gray, dashed] (1.875,5.2) rectangle (6.375,6.0);      
      \node at (-0.5,2.8) {$p$};
      \node at (0.5,1.8) {$S$};      
      \node at (5.25,-1) {$S_0$};
      \node at (5.25,6.6) {$S_0'$};
      \draw[densely dotted] (1.25,-0.6) -- ++ (0,-0.3) -- ++ (1.25,0) -- ++ (0,0.3);
      \node at (1.875,-1.5) {$T$};
      \draw[densely dotted] (1.25,6.2) -- ++ (0,0.3) -- ++ (1.25,0) -- ++ (0,-0.3);
      \node at (1.875,7.1) {$T'$};
    \end{tikzpicture}
    \quad
    \begin{tikzpicture}[scale=0.666,baseline=-1.111cm]
      \draw [dotted] (-0.5,0) -- (4.25,0);
      \draw [dotted] (-0.5,2.4) -- (4.25,2.4);
      \node [circle, scale=0.4, draw=black, fill=white] (y1a) at (0,0) {};
      \node [circle, scale=0.4, draw=black, fill=white] (y2a) at (0,2.4) {};
      \node [circle, scale=0.4, draw=black, fill=black] (y1b) at (3.75,0) {};
      \node [circle, scale=0.4, draw=black, fill=black] (y2b) at (3.75,2.4) {};
      \draw (y1a) -- ++ (0,0.8) -| (y1b);
      \draw (y2a) -- ++ (0,-0.8) -| (y2b);      
      \node [circle, scale=0.4, draw=gray, fill=gray] (z1a) at (3,0) {};
      \node [circle, scale=0.4, draw=gray, fill=gray] (z1b) at (3,2.4) {};
      \node [circle, scale=0.4, draw=gray, fill=white] (z2a) at (2.25,0) {};
      \node [circle, scale=0.4, draw=gray, fill=white] (z2b) at (2.25,2.4) {};
      \draw[gray] (z1a) to (z1b);
      \draw[gray] (z2a) to (z2b);
      \node at (1.875,-0.8) {$B(p,S_0)$};      
    \end{tikzpicture}
    \quad
    \begin{tikzpicture}[scale=0.666,baseline=-1.111cm]
      \draw [dotted] (-0.5,0) -- (4.25,0);
      \draw [dotted] (-0.5,2.4) -- (4.25,2.4);
      \node [circle, scale=0.4, draw=black, fill=white] (y1a) at (0,0) {};
      \node [circle, scale=0.4, draw=black, fill=white] (y2a) at (0,2.4) {};
      \node [circle, scale=0.4, draw=black, fill=black] (y1b) at (3.75,0) {};
      \node [circle, scale=0.4, draw=black, fill=black] (y2b) at (3.75,2.4) {};
      \draw (y1a) -- ++ (0,0.8) -| (y1b);
      \draw (y2a) -- ++ (0,-0.8) -| (y2b);      
      \node [circle, scale=0.4, draw=gray, fill=gray] (z1a) at (1.5,0) {};
      \node [circle, scale=0.4, draw=gray, fill=gray] (z1b) at (1.5,2.4) {};
      \node [circle, scale=0.4, draw=gray, fill=white] (z2a) at (0.75,0) {};
      \node [circle, scale=0.4, draw=gray, fill=white] (z2b) at (0.75,2.4) {};
      \draw[gray] (z1a) to (z1b);
      \draw[gray] (z2a) to (z2b);
      \node at (1.875,-0.8) {$B(p,S_0')$};      
    \end{tikzpicture}
  \end{align*}
  The brackets $B(p,S_0)$ and $B(p,S_0')$, both elements of $\mathcal C$ by Lemma~\ref{lemma:associated_brackets}, have at most $2(n-1)$ points and are therefore elements of $\left\langle\mc C\cap \resbr\right\rangle$ by the induction hypothesis. The partition $E(E(p,T),T')$ is a bracket of $\mathcal C$ with at most $2(n-1)$ points as well. It, too, is hence contained in $\left\langle \mc C\cap\resbr\right\rangle$ by the induction hypothesis. Because $\left\langle \mc C\cap\resbr\right\rangle$ now contains both $E(E(p,T),T')$ and $B(p,S_0')=B(E(p,T),S_0')$, applying Lemma~\ref{lemma:reversing_erasing_of_turns} yields $E(p,T)\in \left\langle \mc C\cap \resbr\right\rangle$. A second application of this lemma hence shows $p\in \left\langle \mc C\cap \resbr \right\rangle$ as $E(p,T)\in \left\langle \mc C\cap \resbr\right\rangle$ and $B(p,S_0)\in \left\langle \mc C\cap \resbr\right \rangle$.\par
  \emph{Case 3:} Lastly, assume that $p$ is connected but that only through-blocks emanate from any turns $\inte(S)$ might contain. If $\inte(S)$ contains no turns, then $p$ is residual of the first kind and hence an element of $\mc C\cap \resbr$. So, suppose that $\inte(S)$ contains at least one turn and let $T$ be the rightmost turn inside $\inte(S)$ and let $S'$ be the sector to the right of the right turn block $\partial S'$ of $T$ (see next page for an illustration).
  Now, crucially, the bracket $B(p,S')^{\dagger}$ is residual of the second kind and hence an element of $\mc C\cap\resbr$. Indeed, since $S'$ spreads symmetrically across both rows of $p$ and $p$ is projective, $B(p,S')$ is verticolor-reflexive. Dualizable is $B(p,S')$ because $S'$ has at its center the boundaries of the sectors of $p$ given by the lower and upper row of the projective $p$ and because $\partial S'$ is a through block of $p$. With $B(p,S')$ being a dualizable bracket, so is $B(p,S')^\dagger$. Lastly, there is exactly one turn in the interior of the lower row of $B(p,S')^\dagger$ since we chose $T$ specifically to be the rightmost turn of $\inte(S)$. Moreover, $B(p,S')^\dagger$ is connected because $p$ is. Especially, if $p$ is dualizable itself, then $p=B(p,S')^\dagger$.
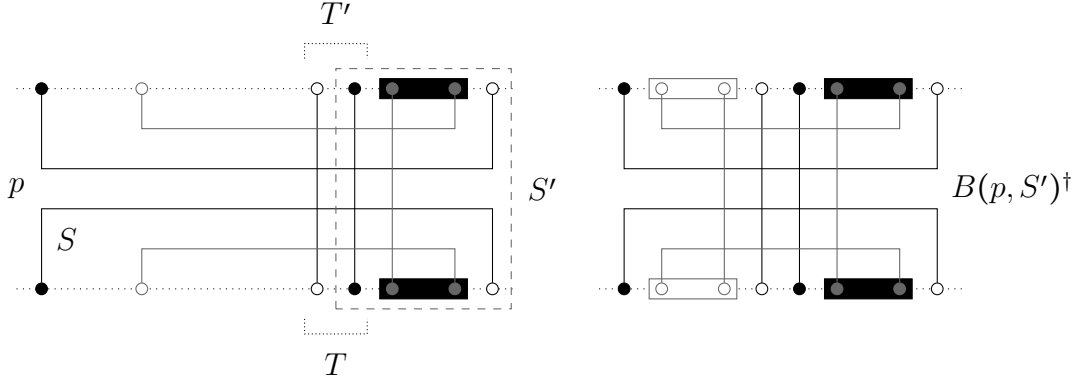
  \begin{align*}
    \begin{tikzpicture}[scale=0.666,baseline=0]
      \draw [dotted] (-0.5,0) -- (9.5,0);
      \draw [dotted] (-0.5,4) -- (9.5,4);
      \draw[fill=black] (6.75,-0.2) rectangle (8.5, 0.2);
      \draw[fill=black] (6.75,3.8) rectangle (8.5, 4.2);            
      \node [circle, scale=0.4, draw=black, fill=black] (x1a) at (0,0) {};
      \node [circle, scale=0.4, draw=black, fill=white] (x1b) at (9,0) {};
      \node [circle, scale=0.4, draw=black, fill=black] (x2a) at (0,4) {};
      \node [circle, scale=0.4, draw=black, fill=white] (x2b) at (9,4) {};
      \node [circle, scale=0.4, draw=black, fill=white] (y1a) at (5.5,0) {};
      \node [circle, scale=0.4, draw=black, fill=white] (y1b) at (5.5,4) {};
      \node [circle, scale=0.4, draw=black, fill=black] (y2a) at (6.25,0) {};
      \node [circle, scale=0.4, draw=black, fill=black] (y2b) at (6.25,4) {};      
      \draw (x1a) -- ++ (0,1.6) -| (x1b);
      \draw (x2a) -- ++ (0,-1.6) -| (x2b);
      \draw (y1a) to (y1b);
      \draw (y2a) to (y2b);
      \node [circle, scale=0.4, draw=gray, fill=white] (z1a) at (2,0) {};
      \node [circle, scale=0.4, draw=gray, fill=gray] (z1b) at (8.25,0) {};
      \draw[gray] (z1a) -- ++ (0,0.8) -| (z1b);
      \node [circle, scale=0.4, draw=gray, fill=white] (z2a) at (2,4) {};
      \node [circle, scale=0.4, draw=gray, fill=gray] (z2b) at (8.25,4) {};
      \draw[gray] (z2a) -- ++ (0,-0.8) -| (z2b);
      \node [circle, scale=0.4, draw=gray, fill=gray] (z3a) at (7,0) {};
      \node [circle, scale=0.4, draw=gray, fill=gray] (z3b) at (7,4) {};
      \draw[gray] (z3a)  to (z3b);
      \node at (-0.5,2) {$p$};
      \node at (0.5,1) {$S$};
      \draw[gray, dashed] (5.875,-0.4) rectangle (9.375,4.4);
      \node at (10,2) {$S'$};
      \draw[densely dotted] (5.25,-0.6) -- ++ (0,-0.3) -- ++ (1.25,0) -- ++ (0,0.3);
      \node at (5.875,-1.5) {$T$};
      \draw[densely dotted] (5.25,4.6) -- ++ (0,0.3) -- ++ (1.25,0) -- ++ (0,-0.3);
      \node at (5.875,5.5) {$T'$};
    \end{tikzpicture}
    \quad
    \begin{tikzpicture}[scale=0.666,baseline=0]
      \draw [dotted] (2.25,0) -- (9.5,0);
      \draw [dotted] (2.25,4) -- (9.5,4);
      \draw[fill=black] (6.75,-0.2) rectangle (8.5, 0.2);
      \draw[fill=black] (6.75,3.8) rectangle (8.5, 4.2);
      \draw[draw=gray, fill=white] (3.25,-0.2) rectangle (5, 0.2);
      \draw[draw=gray, fill=white] (3.25,3.8) rectangle (5, 4.2);                  
      \node [circle, scale=0.4, draw=black, fill=black] (x1a) at (2.75,0) {};
      \node [circle, scale=0.4, draw=black, fill=white] (x1b) at (9,0) {};
      \node [circle, scale=0.4, draw=black, fill=black] (x2a) at (2.75,4) {};
      \node [circle, scale=0.4, draw=black, fill=white] (x2b) at (9,4) {};
      \node [circle, scale=0.4, draw=black, fill=white] (y1a) at (5.5,0) {};
      \node [circle, scale=0.4, draw=black, fill=white] (y1b) at (5.5,4) {};
      \node [circle, scale=0.4, draw=black, fill=black] (y2a) at (6.25,0) {};
      \node [circle, scale=0.4, draw=black, fill=black] (y2b) at (6.25,4) {};      
      \draw (x1a) -- ++ (0,1.6) -| (x1b);
      \draw (x2a) -- ++ (0,-1.6) -| (x2b);
      \draw (y1a) to (y1b);
      \draw (y2a) to (y2b);
      \node [circle, scale=0.4, draw=gray, fill=white] (z1a) at (3.5,0) {};
      \node [circle, scale=0.4, draw=gray, fill=gray] (z1b) at (8.25,0) {};
      \draw[gray] (z1a) -- ++ (0,0.8) -| (z1b);
      \node [circle, scale=0.4, draw=gray, fill=white] (z2a) at (3.5,4) {};
      \node [circle, scale=0.4, draw=gray, fill=gray] (z2b) at (8.25,4) {};
      \draw[gray] (z2a) -- ++ (0,-0.8) -| (z2b);
      \node [circle, scale=0.4, draw=gray, fill=gray] (z3a) at (7,0) {};
      \node [circle, scale=0.4, draw=gray, fill=gray] (z3b) at (7,4) {};
      \draw[gray] (z3a)  to (z3b);
      \node [circle, scale=0.4, draw=gray, fill=white] (z4a) at (4.75,0) {};
      \node [circle, scale=0.4, draw=gray, fill=white] (z4b) at (4.75,4) {};
      \draw[gray] (z4a)  to (z4b);      
      \node at (10.5,2) {$B(p,S')^\dagger$};
    \end{tikzpicture}
  \end{align*}
  Now, let $T'$ denote the counterpart of $T$ on the upper row again. The partition $E(E(p,T),T')$ is a bracket of $\mathcal C$ with $2(n-1)$ elements at most and hence contained in $\left\langle \mc C\cap \resbr\right\rangle$ by the induction hypothesis. Because $T'$ is a connected component of $E(p,T)$ whose factor partition is a rotation of $\PartIdenLoBW$ or $\PartIdenLoWB$, we infer $E(p,T)\in \left\langle E(E(p,T),T')\right\rangle\subseteq \left\langle \mc C\cap \resbr\right\rangle$ with the help of Lemma~\ref{lemma:generation_of_categories_by_connected_partitions}. Lemma~\ref{lemma:reversing_erasing_of_turns} now proves $p\in \left\langle \mc C\cap\resbr\right\rangle$ due to $E(p,T)\in \left\langle \mc C\cap\resbr\right\rangle$ and $B(p,S')\in \left\langle \mc C\cap\resbr\right\rangle$.\par
  In conclusion, $p\in \langle \mc C\cap\resbr\rangle$ holds always, which is what we needed to show.
\end{proof}

\subsection{Bracket Arithmetics}
Proposition~\ref{proposition:generation_of_categories_by_residual_brackets} showed that the subcategories of $\Cppnb$ are given precisely by the set  $\{\langle\mc G\rangle\mid \mc G\subseteq \resbr \}$. In this subsection we begin investigating the map $\mathfrak{P}(\resbr)\to\mathfrak{P}(\resbr),\; \mc G\mapsto \langle \mc G\rangle\cap \resbr$. The fixed points of this map are in bijection with the subcategories of $\Cppnb$. Especially, knowing which residual brackets generate which is key to proving the remaining parts of the \hyperref[section:main]{Main Theorems} in Sections~\ref{section:generators} and~\ref{section:classification}. And, in the follow-up article, we determine the full graph of the above map to find all the subcategories of $\mc S_0$.\par
We require more precise notation to be able to address specific brackets and certain maps $\mc B\to \mc B$, particularly those mapping elements of $\resbr$ again to $\resbr$.
\par
\vspace{0.75em}
\noindent
\begin{minipage}[c]{0.666\textwidth}
	\begin{definition}
		\begin{enumerate}[label=(\alph*)]
                  			\item If $p\in\mathcal{P}^{\circ\bullet}$ is a bracket, the projective partition $\Arg(p)$ which is obtained from $p$ by erasing in every row the left- and the rightmost point, is called the \emph{argument} of $p$. 
			\item Conversely, for each projective $a\in\mathcal P^{\circ\bullet}_{2,\mathrm{nb}}$ and every color $c\in\{\circ,\bullet\}$, denote by $\mathrm{Br}\left(c\mid a\mid \overline{c}\right)$ the bracket whose leftmost lower point is of color $c$ and which has the argument $a$. 
		\end{enumerate}
	\end{definition}
\end{minipage}
\hfill
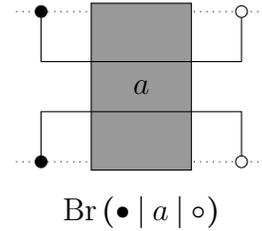
\begin{minipage}{5cm}
	\centering
	\begin{tikzpicture}[scale=0.666]
	\draw [dotted, shift={(-0.5,0)}] (0,0) -- (5,0);
	\draw [dotted, shift={(-0.5,3)}] (0,0) -- (5,0);		
	\node [scale=0.4, fill=black, draw=black,circle] (x1) at (0,0) {};
	\node [scale=0.4, fill=white, draw=black,circle] (x2) at (4,0) {};		
	\node [scale=0.4, fill=black, draw=black,circle] (y1) at (0,3) {};
	\node [scale=0.4, fill=white, draw=black,circle] (y2) at (4,3) {};				
	\draw [fill=lightgray] (1,-0.166) rectangle (3,3.166);
	\draw (x1) -- ++ (0,1) -| (x2);
	\draw (y1) -- ++ (0,-1) -| (y2);
	\node at (2,1.5) {$a$};
	\node at (2,-1) {$\mathrm{Br}\left(\bullet\mid a\mid \circ\right)$};
	\end{tikzpicture}
\end{minipage}

\begin{definition}
	For any tuple $(c_1,\ldots,c_n)$ of colors $c_1,\ldots,c_n\in \{\circ,\bullet\}$, $n\in\mathbb{N}$, denote by $\mathrm{Id}(c_1\ldots c_n)$ the tensor product of $n$ partitions from $\left\{\PartIdenW,\PartIdenB\right\}$ such that the lower row (and thus also the upper row) has the coloration $(c_1,\ldots,c_n)$.
      \end{definition}
      We define on the set of all brackets a partial associative operation applicable to all pairs of brackets whose lower rows start with the same color.
\begin{definition}
  For every $c\in\{ \circ,\bullet\}$ and all projective $a,b\in\Cpp$, we call
    \begin{align*}
      \mathrm{Br}\left(c\mid a\mid \overline{c}\right)\boxtimes \mathrm{Br}\left(c\mid b\mid \overline c\right)\eqpd\mathrm{Br}\left(c\mid a\otimes b\mid \overline c\right)
    \end{align*}
    the \emph{bracket product} of $(\mathrm{Br}\left(c\mid a\mid \overline{c}\right), \mathrm{Br}\left(c\mid b\mid \overline c\right))$.
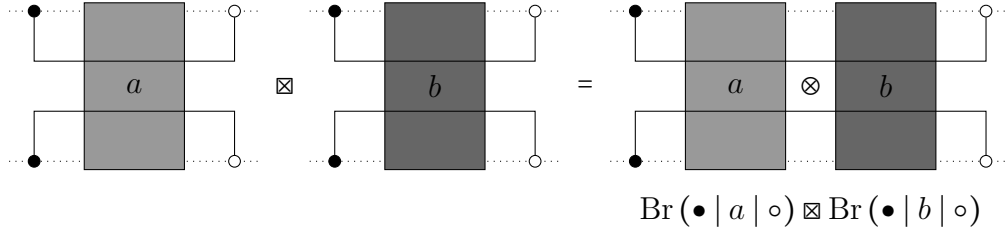
        \begin{gather*}
        \begin{tikzpicture}[scale=0.666]
	\draw [dotted] (-0.5,0) -- (4.5,0);
	\draw [dotted] (-0.5,3) -- (4.5,3);
	\node [scale=0.4, fill=black, draw=black,circle] (x1) at (0,0) {};
	\node [scale=0.4, fill=white, draw=black,circle] (x2) at (4,0) {};		
	\node [scale=0.4, fill=black, draw=black,circle] (y1) at (0,3) {};
	\node [scale=0.4, fill=white, draw=black,circle] (y2) at (4,3) {};
	\draw [fill=lightgray] (1,-0.166) rectangle (3,3.166);
	\draw (x1) -- ++ (0,1) -| (x2);
	\draw (y1) -- ++ (0,-1) -| (y2);
	\node at (2,1.5) {$a$};
        \node at (5,1.5) {$\boxtimes$};
        \begin{scope}[xshift=6cm]
          \draw [dotted] (-0.5,0) -- (4.5,0);
	\draw [dotted] (-0.5,3) -- (4.5,3);
	\node [scale=0.4, fill=black, draw=black,circle] (x1) at (0,0) {};
	\node [scale=0.4, fill=white, draw=black,circle] (x2) at (4,0) {};		
	\node [scale=0.4, fill=black, draw=black,circle] (y1) at (0,3) {};
	\node [scale=0.4, fill=white, draw=black,circle] (y2) at (4,3) {};
	\draw [fill=gray] (1,-0.166) rectangle (3,3.166);
	\draw (x1) -- ++ (0,1) -| (x2);
	\draw (y1) -- ++ (0,-1) -| (y2);
	\node at (2,1.5) {$b$};
          \end{scope}
          \node at (15.5,-1) {$\mathrm{Br}\left(\bullet\mid a\mid \circ\right)\boxtimes \mathrm{Br}\left(\bullet\mid b\mid \circ\right)$};
          \node at (15.5,1.5) {$\otimes$};          
        \node at (11,1.5) {$=$};
          \begin{scope}[xshift=12cm]
        \draw [dotted] (-0.5,0) -- (7.5,0);
	\draw [dotted] (-0.5,3) -- (7.5,3);
	\node [scale=0.4, fill=black, draw=black,circle] (x1) at (0,0) {};
	\node [scale=0.4, fill=white, draw=black,circle] (x2) at (7,0) {};		
	\node [scale=0.4, fill=black, draw=black,circle] (y1) at (0,3) {};
	\node [scale=0.4, fill=white, draw=black,circle] (y2) at (7,3) {};
	\draw [fill=lightgray] (1,-0.166) rectangle (3,3.166);
	\draw [fill=gray] (4,-0.166) rectangle (6,3.166);        
	\draw (x1) -- ++ (0,1) -| (x2);
	\draw (y1) -- ++ (0,-1) -| (y2);
	\node at (2,1.5) {$a$};
	\node at (5,1.5) {$b$};
          \end{scope}
      \end{tikzpicture}
    \end{gather*}
  \end{definition}    
    For every bracket, we define two ways of changing the starting color and, in some sense, retaining the argument.
    \begin{definition}
      For every $c\in \colors$ and projective $a\in \Cppnb$, we call
    \begin{align*}
      \trmw(\mathrm{Br}\left(c\mid a\mid \overline{c}\right))\eqpd \mathrm{Br}\left(\overline{c}\mid \mathrm{Br}\left(c\mid a\mid \overline{c}\right)\mid c\right)
    \end{align*}
    the \emph{weak inversion} and
        \begin{align*}
      \trms(\mathrm{Br}\left(c\mid a\mid \overline{c}\right))\eqpd  \mathrm{Br}\left({\overline{c}}\mid \mathrm{Id}(c)\otimes a\otimes \mathrm{Id}(\overline{c})\mid c\right)
        \end{align*}
        the \emph{strong inversion} of $\mathrm{Br}\left(c\mid a\mid \overline{c}\right)$.
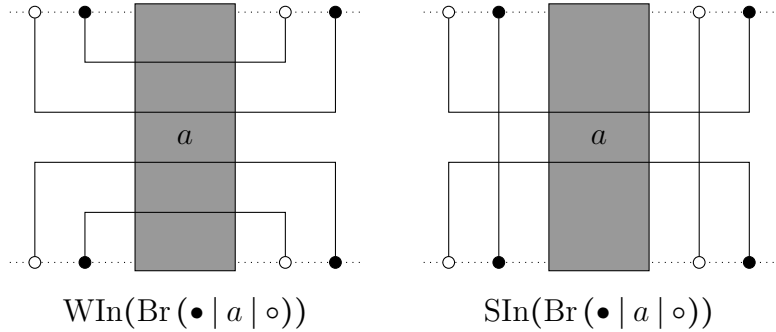
      \begin{gather*}
        \begin{tikzpicture}[scale=0.666]
	\draw [dotted, shift={(-0.5,0)}] (-1,0) -- (6,0);
	\draw [dotted, shift={(-0.5,5)}] (-1,0) -- (6,0);
	\node [scale=0.4, fill=white, draw=black,circle] (z1) at (-1,0) {};
	\node [scale=0.4, fill=black, draw=black,circle] (z2) at (5,0) {};
      	\node [scale=0.4, fill=white, draw=black,circle] (w1) at (-1,5) {};
	\node [scale=0.4, fill=black, draw=black,circle] (w2) at (5,5) {};		        
	\node [scale=0.4, fill=black, draw=black,circle] (x1) at (0,0) {};
	\node [scale=0.4, fill=white, draw=black,circle] (x2) at (4,0) {};		
	\node [scale=0.4, fill=black, draw=black,circle] (y1) at (0,5) {};
	\node [scale=0.4, fill=white, draw=black,circle] (y2) at (4,5) {};				
	\draw [fill=lightgray] (1,-0.166) rectangle (3,5.166);
	\draw (x1) -- ++ (0,1) -| (x2);
	\draw (y1) -- ++ (0,-1) -| (y2);
	\draw (z1) -- ++ (0,2) -| (z2);
	\draw (w1) -- ++ (0,-2) -| (w2);        
	\node at (2,2.5) {$a$};
	\node at (2,-1) {$\trmw(\mathrm{Br}\left(\bullet\mid a\mid \circ\right))$};
      \end{tikzpicture}
      \hspace{2em}
              \begin{tikzpicture}[scale=0.666]
	\draw [dotted, shift={(-0.5,0)}] (-1,0) -- (6,0);
	\draw [dotted, shift={(-0.5,5)}] (-1,0) -- (6,0);
	\node [scale=0.4, fill=white, draw=black,circle] (z1) at (-1,0) {};
	\node [scale=0.4, fill=black, draw=black,circle] (z2) at (5,0) {};
      	\node [scale=0.4, fill=white, draw=black,circle] (w1) at (-1,5) {};
	\node [scale=0.4, fill=black, draw=black,circle] (w2) at (5,5) {};		        
	\node [scale=0.4, fill=black, draw=black,circle] (x1) at (0,0) {};
	\node [scale=0.4, fill=white, draw=black,circle] (x2) at (4,0) {};		
	\node [scale=0.4, fill=black, draw=black,circle] (y1) at (0,5) {};
	\node [scale=0.4, fill=white, draw=black,circle] (y2) at (4,5) {};				
	\draw [fill=lightgray] (1,-0.166) rectangle (3,5.166);
	\draw (x1) --  (y1);
	\draw (x2) --  (y2);
	\draw (z1) -- ++ (0,2) -| (z2);
	\draw (w1) -- ++ (0,-2) -| (w2);        
	\node at (2,2.5) {$a$};
	\node at (2,-1) {$\trms(\mathrm{Br}\left(\bullet \mid a \mid  \circ\right))$};
	\end{tikzpicture}
      \end{gather*}
\end{definition}      
\begin{lemma}
	\label{lemma:bracket_operations}
	Let $p$ and $p'$ be two brackets starting with the same color.
	\begin{enumerate}[label=(\alph*)]
		\item 
		\label{item:bracket_operations-item_1}
               Categories are closed under the bracket product: $p\boxtimes p'\in \langle p,p'\rangle$.
		\item 
                  \label{item:bracket_operations-item_2}
The brackets $p$ and $\tilde p$ start with the same color and $\Arg(\tilde p)$ is the verticolor-reflection of $\Arg(p)$.
                \item
       		\label{item:bracket_operations-item_5}
                 Categories are closed under taking arguments of brackets: $\Arg(p)\in \langle p\rangle$.
		\item
                  \label{item:bracket_operations-item_6}
                  It holds
                	$\langle \PartBracketBWBW\rangle=\langle\PartHalfLibWBW \rangle=\langle \PartHalfLibBWB\rangle=\langle \PartBracketWBWB\rangle$.
		\item 
		\label{item:bracket_operations-item_3}		
                Weak inversion is a reversible category operation: $\langle p\rangle=\langle \trmw(p)\rangle$.
              \item
		\label{item:bracket_operations-item_4}
                Strong inversion is reversible as well, but it is only available in certain categories:
$\langle p, \PartHalfLibWBW \rangle =\langle \trms(p)\rangle$.
	\end{enumerate}
      \end{lemma}
      \begin{proof}
Let $c\in \{\circ,\bullet\}$ be the starting color of $p$ and $p'$ and abbreviate $a\eqpd \Arg(p)$.
	\begin{enumerate}[wide, labelwidth = !, label=(\alph*) ]
		\item Erasing in $p\otimes p'$ the rightmost lower point formerly of $p$ and the leftmost lower point formerly of $p'$ and passing to the bracket associated with the lower row of the resulting partition produces the bracket $p\boxtimes p'$, which proves the claim by Lemma~\ref{lemma:associated_brackets}.
		\item Verticolor reflection turns $\mathrm{Br}\left(c\mid a\mid \overline c\right)$ into $\mathrm{Br}\left(\overline{\overline{c}}\mid \tilde a\mid \overline c\right)=\mathrm{Br}\left(c\mid \tilde a\mid \overline c\right)$. 
                \item Erasing in $p$ the two turns formed, on the one hand, by the leftmost upper point and its successor and, on the other hand, by the rightmost lower point and its successor yields $\Arg(p)$.
                  \item We show the chain of inclusions from the left to the right. For the first, erase in $\PartBracketBWBW$ the two leftmost lower points and rotate the leftmost upper point to the lower row to obtain $\PartHalfLibWBW$. The second inclusion is seen to hold by recognizing that $\PartHalfLibWBW$ and $\PartHalfLibBWB$ are cyclic rotations of each other. Using Lemma~\ref{lemma:associated_brackets} to project to the right side sector of the middle block of $\PartHalfLibBWB$ gives $\PartBracketWBWB$. Since, in this chain, the roles of $\bullet$ and $\circ$ can be exchanged, the claim follows.
		\item The bracket associated with the lower row of the partition $(p^\lcurvearrowdown)^\rcurvearrowdown$ is $\trmw(p)$. So, Lemma~\ref{lemma:associated_brackets} proves one inclusion. The converse is a consequence of Part~\ref{item:bracket_operations-item_5}.
		\item Starting with the generators on the left hand side, employ Parts~\ref{item:bracket_operations-item_1} and~\ref{item:bracket_operations-item_6} to construct $\mathrm{Br}\left(c\mid \mathrm{Id}(\overline{c}c)\mid \overline{c}\right)\boxtimes p\boxtimes \mathrm{Br}\left(c\mid \mathrm{Id}(\overline{c}c)\mid \overline{c}\right)\in\langle p,\PartHalfLibWBW\rangle$. Erasing there  the two last upper and the two first lower points yields $\trms(p)^\circlearrowright$. Conversely, applying the last step of the previous procedure to $\trms(p)$ returns $p^\circlearrowright$. And to obtain $\PartHalfLibWBW$, project with Lemma~\ref{lemma:associated_brackets} to the left side sector of the block $\mathrm{Id}(c)$ in the partition $\trms(p)=\mathrm{Br}\left({\overline{c}}\mid \mathrm{Id}(c)\otimes a\otimes \mathrm{Id}(\overline{c})\mid c\right)$ and use Part~\ref{item:bracket_operations-item_6}. \qedhere
	\end{enumerate}
\end{proof}

\section{Generators\texorpdfstring{ of $\mc S_w$\\{[Main Theorem~\ref*{theorem:main_1}~{\normalfont\ref*{item:main_1-3}}]}}{}}
\label{section:generators}
 
We give a three-step proof of Part~\ref{item:main_1-3} of Main~Theorem~\ref{theorem:main_1}, using the results of Section~\ref{section:brackets}.
First, we apply the general results of Lemma~\ref{lemma:bracket_operations} to the alleged generators of $\mc S_w$ for $w\in \pint$, i.e.\ $w>0$.
\begin{lemma}
  \label{lemma:arithmetics_generators_of_S_w}
  Let $w\in \pint$ be arbitary.
  \begin{enumerate}[label=(\alph*)]
  \item \label{item:arithmetics_generators_of_S_w-item_1}
    It holds
		$\langle \PartBracketBBwW\rangle=\langle\PartBracketBWwW\rangle=\langle \PartBracketWBwB\rangle=\langle\PartBracketWWwB\rangle$.
              \item \label{item:arithmetics_generators_of_S_w-item_2}
                It holds $\PartHalfLibWBW\in \langle\PartBracketBBwW\rangle$.
              \item \label{item:arithmetics_generators_of_S_w-item_3}
                For all $v\in \pint$ hold $\PartBracketBBvWvW\in \langle \PartBracketBBwW\rangle$ and  $\PartBracketWWvBvB\in \langle\PartBracketBBwW\rangle$.
  \end{enumerate}
\end{lemma}
\begin{proof}
  \begin{enumerate}[label=(\alph*), labelwidth=!,wide]
  \item 
    Let $c\in \{\circ,\bullet\}$ be arbitrary. To verify Part~\ref{item:arithmetics_generators_of_S_w-item_1}, by Lemma \hyperref[item:bracket_operations-item_2]{\ref*{lemma:bracket_operations}~\ref*{item:bracket_operations-item_2}}, it suffices to show that $\mathrm{Br}\left(\overline c\mid \mathrm{Id}\left( c^{\otimes w}\right)\mid c\right)$ is contained in $\mc C\eqpd\langle \mathrm{Br}\left( c \mid \mathrm{Id}\left(c^{\otimes w}\right)\mid \overline c\right)\rangle$. By Lemma \hyperref[item:bracket_operations-item_1]{\ref*{lemma:bracket_operations}~\ref*{item:bracket_operations-item_1}} and \hyperref[item:bracket_operations-item_2]{\ref*{item:bracket_operations-item_2}}, the category  $\mc C$ comprises $\mathrm{Br}\left (c \mid \mathrm{Id}\left({\overline{c}}^{\otimes w}\otimes{c}^{\otimes w} \right) \mid \overline{c}\right)$.
    Erasing symmetrically the $2(w-1)$ middle points on the lower row of the latter and passing to the associated bracket of the result yields $\mathrm{Br}\left (c \mid \mathrm{Id}\left(\overline{c}c\right) \mid \overline{c}\right)\in\mathcal{C}$.  Moreover, Lemma \hyperref[item:bracket_operations-item_1]{\ref*{lemma:bracket_operations}~\ref*{item:bracket_operations-item_1}} now allows us to conclude $\mathrm{Br}\left (c \mid \mathrm{Id}\left(\overline{c}\otimes c^{\otimes w+1}\right) \mid \overline{c}\right)\in \mc C$. Then, Lemma~\hyperref[item:bracket_operations-item_4]{\ref*{lemma:bracket_operations}~\ref*{item:bracket_operations-item_4}} implies $\mathrm{Br}\left (\overline{c} \mid \mathrm{Id}\left(c^{\otimes w}\right) \mid c\right)\in \mc C$ as $\mathrm{Br}\left (c \mid \mathrm{Id}\left(\overline{c}\otimes c^{\otimes w+1}\right) \mid \overline{c}\right)=\trms(\mathrm{Br}\left (\overline{c} \mid \mathrm{Id}\left(c^{\otimes w}\right) \mid c\right))$.
    \item In the proof of Part~\ref{item:arithmetics_generators_of_S_w-item_1}, we saw especially that $\langle\PartBracketBBwW\rangle$ contains $\PartBracketBWBW$.  With Lemma~\hyperref[item:bracket_operations-item_6]{\ref*{lemma:bracket_operations}~\ref*{item:bracket_operations-item_6}}, we have thus already proven  Part~\ref{item:arithmetics_generators_of_S_w-item_2}.
                  \item Let $c\in \{\circ,\bullet\}$ be arbitrary, and let $k\in \pint$ be large enough such that $v\leq kw$. Then, using Part~\ref{item:arithmetics_generators_of_S_w-item_1} and Lemma~ \hyperref[item:bracket_operations-item_1]{\ref*{lemma:bracket_operations}~\ref*{item:bracket_operations-item_1}}, the category $\mc C\eqpd \langle\PartBracketBBwW\rangle$ contains the bracket $\brt{c}{\overline c}{\mathrm{Id}(c^{\otimes kw}\otimes \overline c^{\otimes kw})}$. Erasing the $2(kw-v)$ middle points on the lower row and passing to the associated bracket of the result proves Claim~\ref{item:arithmetics_generators_of_S_w-item_3}. \qedhere
                    \end{enumerate}
              \end{proof}

We also need the following relationships between the supposed generators of $\mc S_0$.
      \begin{lemma}
        \label{lemma:arithmetics_generators_of_S_0}
        Let $v\in \pint$ be arbitrary.
        \begin{enumerate}[label=(\alph*)]
        \item \label{item:arithmetics_generators_of_S_0-item_2}
          For every $v'\in \pint$ with $v'\leq v$ holds $\PartBracketBBvdWvdW\in\langle \PartBracketBBvWvW\rangle$ and  $\PartBracketWWvdBvdB\in\langle \PartBracketWWvBvB\rangle$.
        \item \label{item:arithmetics_generators_of_S_0-item_1}
          It holds  $\langle \PartBracketBBvWvW\rangle=\langle \PartBracketWWvBvB\rangle$.
        \item \label{item:arithmetics_generators_of_S_0-item_3}
          The category $\langle\PartBracketBBvWvW\rangle$ comprises all connected dualizable brackets whose lower rows contain at most $2v+2$ elements and exactly one turn.
        \end{enumerate}
      \end{lemma}
      \begin{proof}
        \begin{enumerate}[label=(\alph*), labelwidth=!, wide]
        \item \label{item:arithmetics_generators_of_S_0-proof-item_1}
          For $c\in \{\circ,\bullet\}$ and $v'\in \mathbb N$, erase in $\mathrm {Br}\left(c \mid \mathrm {Id}(c^{\otimes v}\otimes \overline {c}^{\otimes v})\mid \overline{c}\right)$ the neutral set of the middle $2(v-v')$ points on the lower row and pass to the associated bracket of the result to obtain $\mathrm {Br}\left(c \mid \mathrm {Id}(c^{\otimes v'}\otimes \overline {c}^{\otimes v'})\mid \overline{c}\right)$ in $\langle\mathrm {Br}\left(c \mid \mathrm {Id}(c^{\otimes v}\otimes \overline {c}^{\otimes v})\mid \overline{c}\right)\rangle$.
          \item
           For all $c\in \{\circ,\bullet\}$ and $v'\in \mathbb N$, write $p_{c,v'}\eqpd\mathrm {Br}\left(c \mid \mathrm {Id}(c^{\otimes v'}\otimes \overline {c}^{\otimes v'})\mid \overline{c}\right)$.
           Note that $p_{c,1}={p_{\overline c,1}}^\dagger$ proves the claim for $v=1$.
  
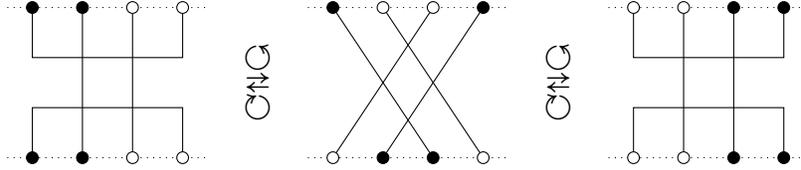
\begin{gather*}
    \begin{tikzpicture}[scale=0.666]
      \draw[dotted] (-0.5,0) -- (3.5,0);
      \draw[dotted] (-0.5,3) -- (3.5,3);
      \node[circle,scale=0.4,draw=black,fill=black] (x1) at (0,0) {};
      \node[circle,scale=0.4,draw=black,fill=black] (x2) at (1,0) {};
      \node[circle,scale=0.4,draw=black,fill=white] (x3) at (2,0) {};
      \node[circle,scale=0.4,draw=black,fill=white] (x4) at (3,0) {};
      \node[circle,scale=0.4,draw=black,fill=black] (y1) at (0,3) {};
      \node[circle,scale=0.4,draw=black,fill=black] (y2) at (1,3) {};
      \node[circle,scale=0.4,draw=black,fill=white] (y3) at (2,3) {};
      \node[circle,scale=0.4,draw=black,fill=white] (y4) at (3,3) {};
      \draw (x1) -- ++(0,1) -| (x4);
      \draw (y1) -- ++(0,-1) -| (y4);
      \draw (x2) to (y2);
      \draw (x3) to (y3);
\node at (4.5,2) {$\circlearrowleft$};      
\node at (4.5,1.5) {$\rightleftarrows$};
\node at (4.5,1) {$\circlearrowright$};
    \begin{scope}[xshift=6cm]
      \draw[dotted] (-0.5,0) -- (3.5,0);
      \draw[dotted] (-0.5,3) -- (3.5,3);
      \node[circle,scale=0.4,draw=black,fill=white] (x1) at (0,0) {};
      \node[circle,scale=0.4,draw=black,fill=black] (x2) at (1,0) {};
      \node[circle,scale=0.4,draw=black,fill=black] (x3) at (2,0) {};
      \node[circle,scale=0.4,draw=black,fill=white] (x4) at (3,0) {};
      \node[circle,scale=0.4,draw=black,fill=black] (y1) at (0,3) {};
      \node[circle,scale=0.4,draw=black,fill=white] (y2) at (1,3) {};
      \node[circle,scale=0.4,draw=black,fill=white] (y3) at (2,3) {};
      \node[circle,scale=0.4,draw=black,fill=black] (y4) at (3,3) {};
      \draw (x1) to (y3);
      \draw (x2) to (y4);
      \draw (x3) to (y1);
      \draw (x4) to (y2);
      \node at (4.5,2) {$\circlearrowleft$};      
\node at (4.5,1.5) {$\rightleftarrows$};
\node at (4.5,1) {$\circlearrowright$};
    \end{scope}
    \begin{scope}[xshift=12cm]
      \draw[dotted] (-0.5,0) -- (3.5,0);
      \draw[dotted] (-0.5,3) -- (3.5,3);
      \node[circle,scale=0.4,draw=black,fill=white] (x1) at (0,0) {};
      \node[circle,scale=0.4,draw=black,fill=white] (x2) at (1,0) {};
      \node[circle,scale=0.4,draw=black,fill=black] (x3) at (2,0) {};
      \node[circle,scale=0.4,draw=black,fill=black] (x4) at (3,0) {};
      \node[circle,scale=0.4,draw=black,fill=white] (y1) at (0,3) {};
      \node[circle,scale=0.4,draw=black,fill=white] (y2) at (1,3) {};
      \node[circle,scale=0.4,draw=black,fill=black] (y3) at (2,3) {};
      \node[circle,scale=0.4,draw=black,fill=black] (y4) at (3,3) {};
      \draw (x1) -- ++(0,1) -| (x4);
      \draw (y1) -- ++(0,-1) -| (y4);
      \draw (x2) to (y2);
      \draw (x3) to (y3);
    \end{scope}
    \end{tikzpicture}
  \end{gather*}
           So, suppose $v>1$. By Part~\ref{item:arithmetics_generators_of_S_0-item_2} holds $p_{c,v'}\in\langle p_{c,v}\rangle$ and thus
         \begin{align*}
           q_{c,v'}\eqpd\mathrm {Id}(c^{\otimes(v-v')})\otimes p_{c,v'}\otimes \mathrm {Id}(\overline c^{\otimes(v-v')})\in \langle p_{c,v}\rangle
         \end{align*}
         for all $v'\in \mathbb N$ with $v'< v$.  The identity
         \begin{align*}
           {p_{\overline{c},v}}^{\dagger}=p_{c,v}q_{c,v-1}\ldots q_{c,2}q_{c,1}
         \end{align*}
         then proves Claim~\ref{item:arithmetics_generators_of_S_0-item_1}.
        \item        In order to prove Part~\ref{item:arithmetics_generators_of_S_0-item_3}, let $v'\in \pint$ satisfy $v'\leq v$, and let $p$ be a connected and dualizable bracket with a lower row $S$ starting with color $c\in \{\circ,\bullet\}$, with $2v'+2$ points and exactly one turn in $S$. The partition $p_{c,v'}$ as above is an element of  $\langle\PartBracketBBvWvW\rangle$ by Parts~\ref{item:arithmetics_generators_of_S_0-item_2} and \ref{item:arithmetics_generators_of_S_0-item_1}. If $p\neq p_{c,v'}$, the partitions $p_{c,v'}$ and $p$ differ by the fact that $S$ has $m\in \pint$ proper subsectors in $p$, whereas such don't exist in $p_{c,v'}$. Denote by $v_1,\ldots,v_m\in \pint$ with $v_i<v'$ for every $i\in \{1,\ldots,m\}$ the numbers such that $S$ has a proper subsector with $2v_i+2$ elements for every $i\in \{1,\ldots,m\}$. Then, again $q_{c,v_i}$ is contained in $\langle \PartBracketBBvWvW\rangle$ by Parts~\ref{item:arithmetics_generators_of_S_0-item_2} and \ref{item:arithmetics_generators_of_S_0-item_1} for every $i\in \{1,\ldots,m\}$. And from $p=p_{c,v'}q_{c,v_m}\ldots q_{c,v_2}q_{c,v_1}$ follows Claim~\ref{item:arithmetics_generators_of_S_0-item_3}.                    \qedhere
        \end{enumerate}
        
      \end{proof}

By combining Proposition~\ref{proposition:generation_of_categories_by_residual_brackets} and the preceding two lemmata, we show Part~\ref{item:main_1-3} of Main Theorem~\ref{theorem:main_1}.
\begin{proposition}
  \label{proposition:generators_of_the_categories_S_w}
  \begin{enumerate}[label=(\alph*)]
  \item \label{item:generators_of_the_categories_S_w-item_1}
    It holds $\mc S_0=\langle \PartBracketBBvWvW,\PartHalfLibWBW\mid v\in \pint\rangle$.
  \item \label{item:generators_of_the_categories_S_w-item_2}
    For every $w\in \pint$ holds $\mc S_w=\langle \PartBracketBBwW\rangle$.
  \end{enumerate}
\end{proposition}
\begin{proof}
  \begin{enumerate}[wide,label=(\alph*)]
  \item Writing $\mc C_0\eqpd\langle \PartBracketBBvWvW,\PartHalfLibWBW\mid v\in \pint\rangle$, it suffices to prove $\mc S_0\cap \resbr\subseteq \mc C_0$ by Proposition~\ref{proposition:generation_of_categories_by_residual_brackets}. So, let $p$ be a residual bracket of $\mc S_0$ and $S$ its lower row. If $p$ was residual of the first kind, then, in particular, $\inte(S)$ would be non-empty and free of turns. It would follow $\sigma_p(S)\neq 0$, violating the assumption $p\in \mc S_0$. Hence, we infer that $p$ must be residual of the second kind. Especially, $\inte(S)$ contains precisely one turn. If this one turn in $\inte(S)$ is not the only turn in all of $S$, then all further turns of $S$ must intersect $\partial S$. If so, then each point of $\partial S$ must be element of a turn as $p$ is verticolor-reflexive. It follows that in all of $S$ there exist either just one turn or three turns.
    Equivalently, 
    \begin{align*}
p=\PartBracketBWBW,\quad p=\PartBracketWBWB,\quad p=p',\quad p=\trmw(p')\quad\text{or} \quad p=\trms(p')      
    \end{align*}
    for a residual bracket of the second kind $p'$ whose lower row contains precisely one turn.
    In the first two cases, Lemma~\hyperref[item:bracket_operations-item_6]{\ref*{lemma:bracket_operations}~\ref*{item:bracket_operations-item_6}} shows $p\in \mc C_0$. In the remaining three,
Lemma~\hyperref[item:arithmetics_generators_of_S_0-item_2]{\ref*{lemma:arithmetics_generators_of_S_0}~\ref*{item:arithmetics_generators_of_S_0-item_2}} and Lemma~\hyperref[item:arithmetics_generators_of_S_0-item_3]{\ref*{lemma:arithmetics_generators_of_S_0}~\ref*{item:arithmetics_generators_of_S_0-item_3}} prove $p'\in \mc C_0$, from which then $p\in \mc C_0$ follows by  Lemma~\hyperref[item:bracket_operations-item_3]{\ref*{lemma:bracket_operations}~\ref*{item:bracket_operations-item_3}} and Lemma~\hyperref[item:bracket_operations-item_4]{\ref*{lemma:bracket_operations}~\ref*{item:bracket_operations-item_4}}.
  \item
    Again, by Proposition~\ref{proposition:generation_of_categories_by_residual_brackets}, it suffices to show $\mc S_w\cap \resbr\subseteq \mc C_w\eqpd \langle  \PartBracketBBwW\rangle$. So, let $p$ be a residual bracket of $\mc S_w$ and $S$ its lower row. \par
    First, suppose that $p$ is residual of the first kind. Because then $\inte(S)$ is non-empty and free of turns, the defining condition $\sigma_p(S)\in w\integers$ of $\mc S_w$ forces $p$ to be of the form
    \begin{align*}
      p=\brt{c_1}{\overline{ c_1}}{\mathrm{Id}(c_2^{\otimes kw})}
    \end{align*}
    for some colors $c_1,c_2\in \{\circ,\bullet\}$ and some $k\in \pint$. Lemma~\hyperref[item:arithmetics_generators_of_S_w-item_1]{\ref*{lemma:arithmetics_generators_of_S_w}~\ref*{item:arithmetics_generators_of_S_w-item_1}} shows that  $\mc C_w$ contains the bracket  $q\eqpd \brt{c_1}{\overline{ c_1}}{\mathrm{Id}(c_2^{\otimes w})}$. Applying Lemma~\hyperref[item:bracket_operations-item_1]{\ref*{lemma:bracket_operations}~\ref*{item:bracket_operations-item_1}} hence proves $p=q^{\boxtimes k}\in \mc C_w$.\par
Now, suppose that $p$ is residual of the second kind. Recognize that this already implies $p\in \mc S_0$. Then, $p\in \langle \PartBracketBBvWvW,\PartHalfLibWBW\mid v\in \pint\rangle$ by Part~\ref{item:generators_of_the_categories_S_w-item_1}. Hence, we can employ all three assertions of Lemma~\ref{lemma:arithmetics_generators_of_S_w} to infer $p\in \mc C_w$.
    \qedhere
      \end{enumerate}
    \end{proof}
    In the follow-up \cite{MaWe18b} to this article, it is shown that the category $\mc S_0$ is not finitely generated.
    \par
    \section{Classification\texorpdfstring{ above $\mc S_0$\\{[Main Theorem~\ref*{theorem:main_2}]}}{}}
\label{section:classification}
For now, it only remains to prove Main~Theorem~\ref{theorem:main_2}.
\begin{proposition}
		\label{proposition:classification_of_neutral_pair_partitions_above_S_0}
	If $\mathcal{C}\subseteq\mathcal P^{\circ\bullet}_{2,\mathrm{nb}}$ is a category, then either $\mathcal{C}\subsetneq\mathcal{S}_0$ or there exists $w\in\nnint$ such that $\mathcal{C}= \mathcal{S}_w$.
\end{proposition}
\begin{proof}
  Writing $\mc B_c$ for the set of all brackets whose lower row starts with a $c$-colored point, $(\mc C\cap \mc B_c,\boxtimes)$ is a monoid for every $c\in \{\circ,\bullet\}$ by Lemma~\ref{lemma:bracket_operations}. Define the mapping 
		\begin{align*}
			H\colon\;  \mathcal{C}\cap \mc B \longrightarrow \mathbb{Z}, \; p\longmapsto \sigma_p(S_p)=\sigma_p\left(\inte(S_p)\right),
		\end{align*}
		where $S_p$ denotes the lower row of $p$. For each $c\in\{\circ,\bullet\}$, the map $H$ is a monoid homomorphism from $(\mathcal{C}\cap \mc B_c, \boxtimes )$ to $\left(\mathbb{Z},+\right)$. Moreover, $H(\tilde{p})=-H(p)$ for all $p\in\mathcal{C}\cap \mc B_c$. Thirdly, $H( \trmw(w))=H(p)$ for all $p\in\mathcal{C}\cap\mc B_{\overline {c}}$, proving $H(\mathcal{C}\cap \mc B_\bullet)=H(\mathcal{C}\cap\mc B_\circ)$ according to Lemma \hyperref[item:bracket_operations-item_3]{\ref*{lemma:bracket_operations}~\ref*{item:bracket_operations-item_3}}. Therefore, $H(\mathcal{C}\cap \mc B)$ is an additive subgroup of $\mathbb{Z}$. Hence, $H(\mc {C}\cap \mc B)=w\mathbb{Z}$ for some $w\in \mathbb N_0$.
                \par 
                For each sector $S$ in a partition $p\in\mathcal{C}$, the bracket $B(p,S)$ associated with $(p,S)$ is an element of $\mathcal{C}\cap \mc B$ and $\sigma_p(S)=H(B(p,S))$. That proves $\mathcal{C}\subseteq \mathcal{S}_w$.
                \par
                Now, assume $\mc C\not\subseteq \mc S_0$.  We then find a bracket $p\in\mathcal{C}\cap \mc B_\bullet$ with lower row $S$ and $\sigma_p(S)=w\neq 0$. We may assume that $\mathrm{int}(S)$ contains no turns of $p$ since erasing a turn would leave $\sigma_p(S)$ unchanged. Hence, all points in $\inte (S)$ must belong to through blocks,  meaning $p=\PartBracketBBwW$ or $p=\PartBracketBWwW$, which, according to Proposition~\ref{proposition:generators_of_the_categories_S_w} and Lemma~\hyperref[item:arithmetics_generators_of_S_w-item_1]{\ref*{lemma:arithmetics_generators_of_S_w}~\ref*{item:arithmetics_generators_of_S_w-item_1}}, proves $\mathcal{S}_w\subseteq \mathcal{C}$.\qedhere
              \end{proof}
              In \cite{MaWe18b} we will determine all subcategories of $\mc S_0$, thus completing the full classification of all categories of pair partitions with neutral blocks. But already the consequence of Propositions~\ref{proposition:classification_of_neutral_pair_partitions_above_S_0} and~\ref{proposition:set_relationships_between_the_categories} that for every category $\mc C\subseteq\Cppnb$ holds either $\mc C\subseteq \mc S_0$ or $\mc S_0\subseteq \mc C$ reveals a remarkable fact about the categories of pair partitions with neutral blocks.
	\section{Concluding Remarks}
\label{section:concluding_remarks}

\subsection{Comparison with the Previous Research \texorpdfstring{on Half-Liberations of $U_n$}{}}

Our results are consistent with the previous research on this topic:\par

\begin{enumerate}
  \item In \cite[Definition~5.5]{BhDADa11} and \cite[Definition~2.8]{BhDADaDa14}, Bhowmick, D'Andrea, Das and Dabrowski used the relations represented by 
		
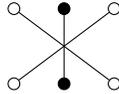
\begin{gather*}
		\begin{tikzpicture}[scale=0.666,baseline=(current  bounding  box.center)]
		\node [circle, scale=0.4, draw=black, fill=white] (x1-1) at (0,0) {};
		\node [circle, scale=0.4, draw=black, fill=black] (x2-1) at (1,0) {};
		\node [circle, scale=0.4, draw=black, fill=white] (x3-1) at (2,0) {};
		\node [circle, scale=0.4, draw=black, fill=white] (x1-2) at (0,1.5) {};
		\node [circle, scale=0.4, draw=black, fill=black] (x2-2) at (1,1.5) {};
		\node [circle, scale=0.4, draw=black, fill=white] (x3-2) at (2,1.5) {};
		\draw (x1-1) -- (x3-2);
		\draw (x2-1) -- (x2-2);
		\draw (x3-1) -- (x1-2);
		\end{tikzpicture}
		\end{gather*}
		to define an algebra $A_u^*(n)$. For the quantum group associated with $A_u^*(n)$, Banica and Bichon later used the symbol $U_n^\times$ in \cite[Definition~3.2 (3)]{BaBi17}.\par Proposition~\hyperref[item:generators_of_the_categories_S_w-item_1]{\ref*{proposition:generators_of_the_categories_S_w}~\ref*{item:generators_of_the_categories_S_w-item_1}} shows that $\langle\PartHalfLibWBW\rangle$ is a subcategory of $\mc S_0$. We will put the category $\langle\PartHalfLibWBW\rangle$ into a broader context in \cite{MaWe18b}.
\item Bichon and Dubois-Violette defined in \cite[Example 4.10]{BiDu13} a quantum group with algebra $A_u^{**}(n)$. The partition 
		
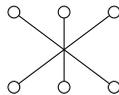
\begin{gather*}
		\begin{tikzpicture}[scale=0.666,baseline=(current  bounding  box.center)]
		\node [circle, scale=0.4, draw=black, fill=white] (x1-1) at (0,0) {};
		\node [circle, scale=0.4, draw=black, fill=white] (x2-1) at (1,0) {};
		\node [circle, scale=0.4, draw=black, fill=white] (x3-1) at (2,0) {};
		\node [circle, scale=0.4, draw=black, fill=white] (x1-2) at (0,1.5) {};
		\node [circle, scale=0.4, draw=black, fill=white] (x2-2) at (1,1.5) {};
		\node [circle, scale=0.4, draw=black, fill=white] (x3-2) at (2,1.5) {};
		\draw (x1-1) -- (x3-2);
		\draw (x2-1) -- (x2-2);
		\draw (x3-1) -- (x1-2);
		\end{tikzpicture}
		\end{gather*}
		generates its intertwiner spaces.  Later, Banica and Bichon denoted the corresponding quantum group by $U_n^{**}$ in \cite{BaBi17b} and \cite{BaBi17}. 
                \par
                	And in \cite[Definition~7.1]{BaBi17b}, Banica and Bichon introduced a series of quantum groups referred to as $\left( U_{n,k} \right)_{k\in\mathbb{N}}$ in \cite{BaBi17b} and later in \cite{BaBi17}. For all $k\in\mathbb{N}$, the \enquote{$k$-half-liberated unitary quantum group} $U_{n,k}$ has its intertwiner spaces generated by
		
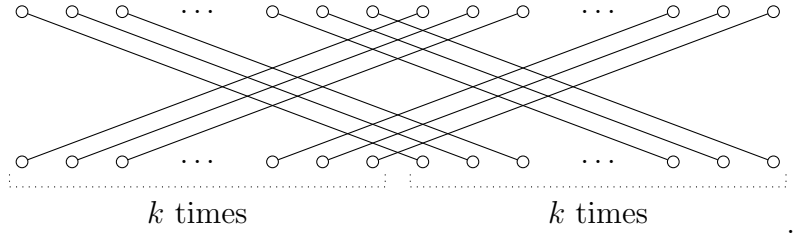
\begin{gather*}
		\begin{tikzpicture}[scale=0.666]
		\node [circle, scale=0.4, draw=black, fill=white] (x1-1) at (0,0) {};
		\node [circle, scale=0.4, draw=black, fill=white] (x2-1) at (1,0) {};
		\node [circle, scale=0.4, draw=black, fill=white] (x3-1) at (2,0) {};						
		\node [circle, scale=0.4, draw=black, fill=white] (y1-1) at (5,0) {};
		\node [circle, scale=0.4, draw=black, fill=white] (y2-1) at (6,0) {};
		\node [circle, scale=0.4, draw=black, fill=white] (y3-1) at (7,0) {};
		\node [circle, scale=0.4, draw=black, fill=white] (y4-1) at (8,0) {};
		\node [circle, scale=0.4, draw=black, fill=white] (y5-1) at (9,0) {};
		\node [circle, scale=0.4, draw=black, fill=white] (y6-1) at (10,0) {};															
		\node [circle, scale=0.4, draw=black, fill=white] (z1-1) at (13,0) {};															
		\node [circle, scale=0.4, draw=black, fill=white] (z2-1) at (14,0) {};															
		\node [circle, scale=0.4, draw=black, fill=white] (z3-1) at (15,0) {};					
		\begin{scope}[yshift=3cm]			
		\node [circle, scale=0.4, draw=black, fill=white] (x1-2) at (0,0) {};
		\node [circle, scale=0.4, draw=black, fill=white] (x2-2) at (1,0) {};
		\node [circle, scale=0.4, draw=black, fill=white] (x3-2) at (2,0) {};						
		\node [circle, scale=0.4, draw=black, fill=white] (y1-2) at (5,0) {};
		\node [circle, scale=0.4, draw=black, fill=white] (y2-2) at (6,0) {};
		\node [circle, scale=0.4, draw=black, fill=white] (y3-2) at (7,0) {};
		\node [circle, scale=0.4, draw=black, fill=white] (y4-2) at (8,0) {};
		\node [circle, scale=0.4, draw=black, fill=white] (y5-2) at (9,0) {};
		\node [circle, scale=0.4, draw=black, fill=white] (y6-2) at (10,0) {};															
		\node [circle, scale=0.4, draw=black, fill=white] (z1-2) at (13,0) {};															
		\node [circle, scale=0.4, draw=black, fill=white] (z2-2) at (14,0) {};															
		\node [circle, scale=0.4, draw=black, fill=white] (z3-2) at (15,0) {};						
		\end{scope}						
		\path (x3-1) edge[draw=none] node {$\ldots$} (y1-1);
		\path (y6-1) edge[draw=none] node {$\ldots$} (z1-1);				
		\path (x3-2) edge[draw=none] node {$\ldots$} (y1-2);
		\path (y6-2) edge[draw=none] node {$\ldots$} (z1-2);								
		\draw (x1-1) -- (y4-2);
		\draw (x2-1) -- (y5-2);
		\draw (x3-1) -- (y6-2);								
		\draw (x1-2) -- (y4-1);
		\draw (x2-2) -- (y5-1);
		\draw (x3-2) -- (y6-1);								
		\draw (z1-1) -- (y1-2);
		\draw (z2-1) -- (y2-2);
		\draw (z3-1) -- (y3-2);								
		\draw (z1-2) -- (y1-1);
		\draw (z2-2) -- (y2-1);
		\draw (z3-2) -- (y3-1);												
		\draw [dotted] (-0.25,-0.25) -- ++(0,-0.25) -- (7.25,-0.5) -- ++ (0,0.25);
		\draw [dotted, xshift=8cm] (-0.25,-0.25) -- ++(0,-0.25) -- (7.25,-0.5) -- ++ (0,0.25);
		\node at (3.5,-1) {$k$ times};
		\node at (11.5,-1) {$k$ times};				
		\end{tikzpicture}.
              \end{gather*}
            Especially, $U_{n,1}$ corresponds to $U_n\sim\langle\PartCrossWW\rangle$ and they show that $U_{n,2}$ is identical to $U_{n}^{\ast\ast}\sim \langle\PartHalfLibWWW\rangle$. In general, by composing the above generator of $U_{n,k}$ with elements of $\langle\emptyset\rangle$ in the following way

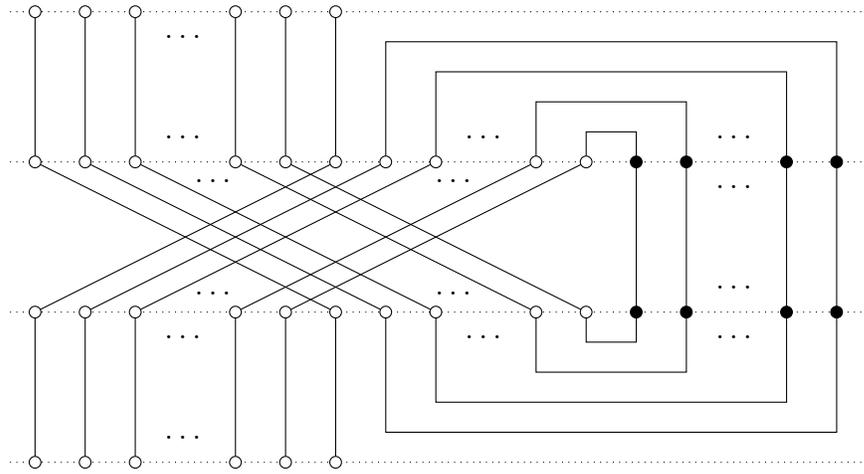
\begin{center}
  \begin{tikzpicture}[scale=0.666]
    \def\dx{1}
    \def\dy{1}
    \def\dd{0.5}
    \def\xzero{0}
    \def\yzero{0}    
    \def\yh{3/5}
    
    \draw [dotted] ({\xzero-\dd},{0*\dy+\yzero}) to ({\xzero+16*\dx+\dd},{0*\dy+\yzero});
    \draw [dotted] ({\xzero-\dd},{3*\dy+\yzero}) to ({\xzero+16*\dx+\dd},{3*\dy+\yzero});
    \draw [dotted] ({\xzero-\dd},{6*\dy+\yzero}) to ({\xzero+16*\dx+\dd},{6*\dy+\yzero});
    \draw [dotted] ({\xzero-\dd},{9*\dy+\yzero}) to ({\xzero+16*\dx+\dd},{9*\dy+\yzero});

    \node [scale=0.4, circle, draw=black, fill=white] (n1-1) at ({\xzero+0*\dx},{\yzero+0*\dy}) {};
    \node [scale=0.4, circle, draw=black, fill=white] (n1-2) at ({\xzero+1*\dx},{\yzero+0*\dy}) {};
    \node [scale=0.4, circle, draw=black, fill=white] (n1-3) at ({\xzero+2*\dx},{\yzero+0*\dy}) {};
    \node [scale=0.4, circle, draw=black, fill=white] (n1-5) at ({\xzero+4*\dx},{\yzero+0*\dy}) {};
    \node [scale=0.4, circle, draw=black, fill=white] (n1-6) at ({\xzero+5*\dx},{\yzero+0*\dy}) {};
    \node [scale=0.4, circle, draw=black, fill=white] (n1-7) at ({\xzero+6*\dx},{\yzero+0*\dy}) {};

    \node [scale=0.4, circle, draw=black, fill=white] (n2-1) at ({\xzero+0*\dx},{\yzero+3*\dy}) {};
    \node [scale=0.4, circle, draw=black, fill=white] (n2-2) at ({\xzero+1*\dx},{\yzero+3*\dy}) {};
    \node [scale=0.4, circle, draw=black, fill=white] (n2-3) at ({\xzero+2*\dx},{\yzero+3*\dy}) {};
    \node [scale=0.4, circle, draw=black, fill=white] (n2-5) at ({\xzero+4*\dx},{\yzero+3*\dy}) {};
    \node [scale=0.4, circle, draw=black, fill=white] (n2-6) at ({\xzero+5*\dx},{\yzero+3*\dy}) {};
    \node [scale=0.4, circle, draw=black, fill=white] (n2-7) at ({\xzero+6*\dx},{\yzero+3*\dy}) {};
    \node [scale=0.4, circle, draw=black, fill=white] (n2-8) at ({\xzero+7*\dx},{\yzero+3*\dy}) {};
    \node [scale=0.4, circle, draw=black, fill=white] (n2-9) at ({\xzero+8*\dx},{\yzero+3*\dy}) {};
    \node [scale=0.4, circle, draw=black, fill=white] (n2-11) at ({\xzero+10*\dx},{\yzero+3*\dy}) {};
    \node [scale=0.4, circle, draw=black, fill=white] (n2-12) at ({\xzero+11*\dx},{\yzero+3*\dy}) {};
    \node [scale=0.4, circle, draw=black, fill=black] (n2-13) at ({\xzero+12*\dx},{\yzero+3*\dy}) {};
    \node [scale=0.4, circle, draw=black, fill=black] (n2-14) at ({\xzero+13*\dx},{\yzero+3*\dy}) {};
    \node [scale=0.4, circle, draw=black, fill=black] (n2-16) at ({\xzero+15*\dx},{\yzero+3*\dy}) {};
    \node [scale=0.4, circle, draw=black, fill=black] (n2-17) at ({\xzero+16*\dx},{\yzero+3*\dy}) {};

    \node [scale=0.4, circle, draw=black, fill=white] (n3-1) at ({\xzero+0*\dx},{\yzero+6*\dy}) {};
    \node [scale=0.4, circle, draw=black, fill=white] (n3-2) at ({\xzero+1*\dx},{\yzero+6*\dy}) {};
    \node [scale=0.4, circle, draw=black, fill=white] (n3-3) at ({\xzero+2*\dx},{\yzero+6*\dy}) {};
    \node [scale=0.4, circle, draw=black, fill=white] (n3-5) at ({\xzero+4*\dx},{\yzero+6*\dy}) {};
    \node [scale=0.4, circle, draw=black, fill=white] (n3-6) at ({\xzero+5*\dx},{\yzero+6*\dy}) {};
    \node [scale=0.4, circle, draw=black, fill=white] (n3-7) at ({\xzero+6*\dx},{\yzero+6*\dy}) {};
    \node [scale=0.4, circle, draw=black, fill=white] (n3-8) at ({\xzero+7*\dx},{\yzero+6*\dy}) {};
    \node [scale=0.4, circle, draw=black, fill=white] (n3-9) at ({\xzero+8*\dx},{\yzero+6*\dy}) {};
    \node [scale=0.4, circle, draw=black, fill=white] (n3-11) at ({\xzero+10*\dx},{\yzero+6*\dy}) {};
    \node [scale=0.4, circle, draw=black, fill=white] (n3-12) at ({\xzero+11*\dx},{\yzero+6*\dy}) {};
    \node [scale=0.4, circle, draw=black, fill=black] (n3-13) at ({\xzero+12*\dx},{\yzero+6*\dy}) {};
    \node [scale=0.4, circle, draw=black, fill=black] (n3-14) at ({\xzero+13*\dx},{\yzero+6*\dy}) {};
    \node [scale=0.4, circle, draw=black, fill=black] (n3-16) at ({\xzero+15*\dx},{\yzero+6*\dy}) {};
    \node [scale=0.4, circle, draw=black, fill=black] (n3-17) at ({\xzero+16*\dx},{\yzero+6*\dy}) {};    

    \node [scale=0.4, circle, draw=black, fill=white] (n4-1) at ({\xzero+0*\dx},{\yzero+9*\dy}) {};
    \node [scale=0.4, circle, draw=black, fill=white] (n4-2) at ({\xzero+1*\dx},{\yzero+9*\dy}) {};
    \node [scale=0.4, circle, draw=black, fill=white] (n4-3) at ({\xzero+2*\dx},{\yzero+9*\dy}) {};
    \node [scale=0.4, circle, draw=black, fill=white] (n4-5) at ({\xzero+4*\dx},{\yzero+9*\dy}) {};
    \node [scale=0.4, circle, draw=black, fill=white] (n4-6) at ({\xzero+5*\dx},{\yzero+9*\dy}) {};
    \node [scale=0.4, circle, draw=black, fill=white] (n4-7) at ({\xzero+6*\dx},{\yzero+9*\dy}) {};

    \draw (n1-1) to (n2-1);
    \draw (n1-2) to (n2-2);
    \draw (n1-3) to (n2-3);
    \draw (n1-5) to (n2-5);
    \draw (n1-6) to (n2-6);
    \draw (n1-7) to (n2-7);
    \draw (n2-8) -- ++ (0,{-4*\yh}) -| (n2-17);
    \draw (n2-9) -- ++ (0,{-3*\yh}) -| (n2-16);
    \draw (n2-11) -- ++ (0,{-2*\yh}) -| (n2-14);
    \draw (n2-12) -- ++ (0,{-1*\yh}) -| (n2-13);

    \draw (n3-1) to (n4-1);
    \draw (n3-2) to (n4-2);
    \draw (n3-3) to (n4-3);
    \draw (n3-5) to (n4-5);
    \draw (n3-6) to (n4-6);
    \draw (n3-7) to (n4-7);    
    \draw (n3-8) -- ++ (0,{4*\yh}) -| (n3-17);
    \draw (n3-9) -- ++ (0,{3*\yh}) -| (n3-16);
    \draw (n3-11) -- ++ (0,{2*\yh}) -| (n3-14);
    \draw (n3-12) -- ++ (0,{1*\yh}) -| (n3-13);

    \draw (n2-1) to (n3-7);
    \draw (n2-2) to (n3-8);
    \draw (n2-3) to (n3-9);
    \draw (n2-5) to (n3-11);
    \draw (n2-6) to (n3-12);
    
    \draw (n3-1) to (n2-7);
    \draw (n3-2) to (n2-8);
    \draw (n3-3) to (n2-9);
    \draw (n3-5) to (n2-11);
    \draw (n3-6) to (n2-12);

    \draw (n2-13) to (n3-13);
    \draw (n2-14) to (n3-14);
    \draw (n2-16) to (n3-16);
    \draw (n2-17) to (n3-17);
    
    \node at ({\xzero+3*\dx},{\yzero+0.5*\dy}) {\ldots};
    \node at ({\xzero+3*\dx},{\yzero+2.5*\dy}) {\ldots};
    \node at ({\xzero+3*\dx},{\yzero+6.5*\dy}) {\ldots};
    \node at ({\xzero+3*\dx},{\yzero+8.5*\dy}) {\ldots};
    \node at ({\xzero+3*\dx+0.6},{\yzero+3.375*\dy}) {\ldots};
    \node at ({\xzero+3*\dx+0.6},{\yzero+(6-0.375)*\dy}) {\ldots};
    \node at ({\xzero+9*\dx-0.6},{\yzero+3.375*\dy}) {\ldots};
    \node at ({\xzero+9*\dx-0.6},{\yzero+(6-0.375)*\dy}) {\ldots};
    \node at ({\xzero+14*\dx},{\yzero+2.5*\dy}) {\ldots};
    \node at ({\xzero+14*\dx},{\yzero+3.5*\dy}) {\ldots};
    \node at ({\xzero+14*\dx},{\yzero+5.5*\dy}) {\ldots};
    \node at ({\xzero+14*\dx},{\yzero+6.5*\dy}) {\ldots};
    \node at ({\xzero+9*\dx},{\yzero+2.5*\dy}) {\ldots};
    \node at ({\xzero+9*\dx},{\yzero+6.5*\dy}) {\ldots};
  \end{tikzpicture},
\end{center}
it is seen that the category of $U_{n,k}$ contains the partition
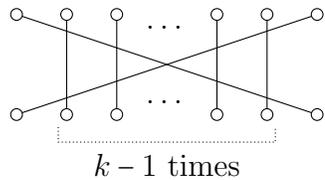
\begin{center}
  \begin{tikzpicture}[scale=0.666, baseline=0cm]
    \node [scale=0.4, circle, draw=black, fill=white] (a0) at (0,0) {};
    \node [scale=0.4, circle, draw=black, fill=white] (a1) at (1,0) {};
    \node [scale=0.4, circle, draw=black, fill=white] (a2) at (2,0) {};    
    \node [scale=0.4, circle, draw=black, fill=white] (a4) at (4,0) {};
    \node [scale=0.4, circle, draw=black, fill=white] (a5) at (5,0) {};
    \node [scale=0.4, circle, draw=black, fill=white] (a6) at (6,0) {};    
    \node [scale=0.4, circle, draw=black, fill=white] (b0) at (0,2) {};
    \node [scale=0.4, circle, draw=black, fill=white] (b1) at (1,2) {};
    \node [scale=0.4, circle, draw=black, fill=white] (b2) at (2,2) {};    
    \node [scale=0.4, circle, draw=black, fill=white] (b4) at (4,2) {};
    \node [scale=0.4, circle, draw=black, fill=white] (b5) at (5,2) {};
    \node [scale=0.4, circle, draw=black, fill=white] (b6) at (6,2) {};    
    \draw (a0) to (b6);
    \draw (a6) to (b0);
    \draw (a1) to (b1);
    \draw (a2) to (b2);
    \draw (a4) to (b4);
    \draw (a5) to (b5);
    \node at (3,0.25) {\ldots};
    \node at (3,1.75) {\ldots};
    \draw[densely dotted] (0.833,-0.25) --++(0,-0.3) -|(5.166,-0.25);
    \node at (3,-1) {$k-1$ times};
  \end{tikzpicture}.
\end{center}
Conversely, by composing this partition $k$ times with itself with successively increasing offset, we can recover the original generator of $U_{n,k}$:
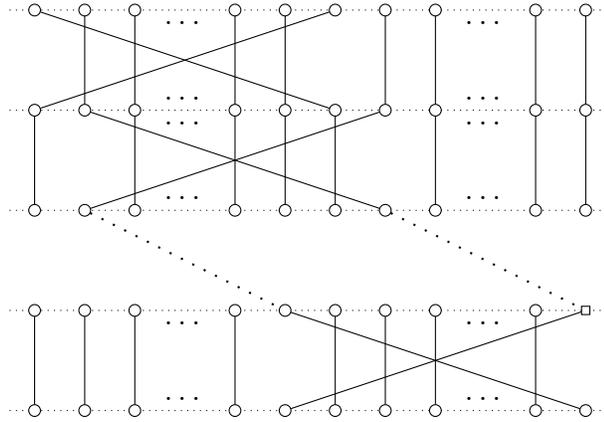
  \begin{center}
  \begin{tikzpicture}[scale=0.666]
    \def\dx{1}
    \def\dy{1}
    \def\dd{0.5}
    \def\xzero{0}
    \def\yzero{0}    

    \draw [dotted] ({\xzero-\dd},{0*\dy+\yzero}) to ({\xzero+11*\dx+\dd},{0*\dy+\yzero});
    \draw [dotted] ({\xzero-\dd},{2*\dy+\yzero}) to ({\xzero+11*\dx+\dd},{2*\dy+\yzero});
    \draw [dotted] ({\xzero-\dd},{4*\dy+\yzero}) to ({\xzero+11*\dx+\dd},{4*\dy+\yzero});
    \draw [dotted] ({\xzero-\dd},{6*\dy+\yzero}) to ({\xzero+11*\dx+\dd},{6*\dy+\yzero});
    \draw [dotted] ({\xzero-\dd},{8*\dy+\yzero}) to ({\xzero+11*\dx+\dd},{8*\dy+\yzero});

    \node [scale=0.4, circle, draw=black, fill=white] (n5-1) at ({\xzero+0*\dx},{\yzero+8*\dy}) {};
    \node [scale=0.4, circle, draw=black, fill=white] (n5-2) at ({\xzero+1*\dx},{\yzero+8*\dy}) {};
    \node [scale=0.4, circle, draw=black, fill=white] (n5-3) at ({\xzero+2*\dx},{\yzero+8*\dy}) {};
    \node [scale=0.4, circle, draw=black, fill=white] (n5-5) at ({\xzero+4*\dx},{\yzero+8*\dy}) {};
    \node [scale=0.4, circle, draw=black, fill=white] (n5-6) at ({\xzero+5*\dx},{\yzero+8*\dy}) {};
    \node [scale=0.4, circle, draw=black, fill=white] (n5-7) at ({\xzero+6*\dx},{\yzero+8*\dy}) {};
    \node [scale=0.4, circle, draw=black, fill=white] (n5-8) at ({\xzero+7*\dx},{\yzero+8*\dy}) {};
    \node [scale=0.4, circle, draw=black, fill=white] (n5-9) at ({\xzero+8*\dx},{\yzero+8*\dy}) {};
    \node [scale=0.4, circle, draw=black, fill=white] (n5-11) at ({\xzero+10*\dx},{\yzero+8*\dy}) {};
    \node [scale=0.4, circle, draw=black, fill=white] (n5-12) at ({\xzero+11*\dx},{\yzero+8*\dy}) {};

    \node [scale=0.4, circle, draw=black, fill=white] (n4-1) at ({\xzero+0*\dx},{\yzero+6*\dy}) {};
    \node [scale=0.4, circle, draw=black, fill=white] (n4-2) at ({\xzero+1*\dx},{\yzero+6*\dy}) {};
    \node [scale=0.4, circle, draw=black, fill=white] (n4-3) at ({\xzero+2*\dx},{\yzero+6*\dy}) {};
    \node [scale=0.4, circle, draw=black, fill=white] (n4-5) at ({\xzero+4*\dx},{\yzero+6*\dy}) {};
    \node [scale=0.4, circle, draw=black, fill=white] (n4-6) at ({\xzero+5*\dx},{\yzero+6*\dy}) {};
    \node [scale=0.4, circle, draw=black, fill=white] (n4-7) at ({\xzero+6*\dx},{\yzero+6*\dy}) {};
    \node [scale=0.4, circle, draw=black, fill=white] (n4-8) at ({\xzero+7*\dx},{\yzero+6*\dy}) {};
    \node [scale=0.4, circle, draw=black, fill=white] (n4-9) at ({\xzero+8*\dx},{\yzero+6*\dy}) {};
    \node [scale=0.4, circle, draw=black, fill=white] (n4-11) at ({\xzero+10*\dx},{\yzero+6*\dy}) {};
    \node [scale=0.4, circle, draw=black, fill=white] (n4-12) at ({\xzero+11*\dx},{\yzero+6*\dy}) {};

    \node [scale=0.4, circle, draw=black, fill=white] (n3-1) at ({\xzero+0*\dx},{\yzero+4*\dy}) {};
    \node [scale=0.4, circle, draw=black, fill=white] (n3-2) at ({\xzero+1*\dx},{\yzero+4*\dy}) {};
    \node [scale=0.4, circle, draw=black, fill=white] (n3-3) at ({\xzero+2*\dx},{\yzero+4*\dy}) {};
    \node [scale=0.4, circle, draw=black, fill=white] (n3-5) at ({\xzero+4*\dx},{\yzero+4*\dy}) {};
    \node [scale=0.4, circle, draw=black, fill=white] (n3-6) at ({\xzero+5*\dx},{\yzero+4*\dy}) {};
    \node [scale=0.4, circle, draw=black, fill=white] (n3-7) at ({\xzero+6*\dx},{\yzero+4*\dy}) {};
    \node [scale=0.4, circle, draw=black, fill=white] (n3-8) at ({\xzero+7*\dx},{\yzero+4*\dy}) {};
    \node [scale=0.4, circle, draw=black, fill=white] (n3-9) at ({\xzero+8*\dx},{\yzero+4*\dy}) {};
    \node [scale=0.4, circle, draw=black, fill=white] (n3-11) at ({\xzero+10*\dx},{\yzero+4*\dy}) {};
    \node [scale=0.4, circle, draw=black, fill=white] (n3-12) at ({\xzero+11*\dx},{\yzero+4*\dy}) {};

    \node [scale=0.4, circle, draw=black, fill=white] (n2-1) at ({\xzero+0*\dx},{\yzero+2*\dy}) {};
    \node [scale=0.4, circle, draw=black, fill=white] (n2-2) at ({\xzero+1*\dx},{\yzero+2*\dy}) {};
    \node [scale=0.4, circle, draw=black, fill=white] (n2-3) at ({\xzero+2*\dx},{\yzero+2*\dy}) {};
    \node [scale=0.4, circle, draw=black, fill=white] (n2-5) at ({\xzero+4*\dx},{\yzero+2*\dy}) {};
    \node [scale=0.4, circle, draw=black, fill=white] (n2-6) at ({\xzero+5*\dx},{\yzero+2*\dy}) {};
    \node [scale=0.4, circle, draw=black, fill=white] (n2-7) at ({\xzero+6*\dx},{\yzero+2*\dy}) {};
    \node [scale=0.4, circle, draw=black, fill=white] (n2-8) at ({\xzero+7*\dx},{\yzero+2*\dy}) {};
    \node [scale=0.4, circle, draw=black, fill=white] (n2-9) at ({\xzero+8*\dx},{\yzero+2*\dy}) {};
    \node [scale=0.4, circle, draw=black, fill=white] (n2-11) at ({\xzero+10*\dx},{\yzero+2*\dy}) {};
    \node [scale=0.4, draw=black, fill=white] (n2-12) at ({\xzero+11*\dx},{\yzero+2*\dy}) {};

    \node [scale=0.4, circle, draw=black, fill=white] (n1-1) at ({\xzero+0*\dx},{\yzero+0*\dy}) {};
    \node [scale=0.4, circle, draw=black, fill=white] (n1-2) at ({\xzero+1*\dx},{\yzero+0*\dy}) {};
    \node [scale=0.4, circle, draw=black, fill=white] (n1-3) at ({\xzero+2*\dx},{\yzero+0*\dy}) {};
    \node [scale=0.4, circle, draw=black, fill=white] (n1-5) at ({\xzero+4*\dx},{\yzero+0*\dy}) {};
    \node [scale=0.4, circle, draw=black, fill=white] (n1-6) at ({\xzero+5*\dx},{\yzero+0*\dy}) {};
    \node [scale=0.4, circle, draw=black, fill=white] (n1-7) at ({\xzero+6*\dx},{\yzero+0*\dy}) {};
    \node [scale=0.4, circle, draw=black, fill=white] (n1-8) at ({\xzero+7*\dx},{\yzero+0*\dy}) {};
    \node [scale=0.4, circle, draw=black, fill=white] (n1-9) at ({\xzero+8*\dx},{\yzero+0*\dy}) {};
    \node [scale=0.4, circle, draw=black, fill=white] (n1-11) at ({\xzero+10*\dx},{\yzero+0*\dy}) {};
    \node [scale=0.4, circle, draw=black, fill=white] (n1-12) at ({\xzero+11*\dx},{\yzero+0*\dy}) {};

    \draw (n4-1) to (n5-7);
    \draw (n4-2) to (n5-2);  
    \draw (n4-3) to (n5-3);
    \draw (n4-5) to (n5-5);
    \draw (n4-6) to (n5-6);    
    \draw (n4-7) to (n5-1);
    \draw (n4-8) to (n5-8);
    \draw (n4-9) to (n5-9);
    \draw (n4-11) to (n5-11);
    \draw (n4-12) to (n5-12);        

    \draw (n3-1) to (n4-1);    
    \draw (n3-2) to (n4-8);
    \draw (n3-3) to (n4-3);
    \draw (n3-5) to (n4-5);
    \draw (n3-6) to (n4-6);    
    \draw (n3-7) to (n4-7);    
    \draw (n3-8) to (n4-2);
    \draw (n3-9) to (n4-9);
    \draw (n3-11) to (n4-11);
    \draw (n3-12) to (n4-12);            
    
    \draw (n1-1) to (n2-1);
    \draw (n1-2) to (n2-2);
    \draw (n1-3) to (n2-3);
    \draw (n1-5) to (n2-5);        
    \draw (n1-6) to (n2-12);
    \draw (n1-7) to (n2-7);
    \draw (n1-8) to (n2-8);
    \draw (n1-9) to (n2-9);
    \draw (n1-11) to (n2-11);        
    \draw (n1-12) to (n2-6);

    \draw[loosely dotted, thick] (n3-2) to (n2-6);
    \draw[loosely dotted, thick] (n3-8) to (n2-12);
    
    \node at ({\xzero+3*\dx},{\yzero+6.25*\dy}) {\ldots};
    \node at ({\xzero+3*\dx},{\yzero+7.75*\dy}) {\ldots};
    \node at ({\xzero+9*\dx},{\yzero+6.25*\dy}) {\ldots};
    \node at ({\xzero+9*\dx},{\yzero+7.75*\dy}) {\ldots};
    
    \node at ({\xzero+3*\dx},{\yzero+4.25*\dy}) {\ldots};
    \node at ({\xzero+3*\dx},{\yzero+5.75*\dy}) {\ldots};
    \node at ({\xzero+9*\dx},{\yzero+4.25*\dy}) {\ldots};
    \node at ({\xzero+9*\dx},{\yzero+5.75*\dy}) {\ldots};
    
    \node at ({\xzero+3*\dx},{\yzero+0.25*\dy}) {\ldots};
    \node at ({\xzero+3*\dx},{\yzero+1.75*\dy}) {\ldots};
    \node at ({\xzero+9*\dx},{\yzero+0.25*\dy}) {\ldots};
    \node at ({\xzero+9*\dx},{\yzero+1.75*\dy}) {\ldots};
    
  \end{tikzpicture}
\end{center}
So, both partitions generate the category of $U_{n,k}$. Because
  
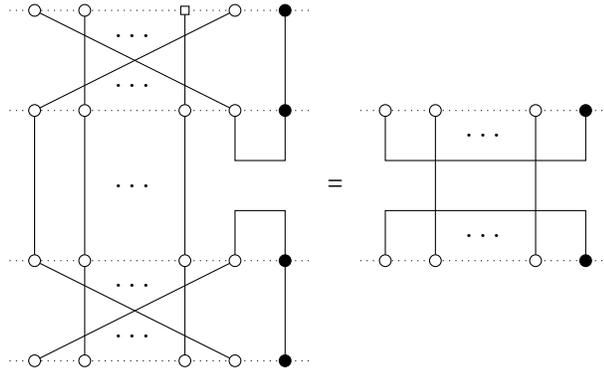
\begin{center}
    \begin{tikzpicture}[scale=0.666]
      \draw [dotted] (-0.5,0) -- (5.5,0);
      \draw [dotted] (-0.5,2) -- (5.5,2);
      \draw [dotted] (-0.5,5) -- (5.5,5);
      \draw [dotted] (-0.5,7) -- (5.5,7);
      \node[scale=0.4, circle, draw=black,fill=white] (z1) at (0,7) {};
      \node[scale=0.4, circle, draw=black,fill=white] (z2) at (1,7) {};      
      \node[scale=0.4, circle, draw=black,fill=white] (z3) at (4,7) {};
      \node[scale=0.4, circle, draw=black,fill=black] (z4) at (5,7) {};      
      \node[scale=0.4, circle, draw=black,fill=white] (y1) at (0,5) {};
      \node[scale=0.4, circle, draw=black,fill=white] (y2) at (1,5) {};      
      \node[scale=0.4, circle, draw=black,fill=white] (y3) at (4,5) {};
      \node[scale=0.4, circle, draw=black,fill=black] (y4) at (5,5) {};
      \node[scale=0.4, circle, draw=black,fill=white] (x1) at (0,2) {};
      \node[scale=0.4, circle, draw=black,fill=white] (x2) at (1,2) {};      
      \node[scale=0.4, circle, draw=black,fill=white] (x3) at (4,2) {};
      \node[scale=0.4, circle, draw=black,fill=black] (x4) at (5,2) {};
      \node[scale=0.4, circle, draw=black,fill=white] (w1) at (0,0) {};
      \node[scale=0.4, circle, draw=black,fill=white] (w2) at (1,0) {};      
      \node[scale=0.4, circle, draw=black,fill=white] (w3) at (4,0) {};
      \node[scale=0.4, circle, draw=black,fill=black] (w4) at (5,0) {};
      \node[scale=0.4,draw=black,fill=white] (s1) at (3,7) {};
      \node[scale=0.4, circle, draw=black,fill=white] (s2) at (3,5) {};
      \node[scale=0.4, circle, draw=black,fill=white] (s3) at (3,2) {};
      \node[scale=0.4, circle, draw=black,fill=white] (s4) at (3,0) {};        
      \draw (w1) to (x3);
      \draw (w3) to (x1);
      \draw (w4) to (x4);
      \draw (x1) to (y1);
      \draw (x3) -- ++(0,1) -| (x4);
      \draw (y3) -- ++(0,-1) -| (y4);
      \draw (y1) to (z3);
      \draw (y3) to (z1);
      \draw (y4) to (z4);
      \draw (w2) to (x2);
      \draw (x2) to (y2);
      \draw (y2) to (z2);
      \draw (s1) to (s2) to (s3) to (s4);
      \node at (2,0.5) {\ldots};
      \node at (2,1.5) {\ldots};
      \node at (2,3.5) {\ldots};
      \node at (2,5.5) {\ldots};
      \node at (2,6.5) {\ldots};            
      \node at (6,3.5) {$=$};
      \draw [dotted] (6.5,5) -- (11.5,5);
      \draw [dotted] (6.5,2) -- (11.5,2);
      \node[scale=0.4, circle, draw=black,fill=white] (a1) at (7,2) {};
      \node[scale=0.4, circle, draw=black,fill=white] (a2) at (8,2) {};
      \node[scale=0.4, circle, draw=black,fill=white] (a3) at (10,2) {};      
      \node[scale=0.4, circle, draw=black,fill=black] (a4) at (11,2) {};      
      \node[scale=0.4, circle, draw=black,fill=white] (b1) at (7,5) {};
      \node[scale=0.4, circle, draw=black,fill=white] (b2) at (8,5) {};
      \node[scale=0.4, circle, draw=black,fill=white] (b3) at (10,5) {};
      \node[scale=0.4, circle, draw=black,fill=black] (b4) at (11,5) {};            
      \draw (a1) --++(0,1) -| (a4);
      \draw (b1) --++(0,-1) -| (b4);      
      \draw (a2) to (b2);
      \draw (a3) to (b3);
      \node at (9,2.5) {\ldots};
      \node at (9,4.5) {\ldots};      
    \end{tikzpicture}    
  \end{center}
and
    
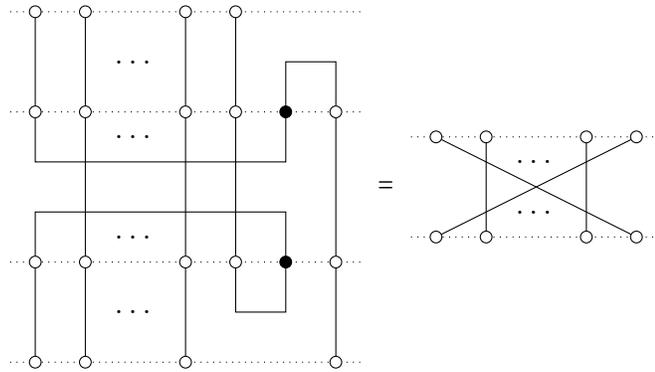
\begin{center}
    \begin{tikzpicture}[scale=0.666]
      \draw [dotted] (-0.5,0) -- (6.5,0);
      \draw [dotted] (-0.5,2) -- (6.5,2);
      \draw [dotted] (-0.5,5) -- (6.5,5);
      \draw [dotted] (-0.5,7) -- (6.5,7);
      \node[scale=0.4, circle, draw=black,fill=white] (z1) at (0,7) {};
      \node[scale=0.4, circle, draw=black,fill=white] (z2) at (1,7) {};
      \node[scale=0.4, circle, draw=black,fill=white] (z3) at (4,7) {};
      \node[scale=0.4, circle, draw=black,fill=white] (s1) at (3,7) {};
      \node[scale=0.4, circle, draw=black,fill=white] (s2) at (3,5) {};
      \node[scale=0.4, circle, draw=black,fill=white] (s3) at (3,2) {};
      \node[scale=0.4, circle, draw=black,fill=white] (s4) at (3,0) {};                                    
      \node[scale=0.4, circle, draw=black,fill=white] (y1) at (0,5) {};
      \node[scale=0.4, circle, draw=black,fill=white] (y2) at (1,5) {};
      \node[scale=0.4, circle, draw=black,fill=white] (y3) at (4,5) {};
      \node[scale=0.4, circle, draw=black,fill=black] (y4) at (5,5) {};
      \node[scale=0.4, circle, draw=black,fill=white] (y5) at (6,5) {};      
      \node[scale=0.4, circle, draw=black,fill=white] (x1) at (0,2) {};
      \node[scale=0.4, circle, draw=black,fill=white] (x2) at (1,2) {};
      \node[scale=0.4, circle, draw=black,fill=white] (x3) at (4,2) {};
      \node[scale=0.4, circle, draw=black,fill=black] (x4) at (5,2) {};
      \node[scale=0.4, circle, draw=black,fill=white] (x5) at (6,2) {};      
      \node[scale=0.4, circle, draw=black,fill=white] (w1) at (0,0) {};
      \node[scale=0.4, circle, draw=black,fill=white] (w2) at (1,0) {};      
      \node[scale=0.4, circle, draw=black,fill=white] (w3) at (6,0) {};
      \draw (w1) to (x1);
      \draw (w2) to (x2);
      \draw (w3) to (x5);      
      \draw (x3) --++(0,-1)-| (x4);
      \draw (x1) --++(0,1)-| (x4);
      \draw (x2) to (y2);
      \draw (x3) to (y3);      
      \draw (x5) to (y5);
      \draw (y1) to (z1);
      \draw (y2) to (z2);
      \draw (y3) to (z3);
      \draw (y4) --++(0,1)-| (y5);
      \draw (y1) --++(0,-1)-| (y4);
      \draw (s1) to (s2) to (s3) to (s4);
      \node at (7,3.5) {$=$};
      \node at (2,1) {\ldots};
      \node at (2,2.5) {\ldots};
      \node at (2,4.5) {\ldots};
      \node at (2,6) {\ldots};      
      \draw [dotted] (7.5,2.5) -- (12.5,2.5);
      \draw [dotted] (7.5,4.5) -- (12.5,4.5);
      \node[scale=0.4, circle, draw=black,fill=white] (a1) at (8,2.5) {};
      \node[scale=0.4, circle, draw=black,fill=white] (a2) at (9,2.5) {};
      \node[scale=0.4, circle, draw=black,fill=white] (a3) at (12,2.5) {};      
      \node[scale=0.4, circle, draw=black,fill=white] (a4) at (8,4.5) {};
      \node[scale=0.4, circle, draw=black,fill=white] (a5) at (9,4.5) {};
      \node[scale=0.4, circle, draw=black,fill=white] (a6) at (12,4.5) {};
      \node[scale=0.4, circle, draw=black,fill=white] (a8) at (11,4.5) {};
      \node[scale=0.4, circle, draw=black,fill=white] (a7) at (11,2.5) {};      
      \draw (a1) to (a6);
      \draw (a2) to (a5);
      \draw (a3) to (a4);
      \draw (a7) to (a8);
      \node at (10,3) {\ldots};
      \node at (10,4) {\ldots};      
    \end{tikzpicture},
  \end{center}
the category of $U_{n,k}$ is precisely our category $\mc S_k$ by Proposition~\hyperref[item:generators_of_the_categories_S_w-item_2]{\ref*{proposition:generators_of_the_categories_S_w}~\ref*{item:generators_of_the_categories_S_w-item_2}}.
              \item Lastly, the quantum group $U_{n,\infty}$ is defined by Banica and Bichon in \cite[Definition~8.1]{BaBi17b} as some limit case of their series $\left( U_{n,k} \right)_{k\in\mathbb{N}}$. The two partitions
		
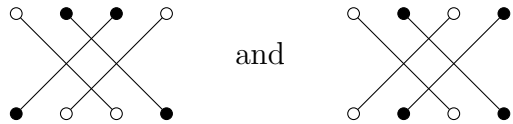
\begin{gather*}
		\begin{tikzpicture}[scale=0.666,baseline=(current  bounding  box.center)]
		\node [circle, scale=0.4, draw=black, fill=black] (x1-1) at (0,0) {};
		\node [circle, scale=0.4, draw=black, fill=white] (x2-1) at (1,0) {};
		\node [circle, scale=0.4, draw=black, fill=white] (x3-1) at (2,0) {};
		\node [circle, scale=0.4, draw=black, fill=black] (x4-1) at (3,0) {};
		\node [circle, scale=0.4, draw=black, fill=white] (x1-2) at (0,2) {};
		\node [circle, scale=0.4, draw=black, fill=black] (x2-2) at (1,2) {};
		\node [circle, scale=0.4, draw=black, fill=black] (x3-2) at (2,2) {};
		\node [circle, scale=0.4, draw=black, fill=white] (x4-2) at (3,2) {};
		\draw (x1-1) -- (x3-2);
		\draw (x2-1) -- (x4-2);
		\draw (x3-1) -- (x1-2);
		\draw (x4-1) -- (x2-2);									
		\end{tikzpicture}
		\quad\quad\text{and}\quad\quad
		\begin{tikzpicture}[scale=0.666,baseline=(current  bounding  box.center)]
		\node [circle, scale=0.4, draw=black, fill=white] (x1-1) at (0,0) {};
		\node [circle, scale=0.4, draw=black, fill=black] (x2-1) at (1,0) {};
		\node [circle, scale=0.4, draw=black, fill=white] (x3-1) at (2,0) {};
		\node [circle, scale=0.4, draw=black, fill=black] (x4-1) at (3,0) {};
		\node [circle, scale=0.4, draw=black, fill=white] (x1-2) at (0,2) {};
		\node [circle, scale=0.4, draw=black, fill=black] (x2-2) at (1,2) {};
		\node [circle, scale=0.4, draw=black, fill=white] (x3-2) at (2,2) {};
		\node [circle, scale=0.4, draw=black, fill=black] (x4-2) at (3,2) {};
		\draw (x1-1) -- (x3-2);
		\draw (x2-1) -- (x4-2);
		\draw (x3-1) -- (x1-2);
		\draw (x4-1) -- (x2-2);									
		\end{tikzpicture}
		\end{gather*}
		generate its associated category.                 We can write these generators equivalently in the form
                		
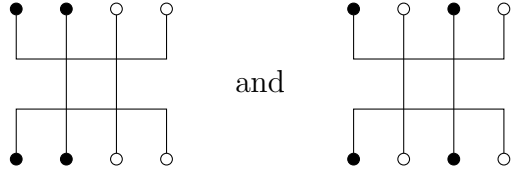
\begin{gather*}
		\begin{tikzpicture}[scale=0.666,baseline=0.888cm]
		\node [circle, scale=0.4, draw=black, fill=black] (x1-1) at (0,0) {};
		\node [circle, scale=0.4, draw=black, fill=black] (x2-1) at (1,0) {};
		\node [circle, scale=0.4, draw=black, fill=white] (x3-1) at (2,0) {};
		\node [circle, scale=0.4, draw=black, fill=white] (x4-1) at (3,0) {};
		\node [circle, scale=0.4, draw=black, fill=black] (x1-2) at (0,3) {};
		\node [circle, scale=0.4, draw=black, fill=black] (x2-2) at (1,3) {};
		\node [circle, scale=0.4, draw=black, fill=white] (x3-2) at (2,3) {};
		\node [circle, scale=0.4, draw=black, fill=white] (x4-2) at (3,3) {};
		\draw (x1-1) -- ++ (0,1) -| (x4-1);
		\draw (x1-2) -- ++ (0,-1) -| (x4-2);                
		\draw (x2-1) -- (x2-2);
		\draw (x3-1) -- (x3-2);
		\end{tikzpicture}
		\quad\quad\text{and}\quad\quad
		\begin{tikzpicture}[scale=0.666,baseline=0.888cm]
		\node [circle, scale=0.4, draw=black, fill=black] (x1-1) at (0,0) {};
		\node [circle, scale=0.4, draw=black, fill=white] (x2-1) at (1,0) {};
		\node [circle, scale=0.4, draw=black, fill=black] (x3-1) at (2,0) {};
		\node [circle, scale=0.4, draw=black, fill=white] (x4-1) at (3,0) {};
		\node [circle, scale=0.4, draw=black, fill=black] (x1-2) at (0,3) {};
		\node [circle, scale=0.4, draw=black, fill=white] (x2-2) at (1,3) {};
		\node [circle, scale=0.4, draw=black, fill=black] (x3-2) at (2,3) {};
		\node [circle, scale=0.4, draw=black, fill=white] (x4-2) at (3,3) {};
		\draw (x1-1) -- ++ (0,1) -| (x4-1);
		\draw (x1-2) -- ++ (0,-1) -| (x4-2);                
		\draw (x2-1) -- (x2-2);
		\draw (x3-1) -- (x3-2);
		\end{tikzpicture}
              \end{gather*}
              by rotating each generator cyclically in counter-clockwise direction. And Lemma~\hyperref[item:bracket_operations-item_6]{\ref*{lemma:bracket_operations}~\ref*{item:bracket_operations-item_6}} showed that the second partition of the last pair and $\PartHalfLibWBW$ generate the same category. So, $U_{n,\infty}$ corresponds to the category $\langle\PartBracketBBWW,\PartHalfLibWBW\rangle$. As $U_n^\times$ and $U_{n,\infty}$ are distinct, an immediate consequence is $\langle\PartHalfLibWBW\rangle\subsetneq \langle\PartBracketBBWW,\PartHalfLibWBW\rangle$. Again, from Proposition~\hyperref[item:generators_of_the_categories_S_w-item_2]{\ref*{proposition:generators_of_the_categories_S_w}~\ref*{item:generators_of_the_categories_S_w-item_2}}, we know $\langle\PartHalfLibWBW\rangle$ and $\langle\PartBracketBBWW,\PartHalfLibWBW\rangle$ to be subcategories of $\mc S_0$, both to be treated in \cite{MaWe18b}. Especially, we will show that $\langle \PartBracketBBWW,\PartHalfLibWBW\rangle$ and $\mc S_0$ are indeed distinct.
\end{enumerate}

\subsection{Outlook on \texorpdfstring{\cite{MaWe18b}}{Part II}}

In the present article we established combinatorially that every category $\mc C\subseteq \langle\PartCrossWW\rangle$ is either one of the categories $(\mc S_w)_{w\in \nnint}$ or a proper subcategory of $\mc S_0$ (Proposition~\ref{proposition:classification_of_neutral_pair_partitions_above_S_0}). In particular, $\mc S_0$ is revealed to be a halfway point for all categories $\mc C$ with $\mc C\subseteq \langle \PartCrossWW\rangle$ as every such category necessarily satisfies either $\mc C\subseteq \mc S_0$ or $\mc S_0\subseteq \mc C$. The follow-up article \cite{MaWe18b} will continue the combinatorial analysis of categories in the realm $\mc S_0$. We will determine all subcategories of $\mc S_0$, their description in combinatorial terms and their generating partitions. Many more categories will be shown to exist in $\mc S_0$ besides the known examples. While the categories $\mc C$ with $\mc S_0\subseteq \mc C\subseteq \langle \PartCrossWW\rangle$ are equivalent to the cyclic groups (see also Section~\ref{section:definition_set_relationships}), we will prove that the subcategories of $\mc S_0$ are dual to the \emph{numerical semigroups}, i.e.\ subsemigroups of $(\nnint, +)$. Specifically, we can contrast the present article and \cite{MaWe18b} as follows: \par
\vspace{0.5em}
\begin{center}
  \begin{tabular}[h]{c p{0.37\textwidth} p{0.37\textwidth}}
    \multicolumn{1}{c}{Main Theorem} & \multicolumn{1}{c}{Present Article} & \multicolumn{1}{c}{\cite{MaWe18b}} \\
  \hline\\[-0.75em]
    1~(a)&Subsets $(\mc S_w)_{w\in \nnint}$ of $\langle\PartCrossWW\rangle$ are categories, where $\nnint$ corresponds to the set of \emph{cyclic groups}.&Subsets $(\mc I_N)_{N\in \mc N}$ of $\langle\PartCrossWW\rangle$ are categories, where $\mc N$ denotes the set of \emph{numerical semigroups}.\\[0.25em]
    1~(b)&$w'\integers\subseteq w\integers \implies \mc S_{w'}\subseteq \mc S_w$.&$N'\subseteq N\implies \mc I_{N'}\supseteq\mc I_N$.\\[0.25em]
    1~(c)&Generators of $\mc S_w$&Generators of $\mc I_N$\\[0.25em]    
  2&Classification of all categories $\mc C$ with $\mc S_0\subseteq\mc C\subseteq \langle\PartCrossWW\rangle$&Classification of all categories $\mc C$ with $\mc S_0\not\subseteq \mc C\subseteq\langle\PartCrossWW\rangle$\\
\end{tabular}
\end{center}
\vspace{0.5em}
In \cite{MaWe18b} we will continue to employ the techniques of \emph{bracket} partitions developed in Section~\ref{section:brackets}. It was shown in Proposition~\ref{proposition:generation_of_categories_by_residual_brackets} that every category $\mc C\subseteq \langle\PartCrossWW\rangle$ is generated by its set of \emph{residual brackets} (Definition~\ref{definition:residual_bracket}). These come in two flavors: None of the \emph{residual brackets of the first kind}, $\PartBracketBBwW$, $\PartBracketBWwW$, $\PartBracketWBwB$ and $\PartBracketWWwB$ for arbitrary $w\in \pint$, are elements of $\mc S_0$. Rather, for fixed $w$, they all generate the category $\mc S_w$ (Proposition~\hyperref[item:generators_of_the_categories_S_w-item_2]{\ref*{proposition:generators_of_the_categories_S_w}~\ref*{item:generators_of_the_categories_S_w-item_2}}). The \emph{residual brackets of the second kind} collectively generate $\mc S_0$ (Proposition~\hyperref[item:generators_of_the_categories_S_w-item_1]{\ref*{proposition:generators_of_the_categories_S_w}~\ref*{item:generators_of_the_categories_S_w-item_1}}). So far, we have only explicitly encountered the few examples: $\PartBracketBBvWvW$ and  $\PartBracketWWvBvB$ for arbitrary $v\in \pint$ as well as $\PartBracketBWBW$ and  $\PartBracketWBWB$. In \cite{MaWe18b} we consider arbitrary sets of residual brackets of the second kind and determine which categories they generate.

\printbibliography

\end{document}